\documentclass[11pt]{amsart}
\usepackage{amssymb,amsmath, amsthm, amsfonts}

\usepackage{hyperref}
\usepackage{cite}
\hypersetup{linktocpage,colorlinks, linkcolor = {blue}, citecolor={magenta}}

\usepackage[normalem]{ulem}

\usepackage[usenames,dvipsnames]{xcolor}
\usepackage{graphicx}
\usepackage{listings}
\usepackage[margin=1in]{geometry}
\usepackage{lstautogobble}
\usepackage{enumerate}
\usepackage[shortlabels]{enumitem}
\usepackage{thmtools}
\usepackage{thm-restate}
\usepackage{amsthm}
\usepackage{verbatim}
\usepackage{accents}
\usepackage{mathtools}
\usepackage{physics}

\usepackage{hyperref}
\usepackage[capitalise]{cleveref}
\usepackage[bottom]{footmisc}
\crefformat{equation}{(#2#1#3)}

\usepackage{float}
\restylefloat{table}

\usepackage{mathrsfs}
\setlist{  
  listparindent=\parindent,
  parsep=0pt,
}

\theoremstyle{plain}
\newtheorem{thm}{Theorem}[section]
\newtheorem{prop}[thm]{Proposition}
\newtheorem{lemma}[thm]{Lemma}
\newtheorem{cor}[thm]{Corollary}

\theoremstyle{definition}
\newtheorem{mydef}[thm]{Definition}

\newtheorem{remark}[thm]{Remark}

\Crefname{thm}{Theorem}{Theorems}
\Crefname{prop}{Proposition}{Propositions}

\numberwithin{equation}{section} 

\DeclarePairedDelimiter\ipp{\langle}{\rangle}

\DeclarePairedDelimiter{\pa}{\lparen}{\rparen}

\DeclarePairedDelimiter{\jp}{\langle}{\rangle}

\DeclareMathOperator{\supp}{supp}

\newcommand{\M}{{\mathcal{M}}}

\newcommand{\p}{{\partial}}

\renewcommand{\d}{\mathsf{d}}

\newcommand{\R}{{\mathbb{R}}}

\newcommand{\N}{{\mathbb{N}}}

\newcommand{\Z}{{\mathbb{Z}}}

\newcommand{\Ss}{{\mathbb{S}}}

\newcommand{\T}{{\mathbb{T}}}
\newcommand{\g}{{\mathsf{g}}}

\newcommand{\Sc}{{\mathcal{S}}}

\renewcommand{\M}{{\mathbb{M}}}
\newcommand{\I}{\mathbb{I}}

\renewcommand{\k}{\mathsf{k}}
\newcommand{\ga}{\gamma}
\newcommand{\nab}{\nabla}

\newcommand{\tl}{\tilde}

\newcommand{\ph}{\phantom{=}}
\newcommand{\nn}{\nonumber}

\newcommand{\ux}{X}

\newcommand{\ep}{\epsilon}
\newcommand{\vep}{\varepsilon}
\newcommand{\al}{\alpha}
\newcommand{\be}{\beta}
\newcommand{\ka}{\kappa}
\newcommand{\la}{\lambda}

\newcommand{\indic}{\mathbf{1}}
\newcommand{\f}{\mathsf{f}}

\newcommand{\W}{{\mathbf{W}}}
\newcommand{\E}{{\mathbb{E}}}

\newcommand{\wh}{\widehat}

\newcommand{\Dm}{|\nabla|}

\newcommand{\Rs}{\mathrm{Riesz}}

\newcommand{\as}{\mathsf{a}}
\newcommand{\rs}{\mathsf{r}}
\newcommand{\cd}{\mathsf{c}_{\mathsf{d},\mathsf{s}}}

\newcommand{\s}{\mathsf{s}}
\renewcommand{\k}{\mathsf{k}}

\newcommand{\zg}{|z|^\ga}

\newcommand{\Fr}{\mathsf{F}}
\renewcommand{\P}{\mathcal{P}}

\let\div\relax
\DeclareMathOperator{\div}{\mathrm{div}}

\def\XXint#1#2#3{{\setbox0=\hbox{$#1{#2#3}{\int}$ }
\vcenter{\hbox{$#2#3$ }}\kern-.6\wd0}}

\let\oldtocsection=\tocsection
 
\let\oldtocsubsection=\tocsubsection
 
\let\oldtocsubsubsection=\tocsubsubsection
 
\renewcommand{\tocsection}[2]{\hspace{0em}\oldtocsection{#1}{#2}}
\renewcommand{\tocsubsection}[2]{\hspace{1em}\oldtocsubsection{#1}{#2}}
\renewcommand{\tocsubsubsection}[2]{\hspace{2em}\oldtocsubsubsection{#1}{#2}}

\title[A sharp commutator estimate for all Riesz modulated energies]{A sharp commutator estimate for all Riesz modulated energies}

\author[E. Hess-Childs]{Elias Hess-Childs}
\address{Elias Hess-Childs, Carnegie Mellon University, Department of Mathematical Sciences, Pittsburgh, PA}
\email{ehesschi@andrew.cmu.edu}
\thanks{E.H.C was supported by the NSF grant DMS-2342349 and by NSF grant DMS-2424139, while the author was in residence at the Simons Laufer Mathematical Sciences Institute in Berkeley, California, during the Fall 2025 semester.}
\author[M. Rosenzweig]{Matthew Rosenzweig}
\address{Matthew Rosenzweig, Carnegie Mellon University, Department of Mathematical Sciences, Pittsburgh, PA} 
\email{mrosenz2@andrew.cmu.edu}
\thanks{M.R. was supported by NSF grants DMS-2441170, DMS-2345533, DMS-2342349.}
\author[S. Serfaty]{Sylvia Serfaty}
\address{Sylvia Serfaty,  Sorbonne Universite, CNRS, Universit\'e de Paris,
 Laboratoire Jacques-Louis Lions (LJLL), F-75005 Paris \\ \& Institut Universitaire de France \\ \&
 Courant Institute of Mathematical Sciences, New York University}
\email{serfaty@cims.nyu.edu}
\thanks{S.S. was supported by NSF grant DMS-2247846 and by the Simons Foundation through the Simons Investigator program.}

\begin{document}
\begin{abstract}
We prove a functional inequality in any dimension controlling the derivative along a transport of the Riesz modulated energy in terms of the modulated energy itself. This modulated energy was introduced by the third author and collaborators in the study of mean-field limits and statistical mechanics of Coulomb/Riesz gases, where this control is an essential ingredient. Previous work of the last two authors and Q.H. Nguyen \cite{NRS2021} showed a similar functional inequality but with an additive $N$-dependent error (where $N$ is the number of particles, $\d$ the dimension, and $\s$ the inverse power of the Riesz potential) which was not sharp. In this paper, we obtain the optimal $N^{\frac{\s}{\d}-1}$  error, for all cases, including the sub-Coulomb case.  Our method is  conceptually simple and, like previous work, relies on the observation that the derivative along a transport of the modulated energy is the quadratic form of a commutator. Through a new potential truncation scheme based on a wavelet-type representation of the Riesz potential to handle its singularity, the proof reduces to averaging over a family of Kato-Ponce type estimates. 

The  commutator estimate has applications to sharp rates of convergence for mean-field limits, quasi-neutral limits, and central limit theorems for the fluctuations of Coulomb/Riesz gases both at and out of thermal equilibrium. In particular, we show here for $\s<\d-2$ the expected  $N^{\frac{\s}{\d}-1}$-rate in the modulated energy distance for the  mean-field convergence of first-order Hamiltonian and gradient flows. This complements the recent work \cite{RS2022} on the optimal rate for the (super-)Coulomb case $\d-2\le \s<\d$  and therefore  resolves the entire potential Riesz case.  

\end{abstract}
\maketitle

\section{Introduction}\label{sec:intro}
\subsection{Motivation}\label{ssec:introMot}
In any dimension $\d$,  consider the class of  \emph{Riesz} interactions 
\begin{equation}\label{eq:gmod}
\g(x)= \begin{cases}  \frac{1}{\s} |x|^{-\s}, \quad  & \s \neq 0\\
-\log |x|, \quad & \s=0,
\end{cases}
\end{equation}
with the assumption that  $0\le \s<\d$. Up to a normalizing constant $\cd$, these interactions are characterized as fundamental solutions of the fractional Laplacian: $(-\Delta)^{\frac{\d-\s}{2}}\g = \cd \delta_0$. The particular value $\s=\d-2$  corresponds to the classical \emph{Coulomb} interaction from physics. We are primarily interested in the \emph{sub-Coulomb} case  $\s < \d-2$, though our main result is also valid for the Coulomb/super-Coulomb case $\s\ge \d-2$. The restriction to $\s\ge 0$ is to only consider the more challenging singular case (see \cref{rem:nonsing}), while the restriction to the potential regime $\s<\d$ is to exclude the \emph{hypersingular} case, which is not of the mean-field type considered in this paper (e.g. see \cite{HSST2020}).

Given a system of $N$ distinct points $\ux_N = (x_1, \dots, x_N)\in (\R^\d)^N$  with interaction energy  
\begin{equation}
\sum_{1\le i\neq j\le N} \g(x_i-x_j),
\end{equation}
one is led to comparing the sequence of \emph{empirical measures} $\mu_N \coloneqq \frac1N \sum_{i=1}^N \delta_{x_i}$ to a \emph{mean-field} density $\mu$. This comparison is conveniently performed by considering a {\it modulated energy}, or Riesz (squared) ``distance" between $\mu_N$ and $\mu$, defined by 
\begin{equation}\label{eq:modenergy}
\Fr_N(\ux_N, \mu) \coloneqq \frac12\int_{(\R^\d)^2 \setminus \triangle} \g(x-y) d\Big(\frac{1}{N} \sum_{i=1}^N \delta_{x_i} - \mu\Big)(x) d\Big(\frac{1}{N} \sum_{i=1}^N \delta_{x_i} - \mu\Big)(y),
\end{equation}
where we excise the diagonal $\triangle\coloneqq \{(x,y)\in(\R^\d)^2:x=y\}$ from the domain of integration in order to remove the infinite self-interaction of each particle.  This object originated in the study of the statistical mechanics of Coulomb/Riesz gases \cite{SS2015,SS2015log,RS2016, PS2017} and later was used in the derivation of mean-field dynamics \cite{Duerinckx2016,Serfaty2020,NRS2021} and following works. We refer to  \cite[Chapter 4]{SerfatyLN} for a comprehensive discussion of the modulated energy.
 
An essential point is to control derivatives of $\Fr_N$ along  a transport $v:\R^\d\rightarrow\R^\d$, i.e. the quantities
\begin{equation}\label{15}
 \frac{d^n}{dt^n}\Big|_{t=0} \Fr_N( (\I + tv)^{\oplus N} (\ux_N), (\I + tv)\# \mu)= {\frac12}\int_{(\R^\d)^2\setminus \triangle} 
\nabla^{\otimes n} \g(x-y):  (v(x)-v(y))^{\otimes n}  d ( \mu_N- \mu)^{\otimes 2} (x,y),
\end{equation}
where $\I:\R^\d\rightarrow\R^\d$ is the identity, $(\I+t v)^{\oplus N} (\ux_N) \coloneqq (x_1 + tv(x_1), \ldots, x_N+ tv(x_N))$, and $:$ denotes the inner product between the tensors. 
For applications to the mean-field limit, $v$ is the velocity field of the limiting evolution as $N \to \infty$. For applications to central limit theorems (CLTs) for the fluctuations (the next-order description), $v$ is the gradient of a test function evolved along the adjoint linearized mean-field flow \cite{HRSclt, CN2025}, and these inequalities are at the core of the ``transport'' approach to fluctuations of canonical Gibbs ensembles \cite{LS2018,BLS2018,Serfaty2023, PS2024}. We also mention that the quantity in \eqref{15} is the same as appears in the loop (Dyson--Schwinger) equations in the random matrix theory literature, e.g. \cite{BG2013, BG2024, BBNY2019}, which are another avatar of the transport method.
 

Our main objective in this paper is to establish a functional inequality asserting that the quantity in \eqref{15} is always bounded by \begin{equation*}
C_1(\Fr_N(\ux_N, \mu) + C_2N^{-\alpha}), \quad \alpha >0, \end{equation*}  where $C_1>0$ is a constant depending on $\d,\s,v$ and $C_2>0$ depends on $\d,\s,\mu$. For $n=1$, this was first proved  in \cite{LS2018} in the 2D Coulomb case, then generalized to the super-Coulomb case in \cite{Serfaty2020}. The proof relied on identifying  a {\it stress-energy tensor} in \eqref{15} and using integration by parts. Subsequently, the second author \cite{Rosenzweig2020spv} observed that the expressions \eqref{15}  may also be viewed as the quadratic form of a \textit{commutator}, akin to the famous Calder\'{o}n commutator \cite{Calderon1980}. This perspective was then used by the last two authors and Q.H. Nguyen \cite{NRS2021} to show the desired functional inequalities for all values of $\s$ and any order\footnote{Strictly speaking, the paper only explicitly considers the case $n=2$, but the argument extends (with more algebra) to any $n\ge 2$.} $n\ge 2$ as well as a broader class of $\g$'s that may be regarded as perturbations of Riesz interactions, including Lennard-Jones type potentials. Although these stress-energy and commutator perspectives may seem distinct, they are in fact dual to each other \cite{RS2022}. Not only are these functional inequalities crucial for proving CLTs for the fluctuations of Riesz gases \cite{LS2018, Serfaty2023, PS2024}, and even more so for deriving mean-field limits \cite{Serfaty2020,Rosenzweig2022,Rosenzweig2022a,NRS2021,CdCRS2023,RS2022, bPCJ2025} and large deviation principles \cite{hC2023}. They have been further used to show joint classical and mean-field limits \cite{GP2021} and supercritical mean-field limits of classical \cite{HkI2021, Rosenzweig2021ne, Menard2022, RSlake} and quantum systems of particles \cite{Rosenzweig2021qe, Porat2022}.

Once such inequalities are established, a natural next problem is to determine the optimal size of the additive error, i.e.~the exponent $\alpha$, which a priori depends on $\d$ and $\s$. The appropriate comparison is with the minimal size of $\Fr_N(\ux_N,\mu)$. In the (super-)Coulomb case, it is known that $\Fr(\ux_N,\mu) + \frac{\log(N\|\mu\|_{L^\infty})}{2\d N}\indic_{\s=0}\ge -C\|\mu\|_{L^\infty}^{\frac\s\d}N^{\frac{\s}{\d}-1}$ and this lower bound is sharp.\footnote{The term $\frac{\log(N\|\mu\|_{L^\infty})}{2\d N}$ in the $\s=0$ case should \emph{not} be viewed as an additive error like $\|\mu\|_{L^\infty}^{\frac\s\d}N^{\frac{\s}{\d}-1}$, but rather a consequence of the fact that the right quantity to consider is $\Fr(\ux_N,\mu) + \frac{\log(N\|\mu\|_{L^\infty})}{2\d N}\indic_{\s=0}$.} For instance, see \cite[Sections 4.2, 12.4]{SerfatyLN}. We show in this work that the same lower bound holds in the sub-Coulomb case as well (see \cref{rem:MElb}), which was not previously known. Moreover, this lower bound is in general sharp \cite{HSSS2017}. In controlling the magnitude of \eqref{15}, one needs a right-hand side which is also nonnegative. Thus, the best error one may hope for is of size $N^{\frac{\s}{\d}-1}$. 

To date, this sharp $N^{\frac{\s}{\d}-1}$ error rate has only been proven in the (super-)Coulomb case: first, the Coulomb case \cite{LS2018,Serfaty2023, Rosenzweig2021ne} up to second order (i.e. $n\le 2$ in \eqref{15}), and then recently, the entire (super-)Coulomb case at any order \cite{RS2022}. In the present work, we prove a functional inequality with the sharp error for the sub-Coulomb case at first order. In fact, we prove a sharp functional inequality for all Riesz cases $0\le \s <\d$. See the remarks at the end of this subsection for a comparison with the (super-)Coulomb inequality of \cite{RS2022}.

In addition to its immediate application to the mean-field limit (see \cref{ssec:introapp}), this sharp functional inequality is crucial for studying fluctuations around the mean-field limit in and out of thermal equilibrium. We use this functional inequality in the forthcoming work \cite{HRSclt} by J. Huang and the last two authors on quantitative CLTs for the fluctuations of Riesz flows in varying temperature regimes. This inequality also may be incorporated into the transport and Stein's method approaches to CLTs for fluctuations of canonical Gibbs ensembles to establish new results \cite{RSstein}.

\subsection{New functional inequality}
To state our main results, we define the microscopic length scale 
\begin{equation}\label{deflambda}
\lambda \coloneqq (N\|\mu\|_{L^\infty})^{-\frac{1}{\d}},
\end{equation}
which one may regard as the typical inter-particle distance. At the cost of  letting all constants depend on $\|\mu\|_{L^\infty}$,  one may simply take $\lambda = N^{-1/\d}$ throughout the paper, if one prefers not to track the dependence of $\|\mu\|_{L^\infty}$.


\begin{thm}\label{thm:FI}
Let $\d\ge 1$, $0\le \s<\d$, and $\as \in (\d,\d+2)$. There exists a constant $C>0$ depending only on $\as,\d,\s$ such that the following holds. Let $\mu \in L^1(\R^\d)\cap L^\infty(\R^\d)$ with $\int_{\R^\d}\mu=1$.  Let  $v:\R^\d\rightarrow\R^\d$ be  a  Lipschitz vector field. For any pairwise distinct configuration $\ux_N \in (\R^\d)^N$, it holds that
\begin{multline}\label{main1}
\Bigg|\int_{(\R^\d)^2\setminus\triangle} (v(x)-v(y))\cdot\nabla\g(x-y) d\Big(\frac1N\sum_{i=1}^N\delta_{x_i}-\mu\Big)^{\otimes 2}(x,y)\bigg| \\
\leq C\Big(\|\nabla v\|+\||\nabla|^{\frac{\as}{2}}v\|_{L^{\frac{2\d}{\as-2}}}\indic_{\as>2}\Big)\Big(\Fr_N(\ux_N,\mu) - \frac{\log\la}{2N}\indic_{\s=0} +C\|\mu\|_{L^\infty}\la^{\d-\s}\Big),
\end{multline}
where $\Dm \coloneqq (-\Delta)^{1/2}$.
\end{thm}

We note that the additive error term  $C\|\mu\|_{L^\infty}\lambda^{\d-\s} \propto N^{-1+\frac{\s}{\d}}$,  which is the announced optimal estimate. The dependence $\|\mu\|_{L^\infty}$ may be weakened to $\|\mu\|_{L^p}$ for $p>\frac{\d}{\d-\s}$, but at the cost of increasing the size of the additive error term. We defer comments on the regularity assumptions for $v$ until \cref{ssec:introRQ}. 

\begin{remark}
Here, optimality is understood in the ``worst case'' sense, i.e. that there exists a point configuration $\ux_N$ that achieves the bound. This should be contrasted with the ``average'' sense, where $\ux_N\sim f_N \in \mathcal{P}((\R^\d)^N)$ and one estimates $\E_{\ux_N\sim f_N}[|\eqref{15}|]$. From the perspective of statistical mechanics, the worst case corresponds to zero temperature, while the average case encompasses a range of temperatures, with the infinite-temperature limit corresponding to iid points. It is well-known that averaging may produce smaller errors. One quickly checks that if $f_N \sim \mu^{\otimes N}$ (i.e. $\mu$-iid), then $\E_{\ux_N\sim f_N}[|\eqref{15}|]\approx N^{-1}$. Moreover, it is known that for bounded interaction potentials, there are \emph{high-temperature/entropic} commutator estimates in which $\E_{\ux_N\sim f_N}[|\eqref{15}|]$ is controlled by a right-hand side consisting of the normalized relative entropy and an $O(N^{-1})$ additive error \cite{JW2018, LLN2020}; and in fact, versions of such estimates with a the $O(N^{-1})$ error also hold for some Riesz cases \cite{DGR}. 
\end{remark}

In fact, we will obtain \cref{thm:FI} as a special case of \cref{thm:FI'}, which establishes an analogous functional inequality valid for a larger class of ``Riesz-type'' potentials introduced in \cref{ssec:MEpt'} (cf. \cite{NRS2021} for a different class of Riesz-type potentials).

Comparing \cref{thm:FI} to \cite[Proposition 4.1]{NRS2021} or \cite[Proposition 5.15]{RS2021}, one sees that the size of the additive error has been improved to $N^{\frac{\s}{\d}-1} $ while  the previous transport regularity $\|\nab v\|_{L^{\infty}} + \|\Dm^{\frac{\d-\s}{2}}v\|_{L^{\frac{2\d}{\d-\s-2}}}\indic_{\s <\d-2}$ has been replaced by $\|\nab v\|_{L^{\infty}} + \|\Dm^{\frac{\as}{2}}v\|_{L^{\frac{2\d}{\as-2}}}$, for $\as \in (\d,\d+2)$, which, by Sobolev embedding, is stronger.  See \cref{ssec:introRQ} for further comments on this point. On the other hand, comparing \cref{thm:FI} to \cite[Theorem 1.1, (1.7)]{RS2022}, which only depends on $\|\nab v\|_{L^\infty}$, it is evident that the estimate we obtain here is worse in the (super-)Coulomb case. Though, we feel the new proof presented in this paper is more elementary than that of \cite{RS2022}.

\begin{remark}
Although we only consider the whole space $\R^\d$ in this paper, the inequality (and its proof) hold mutatis mutandis on the flat torus $\T^\d$ where $\g$ is now taken to be the solution of $\Dm^{\d-\s}\g = \cd(\delta_0-1)$. We leave the details of this extension to the interested reader. 
\end{remark}

\begin{remark}\label{rem:nonsing}
As mentioned at the very beginning of \cref{ssec:introMot}, we restrict to $0\le \s<\d$ to avoid the nonsingular regime $-2<\s<0$.\footnote{The lower bound $\s>-2$ is natural because for $\s\le -2$, the Riesz potential is no longer positive definite on the class of zero mean test functions.} Nonsingular Riesz interactions are also of interest. For $\d=1$, the value $\s=-1$ is Coulomb. More importantly, the modulated energy \eqref{eq:modenergy} with the diagonal reinserted  (the contribution of the diagonal is, in fact, zero) coincides with the square of what is known in the statistics literature as \emph{maximum mean discrepancy (MMD)} \cite{GBRSS2006, GBRSS2007, GBRSS2012}, a type of integral probability metric \cite{Muller1997}. MMDs for varying choices of kernels are widely used in statistics and machine learning contexts as distances between probability measures. They have  advantages over Wasserstein metrics in terms of the curse of dimensionality \cite{MD2024}, and Riesz MMDs are attractive computationally on account of their slicing property \cite{KNSS2022, HWAH2023}.

In the nonsingular case, the inequality \eqref{main1} holds with the diagonal reinserted on the left-hand side and without the $-\frac{\log\la}{2N}\indic_{\s=0} + C\|\mu\|_{L^\infty}\la^{\d-\s}$ on the right-hand side. This estimate, which is sharp, is shown in the forthcoming work \cite{RSW2025} by D. Slep{\v{c}ev}, L. Wang, and the second author. Because the interaction is nonsingular, there is no ``renormalization'' step, hence no additive error; and so the proof reduces to showing a genuine commutator estimate, which follows from the approach of \cite{NRS2021}. Although the techniques used in this paper can also prove such a commutator estimate, they require more transport regularity, and we therefore leave the result to  \cite{RSW2025}.
\end{remark}

\subsection{Proof method}\label{ssec:intropf}
Our method of proof has two main ingredients: (1) a new potential truncation scheme based on an integral representation of the Riesz potential and  (2) commutator estimates of Kato-Ponce type.

The starting point for the first ingredient is the identity
\begin{align}\label{eq:introgintrep}
\forall x\ne 0,\qquad \g(x) = \mathsf{c}_{\phi,\d,\s}\int_0^\infty t^{\d-\s}\phi_t(x)\frac{dt}{t},
\end{align}
where $\phi$ is any sufficiently nice radial function with nonnegative Fourier transform $\hat{\phi}\ge 0$, $\phi_t(x)\coloneqq t^{-\d}\phi(x/t)$, and $c_{\phi,\d,\s}$ depends on $\phi,\d,\s$. In the log case $\s=0$, the integral has to be renormalized at large scales to obtain a convergent expression due to the fact that $\log|x|$ does not decay at infinity. See \cref{lem:gintrep} for precise assumptions ensuring the validity of the identity. The formula, which is a consequence of scaling, expresses that the Riesz potential is a weighted average of the family of approximate identities $\{\phi_t\}_{t>0}$. In \cite[Chapter 6]{Rubin2024}, Rubin refers to this as a ``wavelet type representation.'' The identity \eqref{eq:introgintrep} may also be understood in terms of the scaling/inversion properties of Mellin transforms.\footnote{If $\tl{f}(r) \coloneqq f(r^{-1})$, then $\mathcal{M}(\tl{f})(z) = \mathcal{M}(f)(-z)$; and if for $\la>0$, $f_\la(r) \coloneqq f(\la r)$, then $\mathcal{M}(f_\la(r))(z) = \la^{-z}\mathcal{M}(f)(z)$ \cite[Appendix B]{FS2009}.} Letting $\mathcal{M}(f)(z)$ denote the Mellin transform of a function $f$ and setting $\psi_{|x|}(t)\coloneqq \phi(|x|/t)$, we have the relation
\begin{align}
\int_0^\infty t^{\d-\s}\phi_t(x)\frac{dt}{t} = \int_0^\infty t^{-\s}\phi(|x|/t)\frac{dt}{t} = \mathcal{M}(\psi_{|x|})(-\s) = |x|^{-\s}\mathcal{M}(\phi)(\s).
\end{align}

The utility of the representation \eqref{eq:introgintrep} is that it conveniently allows us to truncate the singularity of the Riesz potential simply by cutting off the integral for small $t$. More precisely, given a scale $\eta>0$, we define the \emph{truncated potential}
\begin{align}\label{eq:introgeta}
\g_\eta(x) \coloneqq  \mathsf{c}_{\phi,\d,\s}\int_\eta^\infty t^{\d-\s}\phi_t(x)\frac{dt}{t}.
\end{align}
The truncated potential has several important properties that are crucial to the analysis of this paper:
\begin{enumerate}
\item While $\g$ is singular at the origin, $\g_\eta$ is finite and bounded by $\g(\eta)$.
\item $\g_\eta$ is a positive definite kernel (i.e. repulsive, or with nonnegative Fourier transform).
\item $\g_\eta \le \g $ and the difference $\g-\g_\eta$ decays rapidly at scale $\eta$ outside the ball $B(0,\eta)$.
\end{enumerate}
We refer to \cref{lem:geta} for all of the properties of $\g_\eta$.

Without using the representation \eqref{eq:introgintrep}, it is not easy to construct a truncation $\g_\eta$ satisfying all of the three requirements at once, and  
a standard mollification  procedure fails to do so. Moreover, the choice of $\phi$ will require care.
Thanks to these properties,  splitting $\Fr_N(\ux_N,\mu)$ into 
\begin{multline}
\Fr_N(\ux_N, \mu)= \frac12 \int_{(\R^\d)^2}\g_\eta(x-y)d\Big(\frac1N\sum_{i=1}^N \delta_{x_i}-\mu\Big)^{\otimes 2}\\
+ \frac12 \int_{(\R^\d)^2\backslash \triangle }(\g-\g_\eta)(x-y)d\Big(\frac1N\sum_{i=1}^N \delta_{x_i}-\mu\Big)^{\otimes 2} - \frac{\g_\eta(0)}{2N},
\end{multline}
where the diagonal can be reinserted in the first integral up to the well-controlled last term on the second line, we easily obtain two sets of controls by $\Fr_N$. The first is on the first term on the right-hand side, which is nonnegative (thanks to the repulsive nature of $\g_\eta$) and is the square of a genuine maximum mean discrepancy (or metric measuring the distance between $\frac1N\sum_{i=1}^N\delta_{x_i}$ and $\mu$). The second is a control on the second term on the right-hand side which, up to well-controlled additive error terms, is also nonnegative and, thanks to properties of $\g_\eta$, controls $\frac{1}{N^2} \sum_{i\neq j} \g(x_i-x_j)$, i.e.,~the interaction energy coming from particle interactions at small scales.  We refer to \cref{prop:MEmon} for the details.
The former control allows to show in  \cref{prop:coer} a new coercivity estimate asserting that the modulated energy controls the squared inhomogeneous Sobolev norm $\|\frac1N\sum_{i=1}^N\delta_{x_i}-\mu\|_{H^{-\frac{\d}{2}-\varepsilon}}^2$, for any $\varepsilon>0$, up to $O(\la^{\d-\s})$ error. Choosing $\eta \propto N^{-1/\d}$ (the microscale), the latter control allows to bound in terms of the modulated energy the interaction energy due to nearest neighbors at the microscale, a fortiori bounding the number of points with nearest-neighbor distance below the microscale, see \cref{cor:MEcount}.


Let us mention that the potential truncation introduced here plays an essential role in our forthcoming work \cite{HcRS2024oq} on optimal quantization via MMDs with Riesz, and more generally kernels having power law-type spectra. Even though the Riesz kernel is nonsingular in this setting, the truncation still allows to control the nearest-neighbor distances for points in terms of the modulated energy. Such control is crucial for showing the lower bound $\min_{\ux_N\in (\R^\d)^N} \Fr_N(\ux_N,\mu) \ge CN^{\frac{\s}{\d}-1}$,\footnote{Note here that the lower bound is now nonnegative compared to the singular setting $\s\ge 0$. This is because there is no longer a self-interaction removed.} for which we obtain a matching upper bound.

We emphasize that in contrast to the previous works \cite{PS2017, Serfaty2020, NRS2021, RS2022}, there is no need to consider a smearing/regularizing of the charges $\delta_{x_i}$, nor to use smearing radii that are  point-dependent and configuration-dependent. Throughout the present paper, we will only need to consider the truncated potential $\g_\eta$.
In these previous works, the strategy behind an estimate of the form \eqref{main1} is to first prove an estimate valid when the diagonal is reinserted and $\frac{1}{N}\sum_{i=1}^N\delta_{x_i}-\mu$ is replaced by a more regular distribution $f$. One then reduces to this regular setting by replacing the point charges $\delta_{x_i}$ by smeared charges $\delta_{x_i}^{(\eta_i)}$ and estimating the error from this replacement, the so-called renormalization step. Important to implementing the renormalization is to show that the modulated energy controls both the potential energy of the difference $\frac1N\sum_{i=1}^N\delta_{x_i}^{(\eta_i)}-\mu$ and the small-scale interactions.

The present work differs crucially in two regards. First, rather than regularize $\frac{1}{N}\sum_{i=1}^N\delta_{x_i}-\mu$ through smearing, we regularize $\g$ through the truncation $\g_\eta$.\footnote{Remark that the work \cite{BJW2019edp} also considered a scheme based on regularizing the interaction $\g$ at small length scales. However, the procedure in that work is completely different than our own, limited to the spatially periodic setting, and does not yield optimal error estimates.} Second, to prove an estimate for  $\g_\eta$ (see \cref{prop:comm}), we use the definition \eqref{eq:introgeta} to reduce to proving estimates for
\begin{align}
\int_{(\R^\d)^2} (v(x)-v(y))\cdot \nabla \phi\Big(\frac{x-y}{t}\Big)d\Big(\frac{1}{N}\sum_{i=1}^N\delta_{x_i}-\mu\Big)^{\otimes 2}(x,y) , \qquad t \in (0,\infty),
\end{align}
which we integrate over $(\eta,\infty)$ with respect to the measure $t^{\d-\s}\frac{dt}{t}$. Provided the dependence on $v$ in our estimate is scale-invariant, this reduces to proving an estimate for $t=1$. 

We now come to the second main ingredient of our proof. The identity \eqref{eq:introgintrep} is valid for a large class of scaling functions $\phi$, but not all choices of $\phi$ will allow us to prove the desired estimate. Indeed, it is a elementary Fourier computation that no estimate of the form
\begin{align}
\int_{(\R^\d)^2}(v(x)-v(y))\cdot\nab\phi(x-y)f(x)f(y) \le C_v \int_{(\R^\d)^2}\phi(x-y)f(x)f(y)
\end{align}
can hold when $\hat\phi$ decays super-polynomially (e.g. $\phi$ Gaussian). It turns out that  a good choice for $\phi$ is a Bessel potential, i.e. $\hat\phi(\xi) = \jp{2\pi\xi}^{-\as}$, which is the kernel of the inhomogeneous Fourier multiplier $\jp{\nab}^{\as}$, where we use the Japanese bracket notation $\jp{x}^{a} \coloneqq (1+|x|^2)^{a/2}$ for $a\in\R$ and $x\in\R^\d$. The Bessel potential is a screening of the Riesz potential which preserves its local behavior at the origin but avoids the issues at low frequency that lead to the slow spatial decay of the Riesz potential  (see \cref{lem:bespot}). Through duality (i.e. setting $h\coloneqq \phi\ast f$ and using that $\jp{\nab}^{\as}h = f$) the  choice of Bessel potential reduces to proving a product rule for $\jp{\nab}^{\as/2}$. Such product rules are known as Kato-Ponce estimates in the harmonic analysis literature on account of their origins \cite{KP1988}. This is the content of \cref{lem:KP},  our main technical lemma. Here, we rely on commutator estimates of Li \cite{Li2019}, as well as a local representation of fractional powers $\jp{\nab}^{\as/2}$ via dimension extension (see \cref{lem:CSinhomext}), analogous to the representation for the fractional Laplacian popularized by Caffarelli and Silvestre \cite{CS2007}.

The parameter $\as$ for the Bessel potential is a degree of freedom, but not one without constraints. We need $\as>\d$, so that $\phi$ is continuous, bounded. On the other hand, if $\as$ is too large, then we will not be able to obtain an estimate whose dependence on $v$ scales properly (see the proof of \cref{lem:KP}). This leads to the constraint $\as<\d+2$. Even accounting for the constraints on $\as$, it is somewhat remarkable that we obtain a homogeneous estimate in the end through an intermediate inhomogeneous estimate.

Returning to a point alluded to in the previous subsection, let us remark that the identity \eqref{eq:introgintrep} itself suggests a class of potentials beyond the exact Riesz case. Namely, we can replace the function $t^{\d-\s}$ by a general function of $t$ which is pointwise controlled above and below by $t^{\d-\s}$. We call these \emph{Riesz-type} potentials and properly introduce them in \cref{ssec:MEpt'} (see \cref{def:Rtype}).

\subsection{Applications}\label{ssec:introapp}
We now discuss applications of \cref{thm:FI} related to mean-field and supercritical mean-field limits. We will be brief in our remarks, given this subsection closely follows \cite[Section 1.4]{RS2022}, which has further background. 

The first application concerns mean-field limits of first-order dynamics of the form
\begin{equation}\label{eq:MFode}
\dot{x}_i^t = \frac{1}{N}\sum_{1\leq j\leq N : j\neq i} \M\nabla\g(x_i^t-x_j^t) - \mathsf{V}^t(x_i^t) , \qquad i \in [N].
\end{equation}
Here, $\M$ is a $\d\times \d$ constant real matrix such that
\begin{equation}\label{eq:Mnd}
\M\xi\cdot\xi \leq 0 \qquad \forall \xi\in\R^\d,
\end{equation}
 which is a  repulsivity assumption, and $\mathsf{V}^t$ is an external field (e.g. $\mathsf{V}^t=-\nabla V$ for some confining potential $V$). Choosing $\M =-\I$ yields \emph{gradient/dissipative} dynamics, while choosing $\M$ to be antisymmetric yields \emph{Hamiltonian/conservative} dynamics. Mixed flows are also permitted. We assume that $\g$ is of the form \eqref{eq:gmod}. 
One can check that our assumption \eqref{eq:Mnd} for $\M$ ensures that the energy for \eqref{eq:MFode} is $O(1)$ locally uniformly in time if $\mathsf{V}^t$ is $L^1$ in time, Lipschitz in space (in fact, nonincreasing if $\mathsf{V}=0$), therefore if the particles are initially separated, they remain separated for all time, so that there is a unique, global strong solution to the system \eqref{eq:MFode}.


The mean-field limit refers to the convergence as $N \to \infty$ of the {empirical measure}  $\mu_N^t = \frac1N\sum_{i=1}^N \delta_{x_i^t}$
associated to a solution $\ux_N^t = (x_1^t, \dots, x_N^t)$ of the system \eqref{eq:MFode}, 
to a solution $\mu^t$ of the \emph{mean-field equation}
\begin{equation}\label{eq:MFlim}
\partial_t \mu= \div ((\mathsf{V}-\M \nabla \g*\mu) \mu),  \qquad (t,x)\in\R_+\times\R^\d.
\end{equation}
The mean-field limit is qualitatively equivalent to {\it propagation of molecular chaos} (see \cite{Golse2016ln,HM2014} and references therein).  This latter notion means that if $f_N^0(x_1, \dots, x_N) = \mathrm{Law}(X_N^0)$ is $\mu^0$-chaotic (i.e.~the $k$-point marginals $f_{N;k}^0\rightharpoonup (\mu^0)^{\otimes k}$ as $N\rightarrow\infty$ for every fixed $k$), then $f_N^t = \mathrm{Law}(\ux_N^t)$ is $\mu^t$-chaotic. Note that the mean-field limit makes sense for purely deterministic initial data, whereas propagation of chaos is a statement for random initial data. Nevertheless, one can leverage deterministic mean-field convergence rates to prove rates for propagation of chaos, as is possible with our methods (see \cref{rem:pc} below). 

There is a long history to mean-field limits for systems of the form \eqref{eq:MFode}---a proper review of which is beyond the scope of this paper. For a discussion of contributions and the techniques behind them, we refer to the recent survey \cite{CD2021}, the lecture notes \cite{Golse2022ln,Golse2016ln, JW2017_survey, Jabin2014}, and the introductions of \cite{Serfaty2020,NRS2021}. 

In the singular Riesz case, only recently has the full range $\s<\d$ been treated: the sub-Coulomb case $\s<\d-2$, \cite{Hauray2009,CCH2014}, the Coulomb/super-Coulomb case $\d-2\leq \s<\d$ \cite{Duerinckx2016,CFP2012,BO2019,Serfaty2020}, and the full case $0\leq \s<\d$ \cite{BJW2019edp, NRS2021}, thanks in large part to the modulated energy method.  For further extensions of the  method in terms of regularity assumptions and incorporating noise, see \cite{Rosenzweig2022, Rosenzweig2022a} and \cite{Rosenzweig2020spv,RS2021, hC2023}, respectively. For gradient dynamics at positive temperature, the modulated energy combines naturally with the previously used relative entropy \cite{JW2018, GlBM2021, FW2023, RS2024ss} in the form of the modulated free energy \cite{BJW2019crm, BJW2019edp, BJW2020, CdCRS2023, RS2023lsi, RS2024ss, CFGW2024}, which is well-suited to proving entropic propagation of chaos and can even handle logarithmically attractive interactions \cite{BJW2019crm,BJW2020,CdCRS2023a}. 

Although this approach is orthogonal to our own, we mention that techniques based on weighted estimates for hierarchies of cumulants can show propagation of chaos even for some Riesz interactions \cite{BDJ2024}. This approach, and its variant in terms of hierarchies of marginals, is better suited to the positive temperature/diffusive case, in contrast to our zero temperature setting, where stronger results are known \cite{BJS2022, BDJ2024, DJ2025}. While these approaches are currently unable to treat the full {potential} Riesz range, they do have advantages in terms of working for both first- and second-order dynamics, are robust to the precise form of the interaction, and in some cases achieve sharp rates \cite{Lacker2023, LlF2023, hCR2023, Wang2024sharp}, though not for any interaction more singular than log. Note that while these results imply probabilistic statements about the mean-field convergence rates, they do not imply the deterministic convergence rates of the form obtainable with our methods. 
 
As explained in \cref{ssec:introMot}, the optimal rate of convergence as $N\rightarrow\infty$ of the empirical measure $\mu_N^t$ to the solution $\mu^t$ of \eqref{eq:MFlim} in the distance $\Fr_N$ is $N^{\frac{\s}{\d}-1}$. This optimal rate was only known in the (super-)Coulomb case.  It was first established for stationary solutions, corresponding to minimizers of the associated Coulomb/Riesz energy, i.e. stationary solutions of \eqref{eq:MFlim} with $\M=-\mathbb{I}$; see \cite{SS2015, SS2015log, RS2016, PS2017}. More recently, the last two authors \cite{RS2022} proved the same rate for nonstationary solutions. The sub-Coulomb case is only known for stationary solutions on the flat torus \cite{HSSS2017}. Previous works \cite{Hauray2009,CCH2014} obtain convergence rates in Wasserstein distance for the sub-Coulomb case, which are in general not directly comparable to the modulated energy and, per our understanding, are far from sharp.

Our first application of \cref{thm:FI} establishes  mean-field convergence at the optimal rate $N^{\frac{\s}{\d}-1}$ for the full Riesz case. By the now standard modulated-energy method, the proof is an immediate consequence of \cref{thm:FI} (see \cref{ssec:appMFmain}).

\begin{thm}\label{thm:mainMF}
Let $\g$ be of the form \eqref{eq:gmod} and $\mathsf{V}$ satisfy $\int_0^T(\|\nabla\mathsf{V}^t\|_{L^\infty} + \|\Dm^{\frac{\as}{2}}\nab \mathsf{V}^t\|_{L^{\frac{2\d}{\as-2}}}\indic_{\as>2})dt<\infty$ for some $\as\in (\d,\d+2)$ and $T>0$. Assume the equation \eqref{eq:MFlim} admits a solution $\mu \in L^\infty([0,T],\P(\R^\d)\cap L^\infty(\R^\d))$, such that
\begin{equation}\label{nmut}
{\mathcal{N}(u^T) \coloneqq \int_0^T (\|\nab u^t\|_{L^\infty} + \|\Dm^{\frac{\as}{2}}u^t\|_{L^{\frac{2\d}{\as-2}}}\indic_{\as>2})dt< \infty, \qquad \text{where} \ u^t\coloneqq -\M\nab\g\ast\mu^t +\mathsf{V}^t.}
\end{equation}
If $\s=0$, then also assume that $\int_{\R^\d}\log(1+|x|)d\mu^t(x) < \infty$ for all $t\in [0,T]$.\footnote{This condition ensures that the convolution $\g\ast\mu$ is pointwise defined, but none of our estimates will depend on it. Through a Gronwall argument, one checks that if $\mu^0$ satisfies this condition, then $\mu^t$ also satisfies this condition uniformly in $[0,T]$.} Let $\ux_N$ solve \eqref{eq:MFode}. Then there exists a constant  $C>0$, depending only on $\d$, $\s$, such that
\begin{multline}\label{distcoul}
\Fr_N(\ux_N^t,\mu^t) + \frac{\log (N\|\mu^t\|_{L^\infty}) }{2 N\d} \indic_{\s=0} + C\|\mu^t\|_{L^\infty}^{\frac{\s}{\d}}N^{\frac{\s}{\d}-1}\\
 \le Ce^{C\mathcal{N}(u^t)}\Bigg(\Fr_N(\ux_N^0,\mu^0) + \sup_{\tau \in [0,t]}\Big( \frac{\log (N\|\mu^\tau\|_{L^\infty}) }{2 N\d} \indic_{\s=0} + C\|\mu^\tau\|_{L^\infty}^{\frac{\s}{\d}}N^{\frac{\s}{\d}-1}\Big)\Bigg).
\end{multline}

Consequently, if $\mu_N^0 \rightharpoonup \mu^0$ in the weak-* topology for measures and $\lim_{N\to \infty}F_N(\ux_N^0, \mu^0) =0$, then
\begin{equation}
\mu_N^t \rightharpoonup \mu^t \qquad \forall t\in [0,T].
\end{equation}
\end{thm}

As a consequence of \cref{thm:FI}, we also have new rates of convergence in the modulated-energy distance for dynamics with noise (i.e. positive temperature) that improve upon the previous work \cite{RS2021} by the last two authors. The pure modulated-energy approach of \cite{RS2021} is limited to the sub-Coulomb case due to the It\^o correction to the evolution of the energy.

At the microscopic level, the system \eqref{eq:MFode} is generalized to
\begin{equation}\label{eq:MFsde}
\begin{cases}
dx_i^t = \displaystyle\frac{1}{N}\sum_{1\leq j\leq N : j\neq i} \M\nabla\g(x_i^t-x_j^t)dt - \mathsf{V}^t(x_i^t) dt + \sqrt{2/\be}dW_i\\
x_i^t\big|_{t=0} = x_i^\circ,
\end{cases}\qquad i\in[N],
\end{equation}
where $\be \in (0,\infty]$ has the interpretation of inverse temperature, $W_1,\ldots,W_N$ are iid $\d$-dimensional Wiener processes, and the stochastic differential $dW_i$ is understood in the It\^{o} sense. Under the above assumptions on $\M,\g,\mathsf{V}$, the system has a unique strong solution and particles almost surely do not collide. At the macroscopic level, the mean-field equation \eqref{eq:MFlim} gains a diffusion term proportional to $\frac1\beta$,
\begin{align}\label{eq:MFlims}
\partial_t \mu= \div ((\mathsf{V}-\M \nabla \g*\mu) \mu) + \frac1\be \Delta\mu,  \qquad (t,x)\in\R_+\times\R^\d.
\end{align}

Analogous to \cref{thm:mainMF}, we have the following result. Note that we have a larger $O(N^{\frac{\s+2}{\d}-1})$ additive error compared to above. This is a consequence of the  It\^o correction to the evolution of the modulated energy, which leads to an expression corresponding to $-\Fr_N(\ux_N^t,\mu^t)$, but with $\g$ replaced by $-\Delta\g$, which is nonnegative up to an $N^{\frac{\s+2}{\d}-1}$ error.

\begin{thm}\label{thm:mainMFs}
Let $\g$ be of the form \eqref{eq:gmod} for $\s < \d-2$. Then under the same assumptions as \cref{thm:mainMF},
\begin{multline}\label{distcouls}
\E\Bigg[\Fr_N(\ux_N^t,\mu^t) + \frac{\log (N\|\mu^t\|_{L^\infty}) }{2 N\d} \indic_{\s=0} + C\|\mu^t\|_{L^\infty}^{\frac{\s}{\d}}N^{\frac{\s}{\d}-1} + \frac{C}{\be}\|\mu^t\|_{L^\infty}^{\frac{\s+2}{\d}}N^{\frac{\s+2}{\d}-1} \Bigg]\\
 \le Ce^{C\mathcal{N}(u^t)}\Bigg(\Fr_N(\ux_N^0,\mu^0) + \sup_{\tau \in [0,t]}\Big( \frac{\log (N\|\mu^\tau\|_{L^\infty}) }{2 N\d} \indic_{\s=0} + C\|\mu^\tau\|_{L^\infty}^{\frac{\s}{\d}}N^{\frac{\s}{\d}-1}+\frac{C}{\be}\|\mu^\tau\|_{L^\infty}^{\frac{\s+2}{\d}}N^{\frac{\s+2}{\d}-1}\Big)\Bigg).
\end{multline}
\end{thm}

In \cref{sec:appMF}, we obtain \Cref{thm:mainMF,thm:mainMFs} as special cases of the more general \cref{thm:mainMF'} valid for Riesz-type potentials. Before commenting on the second application, we make a few remarks.

\begin{remark}\label{rem:V=0}
When $\mathsf{V}^t=0$, the norm $\|\mu^t\|_{L^\infty}$ is nonincreasing when $\mu^t$ solves \eqref{eq:MFlims} for any $\beta \in (0,\infty]$. Consequently, the estimates \eqref{distcoul} and \eqref{distcouls} respectively simplify to
\begin{multline}\label{distcoul'}
\Fr_N(\ux_N^t,\mu^t) + \frac{\log (N\|\mu^t\|_{L^\infty}) }{2 N\d} \indic_{\s=0} + C\|\mu^t\|_{L^\infty}^{\frac{\s}{\d}}N^{\frac{\s}{\d}-1}\\
 \le Ce^{C\int_0^t (\|\nab u^\tau\|_{L^\infty} + \|\Dm^{\frac{\as}{2}}u^\tau\|_{L^{\frac{2\d}{\as-2}}}) d\tau}\Bigg(\Fr_N(\ux_N^0,\mu^0) +  \frac{\log (N\|\mu^0\|_{L^\infty}) }{2 N\d} \indic_{\s=0} + C\|\mu^0\|_{L^\infty}^{\frac{\s}{\d}}N^{\frac{\s}{\d}-1}\Bigg)
\end{multline}
and
\begin{multline}\label{distcouls'}
\E\Bigg[\Fr_N(\ux_N^t,\mu^t) + \frac{\log (N\|\mu^t\|_{L^\infty}) }{2 N\d} \indic_{\s=0} + C\|\mu^t\|_{L^\infty}^{\frac{\s}{\d}}N^{\frac{\s}{\d}-1} + C\|\mu^t\|_{L^\infty}^{\frac{\s+2}{\d}}N^{\frac{\s+2}{\d}-1} \Bigg]\\
 \le Ce^{C\int_0^t (\|\nab u^\tau\|_{L^\infty} + \|\Dm^{\frac{\as}{2}}u^\tau\|_{L^{\frac{2\d}{\as-2}}}) d\tau}\Bigg(\Fr_N(\ux_N^0,\mu^0) +\frac{\log (N\|\mu^0\|_{L^\infty}) }{2 N\d} \indic_{\s=0} + C\|\mu^0\|_{L^\infty}^{\frac{\s}{\d}}N^{\frac{\s}{\d}-1}\\
 +C\|\mu^0\|_{L^\infty}^{\frac{\s+2}{\d}}N^{\frac{\s+2}{\d}-1}\Bigg).
\end{multline}
\end{remark}

\begin{remark}\label{rem:extVFreg}
The above condition for $\mathsf{V}^t = -c^t x + \Phi^t(x)$ is satisfied, where $\int_0^T |c^t|dt <\infty$, and  $\int_0^T(\|\nab \Phi^t\|_{L^\infty} + \|\Dm^{\frac{\as}{2}} \Phi^t\|_{L^{\frac{2\d}{\as-2}}}\indic_{\as>2})dt < \infty$. Indeed, note that since the singular sub-Coulomb case is vacuous for $\d\le 2$, the condition $\d\ge 3$ and $\as\in (\d,\d+2)$ imply that $\frac{\s}{2}>\frac{3}{2}$. As the (vector-valued) distribution $\Dm^{a}x = 0$ for all $a>1$, the claim now follows by the triangle inequality. 
The regularity assumption \eqref{nmut} for $\mu^t$ is a bit more involved. We defer verification until \cref{ssec:appMFreg}, where we show that \eqref{nmut} is satisfied provided that $\mu^0$ is suitably regular. 
\end{remark}

\begin{remark}\label{rem:pc}
By a standard argument (e.g. see \cite[Remark 3.7]{Serfaty2020} or \cite[Remark 1.5]{RS2021} for details), \cref{thm:mainMF} implies propagation of chaos for the system \eqref{eq:MFode}. However, this ``global-to-local'' argument in general leads to a suboptimal rate  \cite{Lacker2023}.
\end{remark}

\begin{remark}\label{rem:nonsing'}
Building upon \cref{rem:nonsing}, one also has an analogous estimate to \eqref{distcoul} in the nonsingular case $-2<\s\le 0$ without additive errors. This is shown in the aforementioned forthcoming work \cite{RSW2025}. 
\end{remark}

The second application concerns so-called supercritical mean-field limits in the form of the derivation of the Lake equation from the Newtonian $N$-body problem in the joint mean-field and quasineutral limit. In the recent paper \cite{RSlake} by the last two named authors on this problem for the (super-)Coulomb case, the sharp first-order commutator estimate \eqref{main1} plays an essential role to show convergence provided $N^{\frac{\s}{\d}-1}/\vep^2 \rightarrow 0$ as $\vep\rightarrow 0$ and $N\rightarrow\infty$, where $\vep$ is the quasineutral parameter. This result is sharp, as when $N^{\frac{\s}{\d}-1}/\vep^2 \not\rightarrow 0$, an explicit example shows convergence  may fail. The present \cref{thm:FI} would allow one to extend the results of \cite{RSlake} to the sub-Coulomb case, provided one could also extend the necessary auxiliary results concerning regularity of the boundary of support for equilibrium measures to the sub-Coulomb case. In the periodic setting, where the equilibrium measure takes full support, this last part is not an issue.

\subsection{Remaining questions}\label{ssec:introRQ}
Although we establish a sharp first-order commutator estimate for all Riesz modulated energies, our work raises and leaves open several questions that we briefly discuss below.

The first question is the quantitative dependence on $v$ in the right-hand side of \eqref{main1}. In the (super-)Coulomb case, it has been shown \cite{LS2018, Serfaty2023, Rosenzweig2021ne, RS2022}  that the estimate holds with only $\|\nab v\|_{L^\infty}$. This dependence is essentially sharp, as we show in the forthcoming work \cite{hCRScx} that it cannot be weakened to $\|\nab v\|_{BMO}$, except if $\d=1$ and $\s=0$. In the sub-Coulomb case, it has been shown \cite{NRS2021} that the estimate holds with $\|\nab v\|_{L^\infty} + \|\Dm^{\frac{\d-\s}{2}}v\|_{L^{\frac{2\d}{\d-\s-2}}}$. Importantly, both terms scale the same way under the transformation $v\mapsto v(\la\cdot)$. Again, this is essentially sharp, as we show in \cite{hCRScx} that  $\|\Dm^{\frac{\d-\s}{2}}v\|_{L^{\frac{2\d}{\d-\s-2}}}$ cannot be weakened to $\|\Dm^{\frac\as2} v\|_{L^{\frac{2\d}{\as-2}}}$ for any $\as<\d-\s$. The dependence $\|\Dm^{\frac{\mathsf{a}}{2}}v\|_{L^{\frac{2\d}{\mathsf{a}-2}}}$, for any $\mathsf{a}\in (\d,\d+2)$, is an artifact of our passing through an intermediate commutator estimate for Bessel potentials  in the proof of \cref{thm:FI}. By Sobolev embedding, it is stronger than $\|\Dm^{\frac{\d-\s}{2}}v\|_{L^{\frac{2\d}{\d-\s-2}}}$. We expect the latter to be the optimal regularity dependence even with the sharp $N^{\frac{\s}{\d}-1}$ error, but it does not seem obtainable with our current proof. 
 
The second  question is the \emph{localizability} of the estimate. To illustrate what we mean, consider the Coulomb case. Following the so-called electric reformulation of the energy, we may set $h_N\coloneqq \g\ast(\mu_N-\mu)$ and express the modulated energy as a renormalization of $\int_{\R^\d}|\nab h_N|^2$. Evidently, this latter quantity may be localized to any domain $\Omega\subset\R^\d$. Letting $v$ be a transport and taking $\Omega=\supp v$, one would like to establish the localized functional inequality
\begin{align}\label{eq:loccomm}
\int_{(\R^\d)^2}(v(x)-v(y))\cdot\nabla\g(x-y)d(\mu_N-\mu)^{\otimes 2} \le C_v\int_{\Omega}|\nab h_N|^2.
\end{align}
Such localized estimates are important for applications to CLTs for the fluctuations of Riesz gases at mesocales (e.g. see \cite{Serfaty2023}). A sharp estimate of the form \eqref{eq:loccomm} was  shown first in the Coulomb case up to second order \cite{LS2018, Serfaty2023} and then recently for the (super-)Coulomb case at any order \cite{RS2022}.  Even for the local case $\s=\d-2k$, where $\g$ is the fundamental solution of $(-\Delta)^k$, the difficulty remains that our proof of the renormalized commutator itself uses nonlocal estimates in the form of these Kato-Ponce commutator bounds, and therefore is not evidently amenable to showing the desired localized estimates.


The last question we mention is a sharp estimate for higher-order commutators, which state that for $n\ge 1$, the quantity in \eqref{15} is again controlled by $ C(\Fr_N(\ux_N, \mu)+ N^{-\alpha})$. Such inequalities were first shown at second order (i.e. $n=2$) in \cite{LS2018,Rosenzweig2020spv}  for the $\d=2$ Coulomb case, then in \cite{Serfaty2023} for the general  Coulomb case (although with an estimate which is not even optimal in its $\Fr_N$ dependence), and later in \cite{NRS2021} for the full Riesz-type case $0\leq \s<\d$. These second-order estimates were important for the same problem of fluctuations of Coulomb gases, as well as for deriving mean-field limits with multiplicative noise.  
Estimates beyond second order are also useful, allowing to  obtain finer estimates  on the fluctuations of Riesz gases to treat a broader class of interactions \cite{PS2024} and also to compute the asymptotics of $n$-th order cumulants \cite{Rosenzweigcum}. Though not explicitly written, the approach of \cite{NRS2021} yields higher-order estimates, although these are in general not sharp in their additive error. In contrast, the sharp estimates from \cite{Serfaty2023} are only to second order for the $\d=2$ Coulomb case. More importantly, their proof is quite intricate and seems impossible to generalize to higher-order derivatives. In \cite{RS2022}, sharp, localized estimates were shown at all orders for the (super-)Coulomb case via a delicate induction argument. In contrast, it is not obvious how to extend the method of this paper to $n\ge 2$, given that desymmetrizing \eqref{15} and writing in terms of iterated commutators entails difficulties of considering commutators of commutators, not too mention the algebraic complexity as $n$ increases.

\subsection{Organization of article}\label{ssec:introOrg}
In \cref{sec:ME}, we describe the truncation procedure and its basic properties for Riesz potentials (\cref{ssec:MEpt}). We then introduce a class of Riesz-type potentials, alluded to in the introduction, to which we extend the potential truncation (\cref{ssec:MEpt'}). Finally, we use the truncation scheme to prove optimal lower bounds for the modulated energy and small scale energy control (\cref{ssec:MEmon}), as well as a new, sharp coercivity estimate for the modulated energy (\cref{ssec:MEcoer}).

In \cref{sec:KP}, we prove some Kato-Ponce type commutator estimates for powers of the inhomogeneous fractional Laplacian. The main result is \cref{prop:comm}. This is the most technical part of the paper. As a consequence, we have \cref{cor:scaledcomm}, which is the desired functional inequality but for the truncated potential (i.e. pre-renormalization).

In \cref{sec:rcom}, we combine the potential truncation of \cref{sec:ME} with the unrenormalized commutator estimate of \cref{sec:KP} to prove \cref{thm:FI}. In fact, we present with \cref{thm:FI'} a more general functional inequality valid for the Riesz-type potentials introduced in \cref{ssec:MEpt'}, which includes \cref{thm:FI} as a special case.  

In \cref{sec:appMF}, we combine \cref{thm:FI'} with the well-known dissipation relation for the modulated energy to prove \cref{thm:mainMF'}, obtaining \Cref{thm:mainMF,thm:mainMFs} as special cases (\cref{ssec:appMFmain}). We close the paper by verifying that solutions of the mean-field equation \eqref{eq:MFlim} have vector fields satisfying the regularity condition of \cref{thm:FI'} (\cref{ssec:appMFreg}). Note this does not immediately follow from an $L^p$ assumption on the initial density $\mu^0$ since for any $\as \in (\d,\d+2)$,  $\Dm^{\frac{\as}{2}}\nab\g\ast\mu = \cd \nab\Dm^{\frac{\as}{2}+\s-\d}\mu$, which cannot be controlled in terms of $\|\mu\|_{L^p}$ in general, except when $\s<\frac{\d}{2}-1$.

Lastly, in \cref{ap:bessel}, we prove some properties of Bessel potentials, stated in \cref{lem:bespot}, that are used in this paper but which do not seem conveniently available in the literature.

\subsection{Notation}\label{ssec:preN}
We close the introduction with the basic notation used throughout the article without further comment, mostly following the conventions of \cite{NRS2021,RS2021, RS2022}.

Given nonnegative quantities $A$ and $B$, we write $A\lesssim B$ if there exists a constant $C>0$, independent of $A$ and $B$, such that $A\leq CB$. If $A \lesssim B$ and $B\lesssim A$, we write $A\sim B$. Throughout, $C$ will be used to denote a generic constant, possibly changing from line to line. 

$\N$ denotes the natural numbers excluding zero, and $\N_0$ including zero. {For $N\in\N$, we abbreviate $[N]\coloneqq \{1,\ldots,N\}$.} $\R_+$ denotes the positive reals. Given $x\in\R^\d$ and $r>0$, $B(x,r)$ and $\p B(x,r)$ respectively denote the ball and sphere centered at $x$ of radius $r$. Given a function $f$, we denote its support by $\supp f$. The notation $\nabla^{\otimes k}f$ denotes the $k$-tensor field with components $(\p_{i_1}\cdots\p_{i_k}f)_{1\leq i_1,\ldots,i_k\leq \d}$. {For $x\in\R$, $\jp{x}$ denotes $(1+|x|^2)^{1/2}$.}

$\P(\R^\d)$ denotes the space of Borel probability measures on $\R^\d$. If $\mu$ is absolutely continuous with respect to Lebesgue measure, we shall abuse notation by letting $\mu$ denote both the measure and its Lebesgue density. When the measure is clearly understood to be Lebesgue, we shall simply write $\int_{\R^{\d}}f$ instead of $\int_{\R^\d}fdx$. $\Sc(\R^\d)$ denotes the space of Schwartz functions.

Lastly, we use the notation $\hat{f}$ or $\widehat{f}$ to denote the Fourier transform $\hat{f}(\xi) \coloneqq \int_{\R^\d}f(x)e^{-2\pi i \xi\cdot x}dx$ and let $\jp{\nab}$ and $\Dm$ denote the Fourier multipliers with symbols $\jp{2\pi \xi}$ and $|2\pi \xi|$, respectively.

{

\section{The modulated energy}\label{sec:ME}
In this section, we discuss properties of the modulated energy functional introduced in \eqref{eq:modenergy}.

\subsection{Truncation for Riesz potentials}\label{ssec:MEpt}
We first introduce the new scheme for truncating the interaction potential $\g$ based on the identity \eqref{eq:introgintrep}. This replaces the potential regularization scheme introduced in \cite{NRS2021} based on averaging with respect to a mollifier. Although this paper is primarily interested in the sub-Coulomb case $\s<\d-2$, we emphasize that all our results in this section and the next are valid in the Coulomb/super-Coulomb case $\d-2\le \s<\d$, providing an alternative to the truncation scheme used in \cite{PS2017, Serfaty2020, AS2021, RS2022} and importantly not requiring a dimension extension for the nonlocal case $\d-2<\s<\d$. 

\begin{lemma}\label{lem:gintrep}
Let $\phi\in C_b(\R^\d)$ be radial such that $\phi(r)\rightarrow 0$ as $r\rightarrow\infty$, where we use radial symmetry to commit an abuse of notation. Assume that
\begin{align}
&\int_0^\infty r^{\s-1}|\phi(r)|dr < \infty, \qquad  \text{if $\s>0$}, \label{eq:gintrep_phicon} \\
&\int_0^\infty|\log(r)\phi'(r)|dr<\infty \ \text{and} \ \lim_{r\rightarrow 0^+}(\phi(r)-\phi(0))\log r=\lim_{r\rightarrow\infty}\phi(r)\log r = 0, \qquad \text{if \ $\s=0$}.\label{eq:gintrep_phicon'}
\end{align}
Then setting $\phi_t \coloneqq t^{-\d}\phi(\cdot/t)$, it holds that
\begin{align}\label{eq:gintrep}
\forall x\ne 0, \qquad \g(x) = \mathsf{c}_{\phi,\d,\s}\lim_{T\rightarrow\infty}\Bigg(\int_0^T t^{\d-\s}\phi_t(x)\frac{dt}{t}  -C_{\phi,T}\indic_{\s=0}\Bigg),
\end{align}
where\footnote{Implicitly, we are assuming that the quantities in defining $\mathsf{c}_{\phi,\d,\s}$ are nonzero.}
\begin{align}
\mathsf{c}_{\phi,\d,\s}^{-1} &\coloneqq \begin{cases}{\s\int_0^\infty r^{\s}\phi(r)\frac{dr}{r}} ,& {\s>0} \\ {\phi(0)}, & {\s=0} ,\end{cases}  \label{eq:cphids_def}\\
C_{\phi,T} &\coloneqq \phi(0)\log(T)-\int_0^\infty\log(r)\phi'(r)dr. \label{eq:CphiT_def}
\end{align}
\end{lemma}
\begin{proof}
When $\s>0$, the above integral clearly converges whenever $x\neq 0$. On the other hand, making the change of variables $r=|x|/t$ we find that
\begin{align}
\int_0^\infty t^{-\s}\phi(t^{-1}x)\frac{dt}{t}=\s \g(x)\int_0^\infty r^{\s}\phi(r)\frac{dr}{r},
\end{align}
which specifies $\mathsf{c}_{\phi,\d,\s}$. 

When $\s=0$, letting $r=t^{-1}|x|$ and integrating by parts, which is justified by our assumption $\phi(r)\log r\rightarrow 0$ as $r\rightarrow\infty$,
\begin{align}
\int_0^Tt^{-1}\phi(t^{-1}x)dt=\int_{|x|/T}^\infty \phi(r)\frac{dr}{r}=-\phi(|x|/T)\log(|x|/T)-\int_{|x|/T}^\infty \log(r)\phi'(r)dr.
\end{align}
Letting $T\rightarrow\infty$ and using the assumption that $\lim_{r\rightarrow 0^+}(\phi(r)-\phi(0))\log r=0$, we see that \eqref{eq:gintrep} holds with $c_{\phi,\d,0}=\frac{1}{\phi(0)}$.
\end{proof}

As mentioned in \cref{ssec:intropf}, the importance of the representation \eqref{eq:gintrep} is that it conveniently allows to define a truncation of $\g$ at small scales simply by truncating the integral for small values of $t$. This truncation well-approximates $\g$, preserves its repulsivity, and is controlled by $\g$.

More precisely, for $\eta>0$, we define the \emph{truncated potential at scale $\eta$} by
\begin{align}\label{eq:getadef}
\g_\eta(x) \coloneqq  \mathsf{c}_{\phi,\d,\s}\lim_{T\rightarrow\infty}\bigg(\int_\eta^T t^{\d-\s}\phi_t(x)\frac{dt}{t}-C_{\phi,T}\indic_{\s=0}\bigg),
\end{align}
where $C_{\phi,T}$ is as above. Note that $\g=\g_0$. In the sequel, it will be convenient to introduce the notation
\begin{align}\label{eq:fetaldef}
\f_\eta \coloneqq \g-\g_\eta =\mathsf{c}_{\phi,\d,\s}\int_0^\eta t^{\d-\s}\phi_t(x)\frac{dt}{t}. 
\end{align}
The reader may  check that the limit on the right-hand side of \eqref{eq:getadef} holds not only pointwise but also in the sense of tempered distributions. 

The following lemma establishes the key properties of the truncated potential under possibly additional assumptions on the scaling function $\phi$ compared to \cref{lem:gintrep}.

\begin{lemma}\label{lem:geta}
Let $\phi$ be as in \cref{lem:gintrep}. For $\eta>0$, the following assertions hold for a constant $C>0$ depending only on $\d,\s,\phi$.
\begin{enumerate}
\item\label{item:geta0'}
\begin{align}
\g_\eta(0)  =
\begin{cases}
\displaystyle\frac{\mathsf{c}_{\phi,\d,\s}\phi(0)}{\s}\eta^{-\s}, & {\s>0} \\
\displaystyle-\log \eta + \frac{1}{\phi(0)}\int_0^\infty \log (r) \phi'(r)dr, & {\s=0}.
\end{cases}
\end{align}
\item\label{item:geta0} If $\hat\phi\ge 0$ and $\phi\in L^1$, then for any test function $\varphi\ge 0$ (resp. such that $\varphi(0) = 0$ if $\s=0$), we have $\ipp{\widehat{\g_\eta},\varphi}\ge 0$.
\item\label{item:geta1} If $\phi\ge 0$ and $\phi$ is decreasing, then $\g_\eta(x)\leq C\max(\g(\eta),1)$. Moreover, if $\s>0$, then $\g_\eta \ge 0$.
\item\label{item:geta1'} Under the preceding assumptions, $\f_\eta(x)\ge C^{-1}{\g(x)}$ for any $|x|\le \eta$. Additionally, if $\s=0$, then $\f_\eta(x)\geq {\frac{\phi(1)}{\phi(0)}\g(\frac{x}{\eta})}$ and  {$|\f_\eta(x) -\g(\frac{x}{\eta})| \le C$ for all $|x|\le \eta$.}  

\item\label{item:geta2} Suppose further that $\phi(x)\le C_\ga\jp{x}^{-\ga}$ for $\ga>\s$. Then $0 \le \f_\eta(x) \le C_\ga{\frac{\eta^{\ga-\s}}{|x|^{\gamma}}}$, and consequently,
$\f_\eta(x) \leq  C_\ga\min(\eta^{-\s}\g(\frac{x}{\eta}),{\frac{\eta^{\ga-\s}}{|x|^{\ga}}})$.
\item\label{item:geta3} Suppose further that $\phi\in C^1(\R^\d\setminus\{0\})$ and there is $c >0$ such that $|x| |\nab\phi(x)| \le \phi(c x)$. Then $|x||\nabla \f_\eta(x)|\leq C \f_{\eta/c}(x)$.
\end{enumerate}
\end{lemma}
\begin{proof}
Starting with (\ref{item:geta0'}), observe that
\begin{align}
\frac{\g_\eta(0)}{\mathsf{c}_{\phi,\d,\s}} &=  \lim_{T\rightarrow\infty}\bigg(\phi(0)\int_\eta^T t^{-\s}\frac{dt}{t} - C_{\phi,T}\indic_{\s=0}\bigg)\nn\\
&=\lim_{T\rightarrow\infty}\bigg( \frac{\phi(0)(\eta^{-\s}-T^{-\s})}{\s}\indic_{\s>0} + \phi(0)\log(\frac{T}{\eta})\indic_{\s=0} - \big(\phi(0)\log(T)+\int_0^\infty\log(r)\phi'(r)dr\big)\indic_{\s=0}\bigg) \nn\\
&=\bigg(\frac{\phi(0)\eta^{-\s}}{\s}\indic_{\s>0} +\big(-\phi(0)\log\eta  + \int_0^\infty \log r \phi'(r)dr \big)\indic_{\s=0}\bigg). \label{eq:geta0form}
\end{align}
Substituting in the definition \eqref{eq:cphids_def} of $c_{\phi,\d,\s}$ then yields (\ref{item:geta0'}).

To see (\ref{item:geta0}), observe from Fubini-Tonelli and the identity $\wh{\phi_t} = \hat{\phi}(t\cdot)$ that for any $T>0$, the Fourier transform
\begin{align}
\int_{\R^\d}e^{-2\pi ix\cdot\xi}\int_{\eta}^T t^{\d-\s}\phi_t(x)\frac{dt}{t}dx = \int_{\eta}^T t^{\d-\s}\hat\phi(t \xi)\frac{dt}{t}.
\end{align}
Since $\hat{1} = \delta_0$ in the sense of distributions, it follows from the continuity of the Fourier transform with respect to distributional limits that
\begin{align}\label{eq:getahat}
\widehat{\g_\eta}(\xi) = \mathsf{c}_{\phi,\d,\s}\lim_{T\rightarrow\infty}\bigg(\int_\eta^T t^{\d-\s}\hat\phi(t\xi)\frac{dt}{t}-C_{\phi,T}\delta_0\indic_{\s=0}\bigg),
\end{align}
where the equality is understood in the sense of distributions. Evidently, the first term on the right-hand side is positive. Consequently, if $\s>0$, then for any test function $\varphi\ge 0$, we have $\ipp{\wh{\g_\eta},\varphi} \ge 0$. If $\s=0$, then the additional hypothesis $\varphi(0)=0$ kills off the contribution of the Dirac mass. This shows (\ref{item:geta0}).

Now consider (\ref{item:geta1}). If $\phi\ge 0$, then when $\s>0$, it is evident from the definition \eqref{eq:getadef} that $\g_\eta\ge 0$. Trivially bounding $\phi_t \le t^{-\d}\|\phi\|_{L^\infty}$, we also have
\begin{align}
\g_\eta(x) = \mathsf{c}_{\phi,\d,\s}\int_\eta^\infty t^{\d-\s}\phi_t(x)\frac{dt}{t} \le \|\phi\|_{L^\infty}\mathsf{c}_{\phi,\d,\s}\int_{\eta}^\infty t^{-\s}\frac{dt}{t} = \|\phi\|_{L^\infty}\mathsf{c}_{\phi,\d,\s}\g(\eta) .
\end{align}
When $\s=0$, note that by making the same change of variable and integrating by parts as in the proof of \cref{lem:gintrep}, 
\begin{align}
\int_\eta^Tt^{\d}\phi_t(x)\frac{dt}{t}-C_{\phi,T} &= \phi\left(\frac{|x|}{\eta}\right)\log \frac{|x|}{\eta} -\phi\left(\frac{|x|}{T}\right)\log\frac{|x|}{T} -\int_{|x|/T}^{|x|/\eta} \log(r)\phi'(r)dr  - C_{\phi,T}\nn\\
&= \phi\left(\frac{|x|}{\eta}\right)\log\frac{|x|}{\eta} -\phi(0)\log|x| - \Big(\phi\left(\frac{|x|}{T}\right)-\phi(0)\Big)\log\frac{|x|}{T} \nn\\
&\ph+ \int_{|x|/\eta}^\infty\log(r)\phi'(r)dr +  \int_0^{|x|/T}\log(r)\phi'(r)dr.
\end{align}
Letting $T\rightarrow\infty$, the last two terms on the preceding line vanish. Thus,
\begin{align}\label{eq:getasum}
\g_\eta(x) &= \frac{\phi(\frac{|x|}{\eta})}{\phi(0)}\log \frac{|x|}{\eta}+ \frac{1}{\phi(0)}\int_{|x|/\eta}^\infty\log(r)\phi'(r)dr  -\log|x|.
\end{align}
By assumption \eqref{eq:gintrep_phicon'} and the monotonicity of $\phi$,
\begin{align}
\frac{1}{\phi(0)}\int_{|x|/\eta}^\infty\log(r)\phi'(r)dr \le \frac{1}{\phi(0)}\int_{|x|/\eta}^{\min(|x|/\eta,1)}|\log(r)\phi'(r)|dr < \infty,
\end{align}
which takes care of the second term on the right-hand side of \eqref{eq:getasum}. If $|x|\ge \eta$, then the first is $\le \log\frac{|x|}{\eta}$ by our assumption that $\phi$ is decreasing, while the third term is trivially $\le -\log\eta$. Writing 
\begin{align}
\frac{\phi(\frac{|x|}{\eta})}{\phi(0)}\log \frac{|x|}{\eta} - \log|x| = \frac{\phi(\frac{|x|}{\eta}) - \phi(0)}{\phi(0)}\log \frac{|x|}{\eta} - \log\eta,
\end{align}
By our assumption that $(\phi(r)-\phi(0))\log r \rightarrow 0$ as $r\rightarrow 0$, implies that the first term is $O(1)$ for $|x|\le\eta$. 
The desired conclusion now follows from assembling all the preceding cases. This completes the argument for (\ref{item:geta1}).

Next, consider (\ref{item:geta1'}). Since $\phi$ is decreasing, for any $|x|\le \eta$, 
\begin{align}
\f_\eta(x) =\mathsf{c}_{\phi,\d,\s}\int_0^\eta t^{-\s}\phi\Big(\frac{x}{t}\Big)\frac{dt}{t}= \mathsf{c}_{\phi,\d,\s}|x|^{-\s}\int_{|x|/\eta}^\infty r^{-\s}  \phi(r)\frac{dr}{r}\geq \mathsf{c}_{\phi,\d,\s}|x|^{-\s}\int_1^\infty r^{-\s}  \phi(r)\frac{dr}{r},
\end{align}
where the second equality follows from the change of variable $r=|x| t^{-1}$. If $\s=0$, then \eqref{eq:getasum} implies that
\begin{align}
\f_\eta(x) = -\frac{\phi(\frac{|x|}{\eta})}{\phi(0)}\log \frac{|x|}{\eta} - \frac{1}{\phi(0)}\int_{|x|/\eta}^\infty\log(r)\phi'(r)dr.
\end{align}
Since $\phi$ is decreasing it thus holds that
\begin{align}
	\f_\eta(x)&\geq -\frac{\phi(\frac{|x|}{\eta})}{\phi(0)}\log \frac{|x|}{\eta}- \frac{\log(|x|/\eta)}{\phi(0)}\int_{|x|/\eta}^1\phi'(r)dr = -\frac{\phi(1)}{\phi(0)}\log \frac{|x|}{\eta}.
\end{align}
On the other hand, the triangle inequality implies that
\begin{align}
\sup_{|x|\le \eta}\Big|\f_\eta(x) +\log\frac{|x|}{\eta}\Big| \le \sup_{0<r\le 1}\frac{|\phi(r)-\phi(0)| |\log r|}{\phi(0)}  + \frac{1}{\phi(0)}\int_0^\infty|\log r| |\phi'(r)|dr <\infty.
\end{align}
This then completes the proof of (\ref{item:geta1'}).

For (\ref{item:geta2}), it is immediate from the definition of $\g_\eta$ and the assumption $\phi\ge 0$ that $\f_\eta\ge 0$. If $\s>0$, then since $\g_\eta \ge 0$, it also follows that $0\le \f_\eta(x)\leq \g(x)$.  Majorizing $|\phi_t(x)| \le C_\ga t^{-\d}\jp{x/t}^{-\ga}$, we have 
\begin{align}
\f_\eta(x)= \mathsf{c}_{\phi,\d,\s}\int_0^\eta t^{\d-\s}\phi_t(x)\frac{dt}{t} &\le   C_\ga\mathsf{c}_{\phi,\d,\s}\int_0^\eta t^{-\s} \jp{|x|/t}^{-\ga}\frac{dt}{t}  \nn\\
&= C_\ga\mathsf{c}_{\phi,\d,\s}|x|^{-\s}\int_{\frac{|x|}{\eta}}^\infty r^{\s}\jp{r}^{-\ga}\frac{dr}{r}  \nn\\
&\le \frac{C_\ga\mathsf{c}_{\phi,\d,\s}}{\ga-\s}|x|^{-\s} (|x|/\eta)^{\s-\ga},
\end{align}
where we have made the change of variable $r=t^{-1}|x|$ and used that the integral in $r$ converges since $\ga>\s$ by assumption. This establishes (\ref{item:geta2}).


Lastly, for (\ref{item:geta3}), we use the definition \eqref{eq:getadef} of $\g_\eta$ to see that
\begin{align}
|x||\nabla\f_\eta(x)|=|x|\bigg| \mathsf{c}_{\phi,\d,\s}\int_0^\eta t^{\d-\s}\nabla\phi_t(x)\frac{dt}{t}\bigg|\leq \mathsf{c}_{\phi,\d,\s}\int_0^\eta t^{-\s}|t^{-1}x||\nabla\phi(t^{-1}x)|\frac{dt}{t}.
\end{align}
By assumption, $|t^{-1}x||\nabla\phi(t^{-1}x)|\leq C\phi(ct^{-1}x)$, for $C>0$, and therefore,
\begin{align}
|x||\nabla\f_\eta(x)|\leq \mathsf{c}_{\phi,\d,\s}C\int_0^\eta t^{-\s}\phi(c t^{-1}x)\frac{dt}{t} &= c^{-\s}\mathsf{c}_{\phi,\d,\s}C\int_0^{\eta/c}t^{-\s}\phi(t^{-1}x)\frac{dt}{t} \nn\\
&= c^{-\s}\mathsf{c}_{\phi,\d,\s}C\f_{\eta/c}(x),
\end{align}
where the penultimate equality follows from the change of variable $t/c\mapsto t$. With this last assertion, the proof of the lemma is complete.    
\end{proof}

\begin{remark}\label{rem:phidec}
The assumption that $\phi$ is decreasing is not strictly necessary, but it does simplify the computations. In any case, it is satisfied by the specific choices of $\phi$ we will use later in this paper to obtain \cref{thm:FI}. 
\end{remark}

\begin{remark}\label{rem:ggetadiffft}
The identity \eqref{eq:getahat} shows that if $\hat{\phi}\ge 0$, then $\widehat{\f_\eta}\ge 0$ in the sense of distributions for all $0\le \s<\d$.
\end{remark}

Let us stress the importance of property (\ref{item:geta2}) of $\f_\eta$ to obtain the optimal additive error $\eta^{\d-\s}$. In the previous work \cite{NRS2021}, the standard mollification scheme used to define $\g_\eta$ only yields the bound $|\f_\eta (x)| \le C\frac{\eta^2}{|x|^{\s+2}}$. This issue of slow spatial decay was the main culprit behind the sub-optimality of the error estimates in that work. Note that for the potential truncation scheme used for the (super-)Coulomb case in \cite{PS2017, Serfaty2020, RS2022}, $\f_\eta(x)$ vanishes when $|x|\ge\eta$, so there was no comparable issue.

So far, we have been unspecific about the choice of $\phi$, only requiring that it satisfies certain properties. For later use in \cref{sec:rcom}, we will choose $\phi$ to be a Bessel potential, that is the fundamental solution of the inhomogeneous Fourier multiplier  whose symbol is given by $\jp{2\pi\xi}^{\as}$. In the next lemma, we recall some important properties of the Bessel potential. Some assertions of the lemma are already in \cite[Proposition 1.2.5]{Grafakos2014m}. We include a proof of them in \cref{ap:bessel} in an effort to make the present work self-contained.

\begin{lemma}\label{lem:bespot}
Let $\as>0$ and $G_\as(x)$ be defined by $\wh{G_\as}(\xi) = (1+4\pi^2|\xi|^2)^{-\as/2}$. Then $G_\as\in L^1(\R^\d) \cap C^\infty(\R^\d\setminus\{0\})$ is strictly positive, radial, and decreasing. $G_{\as}$ belongs to the H\"older-Zygmund space $\mathcal{C}^{\as-\d}$ and for any multi-index $\vec\al\in \N_0^\d$, we have for any $\epsilon>0$,
\begin{align}\label{eq:bespot1}
\forall |x|\ge 2, \qquad |\p_{\vec\al}G_\as(x)| \lesssim_{\ep} e^{-\frac{|x|}{4+\epsilon}},
\end{align}
\begin{align}\label{eq:bespot2}
\forall |x|\le 2, \qquad G_{\as}(x) \approx \begin{cases}1, & {\as>\d} \\ -\log\frac{|x|}{4}, & {\as = \d} \\ |x|^{\as-\d}, & {\as<\d}, \end{cases}
\end{align}
and
\begin{align}\label{eq:bespot2'}
\forall |x|\le 2, \qquad |\p_{\vec\al}G_{\as}(x)| \lesssim 1+|x|^{\as-|\vec\al|-\d} - \log(\frac{|x|}{2})\indic_{\as = |\vec\al|+\d}.
\end{align}
Lastly, for any $\epsilon\in(0,1)$
\begin{align}\label{eq:bespot3}
\forall |x|>0, \qquad |x|^{|\vec\al|} |\p_{\vec\al}G_\as(x)| \lesssim_\ep G_\as(\epsilon x)
\end{align}
Above, the implicit constants depend on $\d,\as,|\vec\al|$. 
\end{lemma}

\begin{remark}\label{rem:nabGsvan}
An examination of the proof of \cref{lem:bespot} reveals that if $\as>\d$, then for any $\epsilon>0$ the matrix field $x\otimes \nab G_\as(x)$ has an $(\as-\d-\epsilon)$-H\"older continuous extension that vanishes at $x=0$. As
\begin{align}
|x\otimes\nab\g_\eta(x) - y\otimes\nab\g_\eta(y)|  &=  \bigg|\mathsf{c}_{\phi,\d,\s}\int_\eta^\infty t^{-\s-2} \Big(x\otimes\nab G_\as(x/t ) - y\otimes\nab G_\as(y/t ) dt\bigg| \nn\\
&\le C_\epsilon \int_{\eta}^\infty t^{-\s-1} |(x-y)/t|^{\as-\d-\epsilon} dt \nn\\
&= C_{\epsilon,\as,\s,\d}\eta^{-\s-\as+\d+\ep}|x-y|^{\as-\d-\epsilon},
\end{align}
provided that $\s+\as-\d-\ep>0$, it follows that $x\otimes\nab\g_\eta$ also admits an $(\as-\d-\epsilon)$-H\"older continuous extension that vanishes at the origin, for any $\eta>0$.

\end{remark}

\begin{remark}\label{rem:getaBP}
In particular, \cref{lem:bespot} shows that assertions (\ref{item:geta0})-(\ref{item:geta1'}) of \cref{lem:geta} are satisfied when $\phi = G_\as$ for $\as>\d$. Furthermore, assertion (\ref{item:geta2}) holds for any $\ga>\s$, and assertion (\ref{item:geta3}) holds for any $c\in (0,1)$. 
\end{remark}

\subsection{Truncation for Riesz-type potentials}\label{ssec:MEpt'}
The identity \eqref{eq:gintrep} and the fact that the choice $\phi = G_\as$ satisfies all the assumptions of \cref{lem:geta} motivate our considering a class of \emph{Riesz-type} potentials, mentioned in the introduction, that have properties analogous to those of the exact Riesz potential.

\begin{mydef}\label{def:Rtype}
Let $\phi:\R^\d\rightarrow[0,\infty)$ satisfy all the conditions imposed in \Cref{lem:gintrep,lem:geta}. 

We say that a potential $\g:\R^\d\setminus\{0\}\rightarrow \mathbb{R}$ is \emph{$(\s,\phi)$-admissible} if $\g$ admits the representation
\begin{align}
\forall x\ne 0, \qquad \g(x) = \lim_{T\rightarrow\infty}\mathsf{c}_{\phi,\d,\s}\Bigg(\int_0^T \zeta(t) \phi_t(x)\frac{dt}{t}  -C_{\phi,T}\indic_{\s=0}\Bigg),
\end{align}
where $c_{\phi,\d,\s}, C_{\phi,T}$ are as in \eqref{eq:cphids_def}, \eqref{eq:CphiT_def}, respectively,  $\zeta: (0,\infty)\rightarrow [0,\infty)$ is a locally integrable function such that for some $C_{\zeta}>0$,
\begin{align}\label{eq:zetas>0}
\forall t>0, \qquad C_{\zeta}^{-1} t^{\d-\s} \le \zeta(t) \le C_\zeta t^{\d-\s},
\end{align}
and if  $\s=0$, then there is a locally integrable function $\rho: (0,\infty) \rightarrow [0,\infty)$ such that
\begin{align}\label{eq:zetas=0}
\zeta(t)= t^{\d} + \rho(t), \qquad \text{where} \ \int_0^\infty \rho(t) t^{-\d}\frac{dt}{t}<\infty.
\end{align}
When $\phi=G_\as$ for some $\as>0$, we simply say that $\g$ is \emph{$(\s,\as)$-admissible} and we let $\mathsf{c}_{\as,\d,\s}:= \mathsf{c}_{G_{\as},\d,\s}$.
\end{mydef}

Unless specified otherwise, $C_\zeta$ hereafter will  refer to the constant in \eqref{eq:zetas>0}. Whenever the notation $C_\zeta$ appears in the $\s=0$ case, one should simply take $C_\zeta = 1$.

To the best of our knowledge, this class of admissible potentials has not been considered in the literature at this level of generality. We note that when $\phi$ is Gaussian, our definition coincides with a special case of so-called \emph{G-type} potentials considered in \cite[Definition 10.5.]{BHS2019} and related works. The work \cite{NRS2021} considered a class of Riez-type potentials that satisfy a superharmonicity condition in a ball around the origin as well as satisfying pointwise physical and Fourier bounds comparable to an exact-Riesz potential. In the super-Coulomb case, similar assumptions are imposed but in a dimension-extended space, viewing the potential $\g$ as the restriction of some $\mathsf{G}$ to a lower-dimensional subspace. Similarly, the reader may check from \cref{def:Rtype} that an $(\s,\phi)$-admissible potential, for $\s>0$, satisfies pointwise bounds in physical and Fourier space comparable to an exact Riesz potential. In the case $\s=0$, an $(\s,\phi)$-admissible potential agrees with the $-\log|x|$ up to a more regular, positive-definite remainder. It is not clear that one definition is a proper subset of the other. In particular, the role of superharmonicity here is not apparent as it is in \cite{NRS2021}. Although the class in \cite{NRS2021} seems broader, in our opinion the generalization in the present work is more straightforward.

Let $\eta>0$. If $\s>0$, then we define $\g_\eta$ just as in \eqref{eq:getadef}. If $\s=0$, then we only need to truncate the contribution from $t^\d$ in the decomposition for $\zeta$. This is because the assumption $\int_0^\infty\rho(t)t^{-\d}\frac{dt}{t}<\infty$ ensures that $\int_0^\infty \rho(t)\phi_t(x)\frac{dt}{t}$ converges for any $x\in\R^\d$. More precisely, we define
\begin{align}\label{eq:getadef'}
\g_\eta(x) \coloneqq c_{\phi,\d,\s}\begin{cases}\displaystyle \int_\eta^{\infty} \zeta(t)\phi_t(x)\frac{dt}{t} , & {\s\ne 0} \\ \displaystyle \lim_{T\rightarrow\infty}\bigg(\int_\eta^T t^{\d}\phi_t(x)\frac{dt}{t} + \int_0^\infty \rho(t)\phi_t(x)\frac{dt}{t}-C_{\phi,T}\indic_{\s=0}\bigg) ,& {\s=0}. \end{cases}
\end{align}

To distinguish between the exact Riesz potential considered in the previous subsection and Riesz-type potentials, it will be convenient to introduce the notation
\begin{align}\label{eq:gRiesz}
\g_{\Rs}(x) \coloneqq \begin{cases}\frac{1}{\s}|x|^{-\s}, & {\s\ne 0} \\ -\log|x|, & {\s=0}.\end{cases}
\end{align}
Similarly, we write $\g_{\Rs,\eta}, \f_{\Rs,\eta}$ to denote $\g_{\eta}, \f_{\eta}$ in the exact Riesz case.

Analogous to \cref{lem:geta}, we have the following proposition, which is the main result of this subsection and underlies the results of \Cref{sec:ME,sec:rcom}.

\begin{prop}\label{prop:geta}
Let $\g$ be $(\s,\phi)$-admissible. The following assertions hold.
\begin{enumerate}
\item\label{item:'geta0'}
It holds that
\begin{align}
\begin{cases}
\displaystyle C_{\zeta}^{-1} \le  \frac{\s\g_\eta(0)}{\mathsf{c}_{\phi,\d,\s}\phi(0)\eta^{-\s}} \leq C_\zeta, & {\s>0} \\
\displaystyle \g_\eta(0) = -\log \eta + \frac{1}{\phi(0)}\int_0^\infty(\log(t)\phi'(t) + \rho(t)t^{-\d-1}\phi(0))dt, & {\s=0}.
\end{cases}
\end{align}
\item\label{item:'geta0} For any test function $\varphi\ge 0$ (resp. such that $\varphi(0) = 0$ if $\s=0$), we have $\ipp{\widehat{\g_\eta},\varphi}\ge 0$.
\item\label{item:'geta1} $\g_\eta(x)\leq CC_\zeta\max(\g(\eta),1)$. Moreover, if $\s>0$, then $\g_\eta \ge 0$.
\item\label{item:'geta1'} $\f_\eta(x)\ge CC_\zeta^{-1}|x|^{-s}$ for any $|x|\le \eta$. Additionally, if $\s=0$, then $\f_\eta(x)\geq -\frac{\phi(1)}{\phi(0)}\log(\frac{|x|}{\eta})$ and $|\f_\eta(x) + \log \frac{|x|}{\eta}| \le C$ for all $|x|\le \eta$.
\item\label{item:'geta2} For $\ga>\s$, we have $0\leq \f_\eta(x)\leq  C_\ga C_\zeta {\frac{\eta^{\ga-\s}}{|x|^{\ga}}}$. Consequently,  $\f_\eta(x) \leq  C_\ga C_\zeta\min(\eta^{-\s}\g(\frac{x}{\eta}),{\frac{\eta^{\ga-\s}}{|x|^{\ga}}})$.
\item\label{item:'geta3} $|x||\nabla\f_\eta(x)|\leq C_\zeta^2\f_{\eta/c}(x)$ for some $c>0$. 
\end{enumerate}
In all cases, the constant $C>0$ depends only on $\d,\s,\phi$. In particular, all the assertions hold, for any $\ga >\d$ in (\ref{item:'geta2}) and any $c \in (0,1)$ in (\ref{item:'geta3}), if $\g$ is $(\s,\as)$-admissible for $\as>\d$.  
\end{prop}

\begin{proof}
The assertions essentially follow from the same reasoning as in the proof of \cref{lem:geta}. We briefly sketch the details.

Consider first the case $\s>0$. Then by the assumption \eqref{eq:zetas>0} for $\zeta$, it follows that
\begin{align}
C_\zeta^{-1}\g_{\Rs,\eta}(x) = \frac{c_{\phi,\d,\s}}{C_\zeta}\int_\eta^{\infty} t^{\d-s}\phi_t(x)\frac{dt}{t} \le \g_\eta(x) =  c_{\phi,\d,\s}\int_\eta^{\infty} \zeta(t)\phi_t(x)\frac{dt}{t} &\le C_\zeta c_{\phi,\d,\s}\int_\eta^\infty t^{\d-\s}\phi_t(x)\frac{dt}{t}\nn\\
& = C_\zeta\g_{\Rs,\eta}(x).
\end{align}
Similarly,
\begin{align}
C_\zeta^{-1}\f_{\Rs,\eta}(x) = \frac{c_{\phi,\d,\s}}{C_\zeta}\int_0^\eta t^{\d-\s}\phi_t(x)\frac{dt}{t} \le \f_{\eta}(x) = c_{\phi,\d,\s}\int_0^\eta \zeta(t)\phi_t(x)\frac{dt}{t}  &\le C_\zeta c_{\phi,\d,\s}\int_0^\eta t^{\d-\s}\phi_t(x)\frac{dt}{t} \nn\\
& = C_\zeta\f_{\Rs,\eta}(x).
\end{align}
We can now appeal to the corresponding assertions for $\g_{\Rs,\eta}, \f_{\Rs,\eta}$ from \cref{lem:geta} to conclude the proposition in the case $\s>0$, with the exception of assertion (\ref{item:'geta3}). For this, we may simply bound
\begin{align}\label{eq:xnabf1}
|x||\nabla\f_\eta(x)| \le  \mathsf{c}_{\phi,\d,\s}\int_0^\eta \zeta(t) |x||\nabla\phi_t(x)|\frac{dt}{t} &\le C_\zeta \mathsf{c}_{\phi,\d,\s}\int_0^\eta t^{-\s}|t^{-1}x||\nabla\phi(t^{-1}x)|\frac{dt}{t} \nn\\
&\le C C_\zeta c^{-\s}\f_{\Rs,\eta/c}(x) \le C C_\zeta^2 c^{-\s}\f_{\eta/c}(x).
\end{align}
where the final inequality follows from proceeding exactly as in the proof of \cref{lem:geta}(\ref{item:geta3}). The desired conclusion now follows. 
and 

For the case $\s=0$, we observe from assumption \eqref{eq:zetas=0} and the definition \eqref{eq:getadef'} of $\g_\eta$ that
\begin{align}
\g_{\eta}(x) - \g_{\Rs,\eta}(x) &= \int_0^\infty\rho(t)\phi_t(x)\frac{dt}{t}, \label{eq:s=01} \\
\f_\eta(x) &= \int_0^\eta t^{\d}\phi_t(x)\frac{dt}{t} = \f_{\Rs,\eta}(x). \label{eq:s=02}
\end{align}
The relation \eqref{eq:s=01} gives assertions (\ref{item:'geta0'}), (\ref{item:'geta0}), (\ref{item:'geta1}); and the relation \eqref{eq:s=02} yields the assertions (\ref{item:'geta1'}), (\ref{item:'geta2}), (\ref{item:'geta3}) for $\f_\eta$. This then completes the proof of the proposition. 

\end{proof}

By \cref{lem:geta}(\ref{item:geta2}) and H\"older's inequality, we have the following convolution bounds, which will be used in \cref{ssec:MEmon} to control the additive errors.

\begin{lemma}\label{lem:feta}
Let $\g$ be $(\s,\phi)$-admissible. Let $p \in (\frac{\d}{\d-\s}, p]$. There exists $C = C(\d,\s,p,\phi) > 0$, such that for $\eta\ge 0$,
\begin{align}
\|\f_\eta\ast\mu\|_{L^\infty} \le C C_\zeta\|\mu\|_{L^{p}}\eta^{\frac{\d(p-1)}{p}-\s}, \label{eq:fetaHol} \\
\| |x||\nab \f_\eta|\ast \mu\|_{L^\infty} \le C {C_\zeta}\|\mu\|_{L^{p}}\eta^{\frac{\d(p-1)}{p}-\s}. \label{eq:nabfetaHol}
\end{align}
\end{lemma}
\begin{proof}
For $x\in\R^\d$, we decompose
\begin{align}
\int_{\R^\d}\f_\eta(x-y)d\mu(y) &= \int_{|x-y|\le \eta/2}\f_\eta(x-y)d\mu(y) + \int_{|x-y|>\eta/2}\f_\eta(x-y)d\mu(y) \nn\\
&\le CC_\zeta\eta^{-\s}\int_{|x-y|\le \eta/2}\g\left(\frac{x-y}{\eta}\right)d\mu(y) \nn\\
&\ph + C C_\zeta\int_{|x-y|>\eta/2}|x-y|^{-\s} (|x-y|/\eta)^{\s-\ga}d\mu(y) \nn\\
&\le C'C_\zeta \eta^{\frac{\d(p_1-1)}{p_1}}\Bigg(1  +\indic_{\s=0}\Big(\int_0^1 |\log r|^{\frac{p_1}{p_1-1}}dr \Big)^{\frac{p_1-1}{p_1}}\Bigg)\|\mu\|_{L^{p_1}} \nn\\
&\ph+ C' C_\zeta\eta^{\frac{\d(p_2-1)}{p_2}-\s}\Big(\int_{1}^\infty r^{-\frac{\s p_2}{p_2-1}} r^{\frac{p_2(\s-\ga)}{p_2-1}}dr\Big)^{\frac{p_2-1}{p_2}} \|\mu\|_{L^{p_2}} \nn\\
&= C''C_\zeta\Big(\|\mu\|_{L^{p_1}}\eta^{\frac{\d(p_1-1)}{p_1}-\s}+ \|\mu\|_{L^{p_2}}\eta^{\frac{\d(p_2-1)}{p_2}-\s}\Big),
\end{align}
where we have used \cref{prop:geta}(\ref{item:'geta2}) and the applications of H\"older's inequality are valid provided that $\frac{\d}{\d-\s}<p_1\le \infty$ and $1\le p_2 < \frac{\d}{\d-\ga}$, with the convention that $p_2=\infty$ is allowed if $\ga>\d$. In particular, if $\ga>\d$, then we may take $p_1=p_2$. Similarly,  using the estimates  \cref{prop:geta}(\ref{item:geta3}) and then applying the previous estimate \eqref{eq:fetaHol} (assuming $\eta \le \frac{c}{2}$), for any $x\in\R^\d$, 
\begin{align}
\int_{\R^\d}|x-y| |\nab\f_\eta(x-y)|d\mu(y) \le C {C_\zeta}\int_{\R^\d}\f_{\R,\eta/c}(x-y)d\mu(y) \le C' C_\zeta^2\|\mu\|_{L^p}\eta^{\frac{\d(p-1)}{p}-\s}, 
\end{align}
again under the same assumptions on $p$.
\end{proof}

\subsection{Energy control}\label{ssec:MEmon}
Throughout this subsection, we assume that $\g$ is $(\s,\phi)$-admissible where the function $\phi$ satisfies all the assumptions imposed in \cref{prop:geta}.

When $0<\s<\d$, then it is immediate from basic potential theory (e.g. see \cite[Remark 2.5]{RS2021}) that $\g\ast\mu \in C_0(\R^\d)$ (it is actually $C^{k,\alpha}$ for some $k\in\N_0$ and $\alpha>0$ depending on the value of $\s$) and therefore the modulated energy is well-defined. If $\s= 0$, then we need to impose a suitable decay assumption on $\mu$ to compensate for the logarithmic growth of $\g$ at infinity. {The energy condition $\int_{(\R^\d)^2}|\g|(x-y)d|\mu|^{\otimes 2}(x,y)<\infty$ from the statement of \cref{thm:FI} suffices.}

Our main proposition states that the modulated energy controls both the truncated Riesz energy of $\frac{1}{N}\sum_{i=1}^N\delta_{x_i}-\mu$ and the difference between the microscopic energies associated to $\g$ and $\g_\eta$, up to an $O(\eta^{\d-\s})$ additive error. As previously remarked, the size of this error is in general optimal. The reader should compare this proposition to \cite[Proposition 2.1]{NRS2021} in the sub-Coulomb case and \cite[Proposition 2.1]{RS2022} in the Coulomb/super-Coulomb case.

\begin{prop}\label{prop:MEmon}
Let $\mu \in L^1(\R^\d)\cap L^p(\R^\d)$ for $\frac{\d}{\d-\s}<p\le \infty$ with $\int_{\R^\d}\mu=1$. Suppose further that $\int_{(\R^\d)^2}\g(x-y)d\mu^{\otimes 2}(x,y)<\infty$ and $\int_{(\R^\d)^2}|\g|(x-y)d|\mu|^{\otimes 2}(x,y)<\infty$ if $\s=0$.  Let $\ux_N \in (\R^\d)^N$ be a pairwise distinct configuration. Then
\begin{multline}\label{eq:MEmon}
\frac{1}{2N^2}\sum_{1\le i\ne j\le N}\f_\eta(x_i-x_j) + \frac12\int_{(\R^\d)^2} \f_\eta(x-y)d\mu^{\otimes2}(x,y)+\frac12\int_{(\R^\d)^2}\g_\eta(x-y)d\Big(\frac1N\sum_{i=1}^N\delta_{x_i}-\mu\Big)^{\otimes 2} \\
\leq \Fr_N(X_N,\mu) + \frac{(\g_{\Rs}(\eta) + C+ C_{\rho})}{2N}\indic_{\s=0} + C C_\zeta\frac{\g_{\Rs}(\eta)}{2N}\indic_{\s>0} +CC_\zeta\|\mu\|_{L^p}\eta^{\frac{\d(p-1)}{p}-\s},  
\end{multline}
where $C = C(\d,\s,p,\phi)>0$  and $C_{\rho}\ge 0$, where $C_{\rho} = 0$ when $\rho=0$.
\end{prop}
\begin{proof}
Unpacking the definition \eqref{eq:modenergy} of $\Fr_N(\ux_N,\mu)$, inserting the identity $\g=\f_\eta + \g_\eta$, and then expanding, we find that
\begin{multline}
\Fr_N(X_N,\mu) =  \frac{1}{2N^2}\sum_{1\le i\ne j\le N} \f_{\eta}(x_i-x_j)  + \frac12\int_{(\R^\d)^2}\f_{\eta}(x-y)d\mu^{\otimes2}(x,y) \\
 + \frac12\int_{(\R^\d)^2}\g_\eta(x-y)d\Big(\frac1N\sum_{i=1}^N \delta_{x_i}-\mu\Big)^{\otimes 2}(x,y) -  \frac{1}{2N} \g_{\eta}(0) {-}\frac{1}{N}\sum_{i=1}^N\int_{\R^\d} \f_{\eta}(x-x_i)d\mu(x).
\end{multline}
Applying the estimate \cref{lem:feta}\eqref{eq:fetaHol} to the ultimate term and \cref{prop:geta}(\ref{item:'geta0'})  to the penultimate term on preceding the right-hand side, the desired conclusion follows. 
\end{proof}


\begin{remark}
At the risk of being obvious, let us note that the three terms on the left-hand side of the inequality \eqref{eq:MEmon} are nonnegative, since $\f_\eta\ge 0$ and $\g_\eta$ is positive definite when integrated against zero mean test functions.
\end{remark}

\begin{remark}\label{rem:MElb}
In particular, for the exact Riesz case, choosing $\eta = \la= (N\|\mu\|_{L^p})^{-\frac{p}{\d(p-1)}}$, we find the lower bound
\begin{align}
\Fr_N(X_N,\mu) -\frac{\log\la}{2N}\indic_{\s=0} + C\|\mu\|_{L^p}\la^{\d-\s} \ge 0,
\end{align}
where $C = C(\d,\s,p)>0$, showing the almost positivity of the modulated energy with the optimal size of additive error when $p=\infty$.\footnote{The value of the sharp constant $C$ is still an open question.}
\end{remark}

As a corollary of \cref{prop:MEmon}, we control the small-scale interaction (cf. \cite[Corollary 3.4]{Serfaty2023}, \cite[Proposition 2.3]{RS2022}). 

\begin{cor}\label{cor:MEcount}
Letting $\la \coloneqq (N\|\mu\|_{L^p})^{-\frac{p}{\d(p-1)}}$, introduce the nearest-neighbor type distance
\begin{align}\label{eq:rsi_def}
\rs_i \coloneqq \frac14\min\Big(\min_{1\le j\le N: j\ne i} |x_i-x_j|, \la\Big).
\end{align}
Under the same assumptions as \cref{prop:MEmon}, there exists $C= C(\d,\s,p,\phi)>0$, such that for every $\eta\le\la$,
\begin{multline}\label{eq:MEcount2}
	\Fr_N(X_N,\mu) + {\frac{(\g_{\Rs}(\eta) + C + C_{\rho})}{2N}\indic_{\s=0} + C C_{\zeta}\frac{\g_{\Rs}(\eta)}{2N}\indic_{\s>0}} +CC_\zeta\|\mu\|_{L^p}\eta^{\frac{\d(p-1)}{p}-\s} \\
	\ge \begin{cases}\displaystyle \frac{1}{2N^2}\sum_{i=1}^N \g(4\rs_i), & {\s >0} \\ \displaystyle \frac{1}{2N^2}\sum_{i=1}^N \g(4\rs_i/\eta), & {\s=0},\end{cases}
\end{multline}
and
\begin{multline}\label{eq:MEcount1}
	\Fr_N(X_N,\mu) + {\frac{(\g_{\Rs}(\eta) + C + C_{\rho})}{2N}\indic_{\s=0} + CC_{\zeta}\frac{\g_{\Rs}(\eta)}{2N}\indic_{\s>0}} +CC_\zeta\|\mu\|_{L^p}\eta^{\frac{\d(p-1)}{p}-\s} \\
	\ge \begin{cases}\displaystyle\frac{1}{2C C_\zeta N^2}\sum_{\substack{1\le i\ne j \le N \\ |x_i-x_j|\le \eta}} |x_i-x_j|^{-\s}, & {\s> 0} \\ \displaystyle\frac{1}{2CN^2}\sum_{\substack{1\le i\ne j \le N \\ |x_i-x_j|\le \eta}} -\log\left(\frac{x_i-x_j}{\eta}\right) , & {\s=0} \end{cases}
\end{multline}
where $C_{\rho}$ is as above.
\end{cor}
\begin{proof}
Let $\eta \le \lambda$. Discarding the (nonnegative) second and third terms on the left-hand side of \eqref{eq:MEmon} and unpacking the definition of $\f_\eta$, we find that
\begin{multline}\label{eq:MEss0}
\frac{1}{2N^2}\sum_{i=1}^N \sum_{1\le j\le N: |x_i-x_j| \le \eta} (\g(x_i-x_j) - \g_\eta(x_i-x_j)) \le \frac{1}{2N^2}\sum_{1\le i\ne j\le N} \f_\eta(x_i-x_j) \\
\le \Fr_N(X_N,\mu) + {\frac{(\g_{\Rs}(\eta) + C + C_{\rho})}{2N}\indic_{\s=0} + CC_{\zeta}\frac{\g_{\Rs}(\eta)}{2N}\indic_{\s>0}} +CC_\zeta\|\mu\|_{L^p}\eta^{\frac{\d(p-1)}{p}-\s}.
\end{multline}
For each $1\le i\le N$, either $\rs_i = \frac14\min_{j\ne i}|x_i-x_j|$, in which case there exists $j\ne i$ such that $\rs_i = \frac14|x_i-x_j|$, or $\rs_i = \frac14\lambda$. In the former case, for such $j$, we have
\begin{align}
\g(x_i-x_j) - \g_\eta(x_i-x_j) \geq \g(4r_i)-\g_\eta(0),
\end{align}
since $\g_\eta$ is decreasing. In the latter case, for every $j\ne i$,
\begin{align}
\g(x_i-x_j) - \g_\eta(x_i-x_j) \ge 0 \ge  \g(\la) - \g(\eta) = \g(4\rs_i)-\g(\eta).
\end{align}
provided that $\eta\le \la$.  Thus, in all cases,
\begin{align}\label{eq:MEss1}
\sum_{1\le j\le N: |x_i-x_j| \le \eta}  (\g(x_i-x_j) - \g_\eta(x_i-x_j)) \ge \g(4\rs_i)-\max(\g(\eta),\g_\eta(0)).
\end{align} 
By \cref{prop:geta}(\ref{item:'geta0'}) and \cref{prop:geta}(\ref{item:'geta0}), when $\s>0$ we have that
\begin{align}\label{eq:MEss2}
|\max(\g(\eta),\g_\eta(0))| \le C_\zeta C \g_{\Rs}(\eta).
\end{align}
On the other hand, when $\s=0$, by \cref{prop:geta}(\ref{item:'geta0'}) and the definition of $\g$ being $(0,\phi)$-admissible,
\begin{align}\label{eq:MEss3}
	\g(4\rs_i)-\max(\g(\eta),\g_\eta(0))\geq \g(4\rs_i/\eta)- (C+C_\rho),
\end{align}
where $C_\rho=0$ if $\rho=0$.  Applying the estimates \eqref{eq:MEss1}, \eqref{eq:MEss2}, and \eqref{eq:MEss3} to our starting relation \eqref{eq:MEss0}, we arrive at \eqref{eq:MEcount2} after rearranging.

On the other hand, by \cref{prop:geta}(\ref{item:'geta1'}), $ \f_\eta(x) \ge C C_\zeta^{-1}|x|^{-\s}$ if $\s>0$ and $\f_\eta(x)\geq -C\log(|x|/\eta)$ if $\s=0$ for all $|x|\leq \eta$. Applying this inequality to \eqref{eq:MEss0} above, we obtain \eqref{eq:MEcount1}.
\end{proof}

By following the proof of \cite[Proposition 2.5]{RS2022}, one may also bound the number of points whose nearest-neighbor distance is at mesoscales. As we do not need such an estimate, we leave the details as an exercise for the reader.


\subsection{Coercivity}\label{ssec:MEcoer}
We show in this subsection that the modulated energy associated to an $(\s,\as)$-admissible potential is coercive, provided $\as$ is sufficiently large, in the sense that it controls a squared negative-order Sobolev norm. In particular, this yields that vanishing of the modulated energy implies weak convergence of the empirical measure to the target measure $\mu$. Such a coercivity estimate is already known in the sub-Coulomb case \cite[Proposition 2.4]{NRS2021} (see also \cite[Corollary 4.22]{SerfatyLN} for a coercivity estimate in the Coulomb/super-Coulomb case), but its $N$-dependent additive error is much larger than that shown in the lemma below. Moreover, the estimate presented below is sharper compared to the previous works in that the control is $H^{-\frac{\d}{2}-\varepsilon}$, for any $\varepsilon>0$, which is the best one can hope for given the Dirac mass is not in $H^{-\frac{\d}{2}}$. 

\begin{prop}\label{prop:coer}
Let $\g$ be $(\s,\phi)$-admissible, such that $\hat\phi(\xi)\ge C_{\phi}\jp{\xi}^{-r}$ for some $r>\d$. Suppose $\frac{\d}{\d-\s}<p\le \infty$. There exists a constant $C=C(\d,\s,r,p,\phi)>0$ such that the following holds: for any $\mu\in L^1(\R^\d)\cap L^p(\R^\d)$ with $\int_{\R^\d}d\mu = 1$ and $\int_{(\R^\d)^2}|\g|(x-y)d|\mu|^{\otimes 2}(x,y)<\infty$ if $\s=0$, $\la \le 1$, and any pairwise distinct configuration $X_N\in (\R^\d)^N$, it holds that
\begin{multline}\label{eq:coer}
\left\|\frac1N\sum_{i=1}^N \delta_{x_i}-\mu\right\|_{H^{-r/2}}^2  \le C C_\zeta\Big(\Fr_N(\ux_N,\mu) +  \frac{(\g_{\Rs}(\la) +C+ C_{\rho})}{2N}\indic_{\s=0} \\
+ C C_\zeta\frac{\g_{\Rs}(\la)}{2N}\indic_{\s>0} +C C_\zeta\|\mu\|_{L^p}\la^{\frac{\d(p-1)}{p}-\s}\Big),
\end{multline}
where $C_{\rho}$ is as in \cref{prop:MEmon}. In particular, if $\g$ is $(\s,\as)$-admissible for $\as>\d$, then \eqref{eq:coer} holds with $r=\as$; and if $\g$ is $\s$-Riesz, then \eqref{eq:coer} holds for any $r>\d$. 
\end{prop}
\begin{proof}
Recalling the identity \eqref{eq:getahat}, which holds with $t^{\d-\s}$ replaced by $\zeta(t)$ in the general case, and using that the Fourier transform of $f\coloneqq \frac1N\sum_{i=1}^N\delta_{x_i} - \mu$ vanishes at the origin, we have that
\begin{align}\label{eq:coer0}
\int_{(\R^\d)^2}\g_\eta(x-y)d\Big(\frac1N\sum_{i=1}^N\delta_{x_i}-\mu\Big)^{\otimes2}(x,y) =  \mathsf{c}_{\phi,\d,\s}\int_\eta^\infty \zeta(t)\int_{\R^\d}\hat\phi(t\xi) |\hat{f}(\xi)|^2 d\xi \frac{dt}{t},
\end{align}
where we have also used Fubini-Tonelli to interchange the order of integration. For $\eta \le t\le 1$, we bound $\hat\phi(t\xi) \ge C_\phi \jp{\xi}^{-r}$; and for $t\ge 1$, we bound $\hat\phi(t\xi)\ge C_\phi t^{-r}\jp{\xi}^{-r}$. Applying these lower bounds to the right-hand side of \eqref{eq:coer0}, we obtain
\begin{align}
&\frac{1}{\mathsf{c}_{\phi,\d,\s} C_{\phi}}\int_{(\R^\d)^2}\g_\eta(x-y)d\Big(\frac1N\sum_{i=1}^N\delta_{x_i}-\mu\Big)^{\otimes2}(x,y)\nn\\
 &\ge \int_{\eta}^1 \zeta(t)\int_{\R^\d}\jp{\xi}^{-r}|\hat{f}(\xi)|^2 d\xi \frac{dt}{t}  + \int_1^\infty \zeta(t)t^{-r}\int_{\R^\d}\jp{\xi}^{-r}|\hat{f}(\xi)|^2 d\xi \frac{dt}{t} \nn\\
&=\Bigg(\int_\eta^1\zeta(t)\frac{dt}{t} + \int_1^{\infty}\zeta(t)t^{-r}\frac{dt}{t}\Bigg)\| f\|_{H^{-\frac{r}{2}}}^2, \label{eq:coer1}
\end{align}
which requires that 
$r>\d$ for the integral in $\xi$ to be finite ($r>\d$ is also necessary for $\phi$ to be bounded, continuous, as assumed in \cref{prop:MEmon}). Bounding the left-hand side of \eqref{eq:coer1} from above using the estimate \eqref{eq:MEmon} from \cref{prop:MEmon} with the choice $\eta = \la$, the desired conclusion follows after a little bookkeeping.
\end{proof}

\begin{remark}
If $\g$ is $(\s,\phi)$-admissible where we simply assume that $\hat\phi>0$, then one can still prove that $\Fr_N(\ux_N,\mu)\rightarrow 0$ implies that $\frac1N\sum_{i=1}^N \delta_{x_i} \rightharpoonup \mu$ as $N\rightarrow\infty$. We leave the proof of this assertion as an exercise for the reader.
\end{remark}

\section{Kato-Ponce commutator estimates}\label{sec:KP}
In this section, we prove some commutator estimates for (fractional) powers of the inhomogeneous Fourier multiplier $\jp{\nabla} = (I-\Delta)^{1/2}$. These types of estimates are usually referred to as Kato-Ponce commutator estimates, originating in \cite{KP1988}. See also \cite{Taylor1991, KPV1993, Taylor2003, GO2014, Li2019, LS2020} for further developments and alternative proofs. We refer to \cite{Li2019,LS2020} for a more thorough discussion of the history. Later in \cref{sec:rcom}, we will combine these commutator estimates with the potential truncation procedure and the modulated energy bounds of \cref{sec:ME} to prove \cref{thm:FI'}.

The main results of this section are the following proposition and corollary. 

\begin{prop}\label{prop:comm}
Let $\al\ge 0$ and write $\al=2m+r$ for integer $m\ge 0$ and $r\in (0,2]$. There exists a constant $C = C(\d,\alpha)>0$ such that for any $v,f\in\Sc(\R^\d)$,
\begin{align}\label{eq:propcomm}
\bigg| \int_{\R^\d} v\cdot \nabla f \jp{\nab}^\al f\bigg|\leq CA_{v,\al}\|f\|_{{H}^{\frac{\alpha}{2}}}^2,
\end{align}
where
\begin{multline}\label{eq:Avaldef}
A_{v,\al} \coloneqq \|\nab v\|_{L^\infty} + \sum_{j=0}^{m-1}\Big(\|\Dm^{j}\nab v\|_{L^{\max(\frac{\d}{j},2)}}\indic_{j\ne \frac{\d}{2}} +\|\Dm^{j}\nab v\|_{L^{2+}}\indic_{j= \frac{\d}{2}} \\
+ \|\Dm^{\frac{r}{2}+j} \nab v\|_{L^{\max(\frac{2\d}{2j+r},2)}}\indic_{j+\frac{r}{2}\ne \frac{\d}{2}}    + \|\Dm^{\frac{r}{2}+j} \nab v\|_{L^{2+}}\indic_{j+\frac{r}{2}= \frac{\d}{2}}\Big).
\end{multline}
Here, the summation is understood as vacuous if $m<1$.

\end{prop}


At first glance, the commutator structure in \cref{prop:comm} may not be apparent. However, writing $\jp{\nab}^{\al} = \jp{\nab}^{\al/2}\jp{\nab}^{\al/2}$ and integrating by parts, the left-hand side of \eqref{eq:propcomm} equals
\begin{align}
\int_{\R^\d}\comm{\jp{\nab}^{\al/2}}{v\cdot}(\nab f)\jp{\nab}^{\al/2}f + \int_{\R^\d}v\cdot\nab\jp{\nab}^{\al/2}f \jp{\nab}^{\al/2}f.
\end{align}
The last term is controlled by $\frac12\|\nab v\|_{L^\infty} \|\jp{\nab}^{\al/2}f\|_{L^2}^2$ after integrating by parts once more, while the first term has an obvious commutator structure. 

As a corollary of \cref{prop:comm}, we obtain the following first-order functional inequality for the truncated potential $\g_\eta$. We deduce \cref{cor:scaledcomm} from \cref{prop:comm} through the identity \eqref{eq:getadef}, which exists by our assumption that $\g$ is $(\s,\as)$-admissible, and averaging over the estimates, noting that $\jp{\nab}^{\as}G_\as = \delta_0$.

\begin{cor}\label{cor:scaledcomm}
Suppose that $\g$ is $(\s,\as)$-admissible for some $\as < \d+2$. There exists $C=C(\d,\s,\as)>0$ such that for every $\eta\ge 0$, $v,f\in\Sc(\R^\d)$ with $f$ zero-mean when $\s=0$,
\begin{multline}\label{eq:scaledcomm}
\bigg|\int_{(\R^\d)^2}(v(x)-v(y))\cdot\nabla\g_\eta(x-y)f(x)f(y)\bigg|\\
\leq C\Big(\|\nab v\|_{L^\infty}  + \|\Dm^{\frac{\as}{2}}v\|_{L^{\frac{2\d}{\as-2}}}{\indic_{\as>2}}\Big)\int_{(\R^\d)^2}\g_\eta(x-y)f(x)f(y).
\end{multline}
\end{cor}

\begin{remark}\label{rem:commrho}
At the risk of stating the obvious, let us note that since both sides of the estimate \eqref{eq:scaledcomm} are linear in the potential $\g$, the triangle inequality implies that the estimate holds for any linear combination of Riesz-type potentials. In particular, recalling our structural assumption in the $\s=0$ case of \cref{def:Rtype}, this implies that \cref{cor:scaledcomm} is valid for any potential of the form
\begin{align}\label{eq:grhoform}
\g(x) = \int_0^\infty \rho(t)t^{-\d}G_{\as}\left(\frac{x}{t}\right)\frac{dt}{t}, \qquad \text{where} \ \int_0^\infty \rho(t) t^{-\d}\frac{dt}{t}<\infty \ \text{and} \ \as <\d+2.
\end{align}
Importantly, this allows for \emph{nonsingular} interactions. 
\end{remark}

\begin{proof}
We first handle the case that $\s>0$. By the assumption that $\g$ is $(\s,\as)$-admissible and Fubini-Tonelli, we have that
\begin{multline}\label{eq:scaledcomm0}
\int_{(\R^\d)^2}(v(x)-v(y))\cdot\nabla\g_\eta(x-y)f(x)f(y)dxdy \\
=\mathsf{c}_{\as,\d,\s}\int_\eta^\infty \frac{\zeta(t)}{t^{\d+1}}\int_{(\R^\d)^2} (v(x)-v(y))\cdot \nabla G_\as\left(\frac{x-y}{t}\right)f(x)f(y)dxdy \frac{dt}{t}.
\end{multline}
Making the change of variable $tz=x$ and $tw=y$ and letting $v_t(x) \coloneqq v(tx)$ and $f_t(x) \coloneqq f(tx)$, we find that
\begin{align}
\int_{(\R^\d)^2} (v(x)-v(y))\cdot\nabla G_\as\left(\frac{x-y}{t}\right)f(x)f(y)&=t^{2\d}\int_{(\R^\d)^2} (v_t(z)-v_t(w))\cdot \nabla G_\as(w-z)f_t(z)f_t(w) \nn\\
&=2t^{2\d}\int_{\R^\d} v_t\cdot\nabla(G_\as\ast f_t)f_t, \label{eq:scaledcomm1}
\end{align}
where the final line follows from desymmetrizing. Letting $h_t \coloneqq G_\as\ast f_t$, noting that $\jp{\nab}^{\as}h_t=f_t$, and applying \cref{prop:comm} with $f$ replaced by $h_t$,  we obtain
\begin{align}\label{eq:scaledcomm20}
\bigg|\int_{\R^\d} v_t\cdot\nabla h_t f_t \bigg|&\leq C A_{v_t, \as} \|h_t\|_{H^{\frac{s}{2}}}^2.
\end{align}
Unpacking the definition of $h_t$ and undoing the change of variables made above,
\begin{align}\label{eq:scaledcomm21}
\|h_t\|_{H^{\frac{\as}{2}}}^2 = \int_{(\R^\d)^2}G_\as(x-y)f_t(x)f_t(y)  = t^{-2\d}\int_{(\R^\d)^2}G_\as\left(\frac{x-y}{t}\right)f(x)f(y).
\end{align}
Write $\as = 2m+r$ for integer $m\ge 0$ and $r\in (0,2]$. Unpacking the definition \eqref{eq:Avaldef} of $A_{v_t, \as}$,
\begin{align}
A_{v_t, \as} &= \|\nab v_t\|_{L^\infty} + \sum_{0\le j \le m-1}\Big(\|\Dm^{j}\nab v_t\|_{L^{\max(\frac{\d}{j},2)}}\indic_{j\ne \frac{\d}{2}} + \|\Dm^{\frac{r}{2}+j} \nab v_t\|_{L^{\max(\frac{2\d}{2j+r},2)}}\indic_{j+\frac{r}{2}\ne \frac{\d}{2}}  \nn\\
&\ph+  \|\Dm^{j}\nab v_t\|_{L^{2+}}\indic_{j= \frac{\d}{2}} + \|\Dm^{\frac{r}{2}+j} \nab v_t\|_{L^{2+}}\indic_{j+\frac{r}{2}= \frac{\d}{2}}\Big) \nn\\
&=t\|\nabla v\|_{L^\infty}+\sum_{0\le j \le m-1}\Big(t^{j+1 - \frac{\d}{\max(\d/j,2)}}\|\Dm^{j}\nab v\|_{L^{\max(\frac{\d}{j},2)}}\indic_{j\ne \frac{\d}{2}} \nn \\
&\ph+ t^{\frac{r}{2}+j+1-\frac{\d}{\max(2\d/(2j+r),2)}}\|\Dm^{\frac{r}{2}+j} \nab v\|_{L^{\max(\frac{2\d}{2j+r},2)}}\indic_{j+\frac{r}{2}\ne \frac{\d}{2}}\nn \\
&\ph+  {t^{\frac{\d}{2}+1-\frac{\d}{2+}}}\|\Dm^{j}\nab {v}\|_{L^{2+}}\indic_{j= \frac{\d}{2}} + {t^{\frac{\d}{2}+1-\frac{\d}{2+}}}\|\Dm^{\frac{r}{2}+j} \nab {v}\|_{L^{2+}}\indic_{j+\frac{r}{2}= \frac{\d}{2}}\Big).
\end{align}
The preceding expression only scales like $t$ if $m-1+\frac{r}{2}\ge \frac{\d}{2}$. This motivates our assumption that $\as<\d+2$, which is equivalent to $m-1+\frac{r}{2} < \frac{\d}{2}$. Hence,
\begin{align}\label{eq:scaledcomm22}
A_{v_t,\as} \le t\Big(\|\nab v\|_{L^\infty} + \sum_{j=0}^{m-1} \|\Dm^j\nab v\|_{L^{\frac{\d}{j}}} + \|\Dm^{\frac{r}{2}+j}\nab v\|_{L^{\frac{2\d}{2j+r}}}\Big) \le Ct\Big(\|\nab v\|_{L^\infty}  + \|\Dm^{m+\frac{r}{2}}v\|_{L^{\frac{2\d}{2(m-1)+r}}}\Big),
\end{align}
where the final inequality is by Sobolev embedding. From the relations \eqref{eq:scaledcomm20}, \eqref{eq:scaledcomm21}, \eqref{eq:scaledcomm22}, we obtain
\begin{multline}\label{eq:scaledcomm2}
\bigg|\int_{\R^\d} v_t\cdot\nabla h_t f_t \bigg| \le Ct^{-2\d-1}\Big(\|\nab v\|_{L^\infty} + \|\Dm^{m+\frac{r}{2}} v\|_{L^{\frac{2\d}{2(m-1)+r}}}\indic_{2m+r>2} \Big)\\
\times\int_{(\R^\d)^2}G_\as\left(\frac{x-y}{t}\right)f(x)f(y).
\end{multline}
Note that $m+\frac{r}{2}=\frac{\as}{2}$ and $\frac{2\d}{2(m-1)+r} = \frac{2\d}{\as-2}$. 

Combining the relations \eqref{eq:scaledcomm3}, \eqref{eq:scaledcomm1}, \eqref{eq:scaledcomm2}, we arrive at
\begin{align}
&\bigg|\int_{(\R^\d)^2} (v(x)-v(y))\cdot\nabla\g_\eta(x-y)f(x)f(y)\bigg| \nn\\
&\le C\int_\eta^\infty \frac{\zeta(t)}{t^{\d+1}} t^{2\d}\int_{(\R^\d)^2} t^{-2\d+1}\Big(\|\nab v\|_{L^\infty}  + \|\Dm^{\frac{\as}{2}}v\|_{L^{\frac{2\d}{\as-2}}}\indic_{\as>2}\Big)G_\as\left(\frac{x-y}{t}\right)f(x)f(y)dxdy \frac{dt}{t}\nn\\
&=C\Big(\|\nab v\|_{L^\infty}  + \|\Dm^{\frac{\as}{2}}v\|_{L^{\frac{2\d}{\as-2}}}\indic_{\as>2}\Big)\int_{(\R^\d)^2}\g_\eta(x-y)f(x)f(y),
\end{align}
which is exactly the desired conclusion.

When $\s=0$, using that $\g$ is $(0,\as)$-admissible we note that
\begin{align}
\nabla\g_\eta(x)= \mathsf{c}_{\as,\d,0}\int_{\eta}^\infty  t^{-1}\nabla G_{\as}\Big(\frac{x}{t}\Big)\frac{dt}{t}+ \mathsf{c}_{\as,\d,0} \int_{0}^\infty \rho(t) t^{-\d-1}\nabla G_{\as}\Big(\frac{x}{t}\Big)\frac{dt}{t},
\end{align}
by dominated convergence. Note that in contrast to the expression \eqref{eq:getadef'} for $\g_\eta$, the integral in $t$ for $\nab\g_\eta$ converges without renormalization. By Fubini-Tonelli we thus have that
\begin{multline}\label{eq:scaledcomm3}
\int_{(\R^\d)^2}(v(x)-v(y))\cdot\nabla\g_\eta(x-y)f(x)f(y)dxdy \\
 =\mathsf{c}_{\as,\d,0}{\lim_{T\rightarrow\infty}\int_\eta^T }t^{-1}\int_{(\R^\d)^2} (v(x)-v(y))\cdot \nabla G_\as\left(\frac{x-y}{t}\right)f(x)f(y)dxdy \frac{dt}{t} \\
	+\mathsf{c}_{\as,\d,0}\int_0^\infty  \rho(t) t^{-\d-1}\int_{(\R^\d)^2} (v(x)-v(y))\cdot \nabla G_\as\left(\frac{x-y}{t}\right)f(x)f(y)dxdy \frac{dt}{t}.
\end{multline}
Proceeding identically as in the case that $\s>0$, we have
\begin{multline}\label{eq:scaledcomm4}
	\bigg|\int_{(\R^\d)^2} (v(x)-v(y))\cdot\nabla\g_\eta(x-y)f(x)f(y)\bigg|\le C\Big(\|\nab v\|_{L^\infty}  + \|\Dm^{\frac{\as}{2}}v\|_{L^{\frac{2\d}{\as-2}}}\Big)\times \\
	\times{\lim_{T\rightarrow\infty}}\bigg({\int_\eta^T}\int_{(\R^\d)^2} G_\as\left(\frac{x-y}{t}\right)f(x)f(y)dxdy \frac{dt}{t}\\
	+ \int_0^\infty\rho(t)t^{-\d}\int_{(\R^\d)^2} G_\as\left(\frac{x-y}{t}\right)f(x)f(y)dxdy \frac{dt}{t}\bigg),
\end{multline}
Since $f$ has zero mean, for any $T>0$, we may smuggle in the renormalization constant $C_{\phi,T}$ from \eqref{eq:CphiT_def}, writing
\begin{align}
 \int_\eta^T \int_{(\R^\d)^2} G_\as\left(\frac{x-y}{t}\right)f(x)f(y)dxdy\frac{dt}{t}&=  \int_{(\R^\d)^2}\int_\eta^T  G_\as\left(\frac{x-y}{t} \right) \frac{dt}{t}f(x)f(y)dxdy\nn\\
 &=\int_{(\R^\d)^2}\bigg(\int_\eta^T  G_\as\left(\frac{x-y}{t} \right) \frac{dt}{t}-C_{G_\as,T}\bigg)f(x)f(y)dxdy,
\end{align}
where the first equality is again by Fubini-Tonelli, which is justified since the integration is only over $[\eta,T]$. We recall from \cref{ssec:MEpt} that $ \int_\eta^T G_{\as}(\frac{\cdot}{t})\frac{dt}{t}-C_{G_\as,T}$ converges to $\mathsf{c}_{\as,\d,0}^{-1}\g_{\Rs,\eta}$ as $T\rightarrow\infty$ in the sense of tempered distributions. Since $f$ is Schwartz, it follows from the continuity of distributional convolutions that
\begin{align}
\mathsf{c}_{\as,\d,0}\lim_{T\rightarrow\infty} \int_\eta^T \int_{(\R^\d)^2} G_\as\left(\frac{x-y}{t}\right)f(x)f(y)dxdy\frac{dt}{t} = \int_{(\R^\d)^2}\g_{\Rs,\eta}(x-y)f(x)f(y)dxdy.
\end{align}
Applying this identity to \eqref{eq:scaledcomm4} and combining terms then completes the proof.

\end{proof}

\begin{remark}\label{rem:commdens}
By standard density arguments, \cref{prop:comm}   extends to hold for $v,f$ such that $A_{\nab v,m,r}<\infty$ and $f\in H^{\frac{\al}{2}}$. Similarly, \cref{cor:scaledcomm} extends to hold for $v$ such that $\|\nab v\|_{L^\infty} + \|\Dm^{\frac{\as}{2}}v\|_{L^{\frac{2\d}{\as-2}}} < \infty$ and any distribution $f$ such that the pairing $\ipp{\g_\eta(x-y), f^{\otimes 2}} < \infty$. 
\end{remark}

To prove \cref{prop:comm}, we need a few preliminary lemmas. The first is a technical result, which may be viewed as an inhomogeneous fractional Leibniz rule.

\begin{lemma}\label{lem:KP}
For any $0<r\leq 2$ and multi-index $\vec\al \in \N_0^\d$, there exists $C=C(\d,r,|\vec\al|)>0$ such that for $f,g\in \Sc(\R^\d)$, 
\begin{align}
\|\jp{\nab}^{r/2}\p_{\vec\al}(fg)\|_{L^2}\leq C\|g\|_{H^{|\vec\al|+\frac{r}{2}}} \tilde{A}_{f,|\vec\al|,r},
\end{align}
where $\tilde{A}_{f,|\vec\al|,r}$ is defined by
\begin{multline}\label{eq:tlAfalrdef}
\tilde{A}_{f,|\vec\al|,r} \coloneqq \sum_{0\le j \le |\vec\al|}\Big(\|\Dm^{j}f\|_{L^{\max(\frac{\d}{j},2)}}\indic_{j\ne \frac{\d}{2}} +\|\Dm^{j}f\|_{L^{2+}}\indic_{j= \frac{\d}{2}}\Big) \\
+  \sum_{0\le j\le |\vec\al|}\Big( \|\Dm^{\frac{r}{2}+j} f\|_{L^{\max(\frac{2\d}{2j+r},2)}}\indic_{j+\frac{r}{2}\ne \frac{\d}{2}}    + \|\Dm^{\frac{r}{2}+j} f\|_{L^{2+}}\indic_{j+\frac{r}{2}= \frac{\d}{2}}\Big).
\end{multline}
In particular, if $|\vec\al|+\frac{r}{2}<\frac{\d}{2}$, then $\tl{A}_{f,|\vec\al|,r} \leq C(\|f\|_{L^\infty} + \|\Dm^{|\vec\al|+\frac{r}{2}}f\|_{L^{\frac{\d}{|\vec\al|+\frac{r}{2}}}})$. 
\end{lemma}
\begin{proof}
Set $m\coloneqq |\vec\al|$. Applying the Leibniz rule to $\p_{\vec\al}$ and using the triangle inequality, we find that
\begin{align}
\|\jp{\nab}^{r/2}\p_{\vec\al}(fg)\|_{L^2} &\leq \sum_{\vec\be \le \vec\al} {\vec\al\choose\vec\be}\|\jp{\nab}^{r/2}(\p_{\vec\be} f\p_{\vec{\al}-\vec\be}g)\|_{L^2} \\
&\leq \sum_{\vec\be \le \vec\al} {\vec\al\choose\vec\be} \|\jp{\nab}^{r/2}(\p_{\vec\be} f\p_{\vec\al-\vec\be}g)-\p_{\vec\be}f\jp{\nab}^{r/2}\p_{\vec\al-\vec\be}g\|_{L^2}+\|\p_{\vec\be}f\jp{\nab}^{r/2}\p_{\vec\al-\vec\be}g\|_{L^2}.    
\end{align}
Appealing to the Kato-Ponce estimate \cite[Theorem 1.9]{Li2019}, we have that
\begin{align}
\|\jp{\nab}^{r/2}(\p_{\vec\be} f\p_{\vec\al-\vec\be}g)-\p_{\vec\be}f\jp{\nab}^{r/2}\p_{\vec\al-\vec\be}g\|_{L^2} \leq \|\jp{\nab}^{\frac{r}{2}-1}\nabla\p_{\vec\be}f\|_{L^{p_1}}  \|\p_{\vec\al-\vec\be}g\|_{L^{p_2}}
\end{align}
for any $1<p_1,p_2\le\infty$ such that $\frac{1}{p_1}+\frac{1}{p_2}=\frac12$. Since $\frac{r}{2}-1\le 0$ and the Fourier multipliers $\frac{\nab}{\jp{\nab}}, \frac{\nab}{\Dm}$ are bounded on $L^p$ for any $1<p<\infty$ by the H\"ormander-Mikhlin theorem, we see that
\begin{align}
\|\jp{\nab}^{\frac{r}{2}-1}\nabla\p_{\vec\be}f\|_{L^{p_1}} \le C\|\Dm^{\frac{r}{2}+|\vec\be|} f\|_{L^{p_1}}
\end{align}
if $p_1<\infty$. We choose $p_1,p_2$ according to
\begin{align}
(p_1,p_2) = \begin{cases} (\frac{2\d}{2|\vec\be|+r}, \frac{2\d}{\d-2|\vec\be|-r}), & {|\vec\be|+\frac{r}{2} < \frac{\d}{2}} \\  (2+\ep, \frac{2(2+\ep)}{\ep}) , & {|\vec\be|+\frac{r}{2} = \frac{\d}{2}} \\ (2,\infty), & {|\vec\be|+\frac{r}{2} > \frac{\d}{2}},\end{cases}
\end{align}
where $\ep>0$ is arbitrary. With this choice, it follows from Sobolev embedding  applied to $ \|\p_{\vec\al-\vec\be}g\|_{L^{p_2}}$ that
\begin{align}
\|\jp{\nab}^{r/2}(\p_{\vec\be} f\p_{\vec\al-\vec\be}g)-\p_{\vec\be}f\jp{\nab}^{r/2}\p_{\vec\al-\vec\be}g\|_{L^2} \leq C\|g\|_{H^{m+\frac{r}{2}}}\begin{cases} \|\Dm^{\frac{r}{2}+|\vec\be|} f\|_{L^{\frac{2\d}{2|\vec\be|+r}}},  & {|\vec\be|+\frac{r}{2} < \frac{\d}{2}} \\    \|\Dm^{\frac{r}{2}+|\vec\be|} f\|_{L^{2+\ep}}, & {|\vec\be|+\frac{r}{2} = \frac{\d}{2}} \\ \|\Dm^{\frac{r}{2}+|\vec\be|} f\|_{L^{2}}, & {|\vec\be|+\frac{r}{2} > \frac{\d}{2}}.\end{cases}
\end{align}
On the other hand, Hölder's inequality implies that for any $1\le p_3,p_4\le \infty$ such that $\frac{1}{p_3}+\frac{1}{p_4}=\frac12$, 
\begin{align}
\|\p_{\vec\be}f\jp{\nab}^{r/2}\p_{\vec\al-\vec\be}g\|_{L^2} \leq \|\p_{\vec\be}f\|_{L^{p_3}}\|\jp{\nab}^{r/2}\p_{\vec\al-\vec\be}g\|_{L^{p_4}}.
\end{align}
We choose $(p_3,p_4)$ according to
\begin{align}
(p_3,p_4) = \begin{cases} (\frac{\d}{|\vec\be|}, \frac{2\d}{\d-2|\vec\be|}), & {|\vec\be| <\frac{\d}{2}} \\ (2+\ep', \frac{2(2+\ep')}{\ep'}) , & {|\vec\be| =\frac{\d}{2}} \\  (2,\infty), & {|\vec\be| >\frac{\d}{2}}, \end{cases}
\end{align}
where $\ep'>0$ is again arbitrary. Applying Sobolev embedding again to $\|\jp{\nab}^{r/2}\p_{\vec\al-\vec\be}g\|_{L^{p_4}}$, it follows that
\begin{align}
\|\p_{\vec\be}f\jp{\nab}^{r/2}\p_{\vec\al-\vec\be}g\|_{L^2} \leq  C\|g\|_{H^{m+\frac{r}{2}}}\begin{cases} \|\Dm^{|\vec\be|}f\|_{L^{\frac{\d}{|\vec\be|}}}, & {|\vec\be| <\frac{\d}{2}} \\ \|\Dm^{|\vec\be|}f\|_{L^{2+\ep'}} , & {|\vec\be| =\frac{\d}{2}} \\ \|\Dm^{|\vec\be|}f\|_{L^2}, & {|\vec\be| >\frac{\d}{2}}. \end{cases}
\end{align}

If $m+\frac{r}{2}<\frac{\d}{2}$, then by Gagliardo-Nirenberg interpolation \cite[Theorem 2.44]{BCD2011},
\begin{align}
\|\Dm^{|\vec\be|}f\|_{L^{\frac{\d}{|\vec\be|}}} \le C \|f\|_{L^\infty}^{1-\frac{|\vec\be|}{m+\frac{r}{2}}} \|\Dm^{m+\frac{r}{2}}f\|_{L^{\frac{2\d}{2m+r}}}^{\frac{|\vec\be|}{m+\frac{r}{2}}}, \\
\|\Dm^{|\vec\be|+\frac{r}{2}}f\|_{L^{\frac{2\d}{2|\vec\be|+r}}} \le C \|f\|_{L^\infty}^{1-\frac{2|\vec\be|+r}{2m+r}} \|\Dm^{m+\frac{r}{2}}f\|_{L^{\frac{2\d}{2m+r}}}^{\frac{2|\vec\be|+r}{2m+r}}.
\end{align}
The desired conclusion follows from Young's inequality and a little bookkeeping.



\end{proof}

The second result we need is a representation of fractional powers $\jp{\nab}^{s/2}$ for $s\in (0,2)$ as the Dirichlet-to-Neumann map of an extension problem to $\R^{\d+1}$ for a (degenerate elliptic) second-order partial differential operator, which may be viewed as an inhomogeneous analogue of the extension for the fractional Laplacian popularized by Caffarelli and Silvestre \cite{CS2007}. Such a dimension extension is a special case of a more general theory established in \cite{ST2010ext} (see also \cite{MN2024} for a substantial generalization to operators on Hilbert spaces). 

\begin{lemma}\label{lem:CSinhomext}
Let $s\in (0,2)$ and $f\in\Sc(\R^\d)$. A solution of the Dirichlet problem
\begin{align}\label{eq:CSinhomext0}
\begin{cases}
(\Delta_x-I)F + \frac{(1-s)}{z}\p_z F + \p_z^2 F =0, & {\text{in} \ \R^{\d}\times (0,\infty)} \\
F(x,z) = f(x), & {\text{on} \ \R^\d\times\{0\}}
\end{cases}
\end{align}
is given by the Poisson formula
\begin{align}\label{eq:CSinhomext1}
F(x,z) = \frac{1}{\Gamma(\frac{s}{2})}\int_0^\infty e^{-t}\Big(e^{\frac{z^2}{4t}(\Delta-I)}f\Big)(x)\frac{dt}{t^{1-\frac{s}{2}}},
\end{align}
and
\begin{align}\label{eq:CSinhomext2}
\lim_{z\rightarrow 0^+} \frac{F(x,z) - F(x,0)}{z^s} = \frac{\Gamma(-\frac{s}{2})}{2^{s}\Gamma(\frac{s}{2})}\jp{\nab}^{s}f(x) = \frac{1}{s}\lim_{z\rightarrow 0^+}z^{1-s}\p_z F(x,z),
\end{align}
where the convergence holds both pointwise and in $L^2(\R^\d)$. Consequently, extending $F$ to $\R^{\d+1}$ by even symmetry, we have
\begin{align}\label{eq:CSinhomext3}
|z|^{1-s}F -\div(|z|^{1-s}\nab F) = \mathsf{c}_{\d,s}\jp{\nab}^{s}f\delta_{\R^\d\times\{0\}}
\end{align}
in the sense of distributions in $\R^{\d+1}$ and
\begin{align}\label{eq:CSinhomext4}
\int_{\R^{\d+1}}|z|^{1-s}(F+|\nab F|^2) = \|f\|_{H^{\frac{s}{2}}}^2.
\end{align}
\end{lemma}
\begin{proof}
The assertions \eqref{eq:CSinhomext1}-\eqref{eq:CSinhomext2} are the content of \cite[Theorem 1.1]{ST2010ext}. To see the assertion \eqref{eq:CSinhomext3}, let $F$ be a solution of \eqref{eq:CSinhomext0} and let $\Phi$ be a compactly supported test function in $\R^{\d+1}$. By dominated convergence,
\begin{align}
&\int_{\R^{\d+1}}\Big(|z|^{1-s}\Phi -\div(|z|^{1-s}\nab \Phi\Big) F  \nn\\
&= \lim_{\epsilon\rightarrow 0^+}\Bigg(\int_{\{|z|\ge \epsilon\} }|z|^{1-s}\Phi F - \int_{\{|z| \ge \epsilon\}}\div(|z|^{1-s}\nab\Phi) F\Bigg). \label{eq:CSpf0}
\end{align}
Integrating by parts twice,  we see that
\begin{multline}\label{eq:CSpf1}
- \int_{\{|z|\ge\epsilon\}}\div(|z|^{1-s}\nab\Phi) F  = -\int_{\{|z|=\epsilon\}}|z|^{1-s}\p_{\nu}\Phi F  +\int_{\{|z|=\epsilon\}}|z|^{1-s}\Phi \p_{\nu}F \\
-  \int_{\{|z|\ge\epsilon\}}\Phi\div(|z|^{1-s}\nab F). 
\end{multline}
It is immediate from \eqref{eq:CSinhomext0} that
\begin{align}\label{eq:CSpf2}
0 = \int_{\{|z|\ge\epsilon\}}z^{1-s}\Phi F - \int_{\{|z|\ge\epsilon\}}\Phi\div(|z|^{1-s}\nab F).
\end{align}
Now using that $F$ is even in the $z$ coordinate,
\begin{align}
\Bigg|\int_{\{|z|=\epsilon\}}|z|^{1-s}\p_{\nu}\Phi F\Bigg|  &= \epsilon^{1-s}\Bigg|\int_{\R^\d}\Big(\p_z\Phi(\cdot,\epsilon) - \p_z\Phi(\cdot,-\epsilon)\Big) F(\cdot,\epsilon)\Bigg| \nn\\
&\le 2\epsilon^{2-s}\|\sup_{|z|\le\epsilon}\p_z^2 \Phi(\cdot,z)\|_{L^1} \|F\|_{L^\infty} ,
\end{align}
where we use the mean-value theorem to obtain the last line. That $\|F\|_{L^\infty}<\infty$ follows from \cref{rem:CSdecay} below. Since $s<2$, the preceding expression vanishes as $\epsilon\rightarrow 0^+$. 
For the remaining boundary term, observe that
\begin{align}
\int_{\{|z|=\epsilon\}}|z|^{1-s}\Phi \p_{\nu}F = -\int_{\R^\d}\epsilon^{1-s}\Big(\Phi(\cdot,\epsilon)\p_z F(\cdot,\epsilon) - \Phi(\cdot,-\epsilon)\p_zF(\cdot,-\epsilon)\Big).
\end{align}
Using that $\frac{1}{s} z^{1-s}\p_z F(\cdot, z)\rightarrow \frac{\Gamma(-\frac{s}{2})}{2^{s}\Gamma(\frac{s}{2})}\jp{\nab}^{s}f$ in $L^2(\R^\d)$ as $z\rightarrow 0^+$, it follows that the preceding right-hand side converges to
\begin{align}
2^{1-s} s \frac{\Gamma(-s/2)}{\Gamma(s/2)}\int_{\R^\d}\Phi(\cdot,0)\jp{\nab}^s f
\end{align}
as $\epsilon\rightarrow 0^+$. 
After a little bookkeeping, we conclude that
\begin{align}
\int_{\R^{\d+1}}\Big(|z|^{1-s}\Phi -\div(|z|^{1-s}\nab \Phi)\Big) F = 2^{1-s} s \frac{\Gamma(-s/2)}{\Gamma(s/2)}\int_{\R^\d}\Phi(\cdot,0)\jp{\nab}^s f.
\end{align}
As $\Phi\in C_c^\infty(\R^{\d+1})$ was arbitrary, this establishes the assertion \eqref{eq:CSinhomext3}.

To see the last assertion \eqref{eq:CSinhomext4}, we note that \eqref{eq:CSinhomext3} implies the identity
\begin{align}
\|f\|_{H^{s/2}}^2 = \int_{\R^\d}f \jp{\nab}^{s}f  = \frac{1}{\mathsf{c}_{\d,s}}\int_{\R^{\d+1}}F \Big(|z|^{1-s}F - \div(|z|^{1-s}\nab F)\Big) = \frac{1}{\mathsf{c}_{\d,s}}\int_{\R^{\d+1}}|z|^{1-s}\Big(|F|^2 + |\nab F|^2\Big),
\end{align}
where the last equality follows from an integration by parts, which is justified by \cref{rem:CSdecay} below. The proof of the lemma is now complete.
\end{proof}

The following lemma provides regularity and decay estimates for the Poisson extension $F$ in terms of the boundary data $f$. In particular, if $f\in \Sc(\R^\d)$, which we may always reduce to through density, there are no issues of regularity/decay in the computations above. The lemma is stated in greater generality than we need, but some of the estimates do not appear to be recorded in the literature and may be useful elsewhere.

\begin{lemma}\label{rem:CSdecay}
Let $f\in \Sc(\R^\d)$. For any $1\le p\le \infty$ and integer $k\ge 0$,
\begin{align}\label{eq:CSdecay0}
\|\nab_x^{\otimes k}F\|_{L^p} \le \|\nab_x^{\otimes k}f\|_{L^p}.
\end{align}
For any integer $k\ge 0$,
\begin{align}\label{eq:CSdecay1}
\forall z> 0, \qquad \|\nab_x^{\otimes k} F(\cdot,z)\|_{L_x^\infty} \le C_{\d,s,k}|z|^{-\d}\|\nab_x^{\otimes k}f\|_{L^1}.
\end{align}
For any integers $m,k\ge 0$,
\begin{align}\label{eq:CSdecay2}
\forall x\in \R^\d, \qquad \|\nab_x^{\otimes k}F(x,\cdot)\|_{L_z^\infty} \le C_{\d,s,m}\|\jp{\cdot}^m \nab^{\otimes k}f\|_{L^\infty}\jp{x}^{-m}.
\end{align}
For any integers $n \ge 1$ and $k\ge 0$,
\begin{align}\label{eq:CSdecay3}
\forall (x,z)\in \R^{\d+1}, \qquad |\p_z^n\nab_x^{\otimes k}F(x,z)| \le C_{\d,n,s}|z|^{-n}\min\Bigg(\|\jp{\cdot}^m \nab^{\otimes k}f\|_{L^\infty}\jp{x}^{-m}, |z|^{-\d}\|\nab_x^{\otimes k}f\|_{L^1} \Bigg).
\end{align}
\end{lemma}
\begin{proof}
For assertion \eqref{eq:CSdecay0}, Young's inequality implies that $\sup_{z>0}\|F(\cdot,z)\|_{L^p} \le C\|f\|_{L^p}$. Moreover, since $\nab_x$ commutes with $e^{\frac{z^2}{4t}\Delta}$, the same is true with $f$ replaced by $\nab_x^{\otimes k}f$, yielding the desired result.

For \eqref{eq:CSdecay1}, we use Young's inequality to bound $\|e^{\frac{z^2}{4t}\Delta}f\|_{L^\infty} \le C_\d(z^2/4t)^{-\d/2}\|f\|_{L^1}$, obtaining that
\begin{align}
|F(x,z)| \le \frac{C_\d}{\Gamma(\frac{s}{2})}\|f\|_{L^1}|z|^{-\d}\int_0^\infty e^{-t}t^{\d/2}\frac{dt}{t^{1-\frac{s}{2}}}.
\end{align}
As we may repeat the logic with $f$ replaced by $\nab_x^{\otimes k}f$, we conclude the desired estimate.

For \eqref{eq:CSdecay2}, we write
\begin{align}
e^{\frac{z^2}{4t}\Delta}f(x) = (\pi z^2/4t)^{-\frac{\d}{2}}\int_{\R^\d} e^{-\frac{4t}{z^2}|x-y|^2}f(y),
\end{align}
and split $\int_{\R^\d} = \int_{|y|\le \frac12|x|} + \int_{|y|>\frac12|x|}$. For the region $|y|>\frac12|x|$, we trivially bound,
\begin{align}
(\pi z^2/4t)^{-\frac{\d}{2}}\int_{|y|>\frac12|x|} e^{-\frac{4t}{z^2}|x-y|^2}|f(y)| &\le \|\jp{\cdot}^{m}f\|_{L^\infty} \jp{x/2}^{-m} (\pi z^2/4t)^{-\frac{\d}{2}}\int_{|y|>\frac12|x|}e^{-\frac{4t}{z^2}|x-y|^2} \nn\\
&\le 2^m\|\jp{\cdot}^{m}f\|_{L^\infty} \jp{x}^{-m} .
\end{align}
Hence,
\begin{align}
 \frac{1}{\Gamma(\frac{s}{2})}\int_0^\infty e^{-t - \frac{z^2}{4t}} (z^2/4t)^{-\frac{\d}{2}}\int_{|y|>\frac12|x|} e^{-\frac{4t}{z^2}|x-y|^2}|f(y)| dy \frac{dt}{t^{1-\frac{s}{2}}} \le C_{s,m}\|\jp{\cdot}^{m}f\|_{L^\infty} \jp{x}^{-m}
\end{align}
For the region $|y|\le \frac12|x|$, first note that $e^{-\frac{2t}{z^2}|x-y|^2} \le e^{-\frac{t}{2z^2}|x|^2}$. Now introduce $\delta \in (0,1)$ and consider the cases $\frac{z^2}{t} \le |x|^{2(1-\delta)}$ and $\frac{z^2}{t} > |x|^{2(1-\delta)}$. In the former, we use that $e^{-\frac{t}{2z^2}|x|^2} \le e^{-|x|^{2\delta}}$; and in the latter, we use that $e^{-\frac{z^2}{4t}} \le e^{-\frac14|x|^{2(1-\delta)}}$. All together, we find that
\begin{align}
 &\frac{1}{\Gamma(\frac{s}{2})}\int_0^\infty e^{-t - \frac{z^2}{4t}} (z^2/4t)^{-\frac{\d}{2}}\int_{|y|\le \frac12|x|} e^{-\frac{4t}{z^2}|x-y|^2}|f(y)| dy \frac{dt}{t^{1-\frac{s}{2}}} \nn\\
 &\le C_{\d,s}\|f\|_{L^\infty}\int_0^\infty e^{-\frac{t}{2z^2}|x|^2} e^{-t - \frac{z^2}{4t}} \frac{dt}{t^{1-\frac{s}{2}}} \nn\\
 &\le C_{\d,s}\|f\|_{L^\infty}\Bigg(\int_0^{\frac{z^2}{|x|^{2(1-\delta)}}}e^{-\frac14|x|^{2(1-\delta)}}e^{-\frac{t}{2z^2}|x|^2} e^{-t} \frac{dt}{t^{1-\frac{s}{2}}} + \int_{\frac{z^2}{|x|^{2(1-\delta)}}}^\infty e^{-|x|^{2\delta}} e^{-t - \frac{z^2}{4t}}\frac{dt}{t^{1-\frac{s}{2}}} \Bigg)  \nn\\
 &\le C_{\d,s}'\|f\|_{L^\infty}\Big(e^{-\frac14|x|^{2(1-\delta)}} + e^{-|x|^{2\delta}}\Big).
\end{align}
The two terms on the final line may be balanced by choosing $\delta= \frac12$. Replacing $f$ by $\nab_x^{\otimes k}f$ then yields the desired \eqref{eq:CSdecay2}.

Lastly, we turn to \eqref{eq:CSdecay3}. For the computations below, it will be convenient to introduce the subscript in $F_s$ to emphasize the dependence of \eqref{eq:CSinhomext1} on $s$. Making the change of variable $t/z^2\mapsto t$, we rewrite \eqref{eq:CSinhomext1} as
\begin{align}
F_s(x,z) = \frac{z^{s}}{\Gamma(\frac{s}{2})}\int_0^\infty e^{-z^2t}\Big(e^{\frac{1}{4t}(\Delta-I)}f\Big)(x)\frac{dt}{t^{1-\frac{s}{2}}}.
\end{align}
Hence,
\begin{align}
\p_z F_s(x,z) &= \frac{sz^{s-1}}{\Gamma(\frac{s}{2})}\int_0^\infty e^{-z^2t}\Big(e^{\frac{1}{4t}(\Delta-I)}f\Big)(x)\frac{dt}{t^{1-\frac{s}{2}}}  - \frac{2z^{s+1}}{\Gamma(\frac{s}{2})}\int_0^\infty t e^{-z^2t}\Big(e^{\frac{1}{4t}(\Delta-I)}f\Big)(x)\frac{dt}{t^{1-\frac{s}{2}}} \nn\\
&=s z^{-1}F_s(x,z)  - 2\frac{\Gamma(\frac{s+1}{2})}{\Gamma(\frac{s}{2})}F_{s+1}(x,z).
\end{align}
Iteration of this relation yields
\begin{align}\label{eq:pznFrel}
\p_z^n F_s(x,z) = \sum_{j=0}^n C_{s,n,j} z^{j-n} F_{s+j}(x,z).
\end{align}
The reader will note that our estimates \eqref{eq:CSdecay0}-\eqref{eq:CSdecay2} are equally valid for all $s>0$, not just $s\in (0,2)$. Furthermore, note that for any $\ep \in (0,1)$,
\begin{align}
|F_{s+j}(x,z)| \le C_{\ep,s}z^{-j}\int_0^\infty e^{-\ep z^2 t}\Big(e^{\frac{1}{4t}(\Delta-I)}|f|\Big)(x)\frac{dt}{t^{1-\frac{s}{2}}}  = C_{\ep,s}'z^{-j}\tl{F}_s(x,\sqrt{\ep}z),
\end{align}
where $\tl{F}_s$ denotes $F_s$ with $f$ replaced by $|f|$. Hence, we may combine these estimates with the relation \eqref{eq:pznFrel}. Replacing $f$ by $\nab_x^{\otimes k}f$ then completes the proof.

\end{proof}


We conclude this section with the proof of \cref{prop:comm}.

\begin{proof}[Proof of \cref{prop:comm}]
Write $\alpha=2m+r$ for integer $m\ge 0$ and $r \in (0,2]$. We will prove by induction that if for some $m\ge 0$, \cref{prop:comm} is true for all $\al=2m+r$ with $r\in (0,2]$, then it is true for all $\al=2(m+1)+r$ with $r\in (0,2]$. 

We begin with the base case $m=0$. Given the datum $f$, let $F$ denote the Poisson extension to $\R^{\d+1}$ given by \cref{lem:CSinhomext}.
With an abuse of notation, we let $v$ denote the trivial extension to a $(\d+1)$-vector field, i.e. the map $(x,z) \mapsto (v(x),0)$.
Inserting the identity  $F\zg -\div(\zg\nab F) = \jp{\nab}^{r}f\delta_{\R^\d\times\{0\}}$ and integrating by parts, we see that
\begin{align}
\int_{\R^\d} v\cdot\nabla f \jp{\nab}^r f&=\int_{\R^{\d+1}} v\cdot \Big(\frac12\nab(F^2)|z|^{\gamma}-\nab F\div(|z|^\gamma\nabla F)\Big)\nn\\
&=-\frac{1}{2}\int_{\R^{\d+1}} \zg \div v F^2+\frac12\int_{\R^{\d+1}}\nabla  v:\comm{\nabla F}{\nabla F},
\end{align}
where $\comm{\nab F}{\nab F}$ is the stress-energy tensor associated to $\nabla F$, that is for vector fields $u_1=(u_1^i),u_2=(u_2^j)$,
\begin{align}\label{eq:SETdef}
\comm{u_1}{u_2}_{ij} \coloneqq \zg\Big(u_1^i u_2^j+u_1^ju_2^i - (u_1\cdot u_2)\delta_{ij}\Big), \qquad i,j\in [\d].
\end{align}
Implicitly, we are using that the extended vector field $v$ has zero $\d+1$-component and that $\nab_x$ commutes with $\zg$. By Cauchy-Schwarz, the preceding right-hand is less or equal to
\begin{align}
C\|\nabla v\|_{L^\infty}\int_{\R^{\d+1}}\zg (F^2+|\nabla F|^2)=C\|\nabla v\|_{L^\infty}\| f\|_{H^{r/2}}^2.
\end{align}
where $C=C(\d)>0$ and the equality is by \cref{lem:CSinhomext}\eqref{eq:CSinhomext4}.

As our induction hypothesis, suppose that for some integer $m\ge 0$, \cref{prop:comm} holds for all $\al = 2m+r$, where $r\in (0,2]$. We will show that \cref{prop:comm} holds for all $\al = 2(m+1)+r$, where $r\in (0,2]$.

Writing $\jp{\nab}^{2(m+1)+r} = (I-\Delta)\jp{\nab}^{2m+r}$ and integrating by parts once, we find that
\begin{align}
\int_{\R^\d} v\cdot\nabla f \jp{\nab}^{2(m+1) +r} f=&\int_{\R^\d} v\cdot \nabla f \jp{\nab}^{2m+r} f+\int_{\R^\d} v\cdot \nabla \partial_j f \jp{\nab}^{2m+r}\partial_j f \nn\\
&\ph+\int_{\R^\d} \partial_j v\cdot \nabla f \jp{\nab}^{2m+r}\partial_j f,\label{eq:KPih1}
\end{align}
where we follow the convention of Einstein summation. Recall the notation $A_{v,\al}$ from \eqref{eq:Avaldef}, which we extend in the obvious manner to allow for $u$ to be tensor-valued. 
By the induction hypothesis, the magnitude of the first term on the right-hand side is $\le$ 
\begin{align}\label{eq:KPih2}
CA_{v,\al-2}\|f\|_{{H}^{m+\frac{r}{2}}}^2 \le CA_{v,\al-2}\|f\|_{{H}^{m+1+\frac{r}{2}}}^2,
\end{align}
where the second inequality follows from the trivial embedding $\|\cdot\|_{H^{s}} \le \|\cdot\|_{H^{s'}}$ for $s\le s'$.\footnote{Remark that here, the fact that the we are considering the inhomogeneous Sobolev \emph{norm} as opposed to the homogeneous \emph{seminorm} is crucial because otherwise, this embedding is false.}
The induction hypothesis (with $f$ replaced by $\p_j f$) also implies that the second term on the right-hand side of \eqref{eq:KPih1} is bounded by
\begin{align}\label{eq:KPih2'}
 CA_{v,\al-2}\sum_{j=1}^\d\|\partial_j f\|_{{H}^{m+\frac{r}{2}}}^2  \leq C A_{v,\al-2}\|f\|_{{H}^{m+1+\frac{r}{2}}}^2.
\end{align}

It remains to bound the third term on the right-hand side of \eqref{eq:KPih1}. Expanding $(I-\Delta)^m$ by binomial formula and integrating by parts $k$ times, we find
\begin{align}
\int_{\R^\d} \partial_j v\cdot \nabla f\jp{\nab}^{2m+r}\partial_j f &=\sum_{k=0}^m{m\choose k}\int_{\R^\d} \partial_j v\cdot \nabla f(-\Delta)^k\jp{\nab}^{r}\partial_j f \nn\\
&=\sum_{k=0}^m{m\choose k}\int_{\R^\d}\jp{\nab}^{r/2}\nabla^{\otimes k}(\partial_j v\cdot \nabla f) : \jp{\nab}^{r/2}\nabla^{\otimes k}\partial_jf.
\end{align}
For all $0\leq k\leq m$, the Cauchy-Schwarz inequality implies that
\begin{align}
\bigg|\int_{\R^\d} \langle\nabla\rangle^{r/2}\nabla^{\otimes k}(\partial_j v\cdot \nabla f) : \jp{\nab}^{r/2}\nabla^{\otimes k}\partial_j f\bigg| &\leq \|\jp{\nab}^{r/2}\nabla^{\otimes k}(\partial_j v\cdot \nabla f)\|_{L^2}\|\langle\nabla\rangle^{r/2}\nabla^{\otimes k}\partial_jf\|_{L^2} \nn\\
&\le  C \tl{A}_{\nab v,k,r} \|\nab f\|_{H^{k+\frac{r}{2}}} \|f\|_{H^{k+1+\frac{r}{2}}} \nn\\
&\le C \tl{A}_{\nab v,k,r} \|f\|_{H^{k+1+\frac{r}{2}}}^2, 
\end{align}
where the second line follows from applying Plancherel's theorem to the second factor and \cref{lem:KP} to the first factor on the right-hand side of the first line (recall the notation \eqref{eq:tlAfalrdef}), and the third line follows from another application of Plancherel. Majorizing each factor $\|f\|_{H^{k+1+\frac{r}{2}}} \le \|f\|_{H^{m+1+\frac{r}{2}}}$, for $0\le k\le m$, it follows now that
\begin{align}\label{eq:KPih3}
\bigg|\int_{\R^\d} \partial_j v\cdot \nabla f\jp{\nab}^{2m+r}\partial_j f\bigg| \le C \|f\|_{H^{m+1+\frac{r}{2}}}^2\sum_{k=0}^m {m\choose k}\tl{A}_{\nab v,k,r}.
\end{align}

Combining the estimates \eqref{eq:KPih1}, \eqref{eq:KPih2}, \eqref{eq:KPih2'}, \eqref{eq:KPih3}, we obtain that
\begin{align}
\le C\|f\|_{{H}^{m+1+\frac{r}{2}}}^2\Bigg(A_{v,\al-2} +\sum_{k=0}^m {m\choose k}\tl{A}_{\nab v,k,r}\Bigg).
\end{align}
Noting that $\al-2 =2m + r$ and $\tl{A}_{\cdot,k,r} \le \tl{A}_{\cdot,m,r}$ for $k\le m$ completes the proof of the induction step and therefore of the proposition.
\end{proof}

\section{Renormalized commutator estimate}\label{sec:rcom}
In this section, we give the proof of \cref{thm:FI}. As advertised in the introduction, we deduce \cref{thm:FI} from the following more general result  for $(\s,\as)$-admissible potentials. The proof proceeds by combining \cref{prop:MEmon} and \cref{cor:scaledcomm}.

\begin{thm}\label{thm:FI'}
Let $\g$ be an $(\s,\as)$-admissible potential for $\as\in (\d,\d+2)$, and let $\frac{\d}{\d-\s}<p\le \infty$. There exists a constant $C= C(\d,\s,p,\as)>0$ such that the following holds. Let $\mu \in L^1(\R^\d)\cap L^p(\R^\d)$ with $\int_{\R^\d}\mu=1$.  Let  $v:\R^\d\rightarrow\R^\d$ be a Lipschitz vector field. For any pairwise distinct configuration $\ux_N \in (\R^\d)^N$, it holds for $\la = (N\|\mu\|_{L^p})^{-\frac{p}{\d(p-1)}}$ that
\begin{multline}
\bigg|\int_{(\R^\d)^2\setminus\triangle}(v(x)-v(y))\cdot\nabla\g(x-y)d\Big(\frac1N\sum_{i=1}^N\delta_{x_i}-\mu\Big)^{\otimes 2}(x,y)\bigg| \\
\leq C{ C_\zeta^2}\Big(\|\nab v\|_{L^\infty}  + \|\Dm^{\frac{\as}{2}}v\|_{L^{\frac{2\d}{\as-2}}}{\indic_{\as>2}}\Big)\Big(\Fr_N(\ux_N,\mu) {+}  \frac{(-\log\la + C + C_\rho)}{2N}\indic_{\s=0} +C C_\zeta\|\mu\|_{L^p}\la^{\d-\s}\Big),
\end{multline}
where $C_\zeta$ is as in \cref{def:Rtype} and $C_\rho = 0$ if $\rho=0$.
\end{thm}
\begin{proof}
Let $\eta>0$, the exact value of which will be specified momentarily. Adding and subtracting $\g_\eta$ and applying the triangle inequality, we find
\begin{multline}\label{eq:rcomm0}
\bigg|\int_{(\R^\d)^2\setminus\triangle}(v(x)-v(y))\cdot\nabla\g(x-y)d\big(\frac1N\sum_{i=1}^N\delta_{x_i}-\mu\big)^{\otimes 2}(x,y)\bigg| \\
\leq \bigg|\int_{(\R^\d)^2}(v(x)-v(y))\cdot\nabla\g_\eta(x-y)d\big(\frac1N\sum_{i=1}^N\delta_{x_i}-\mu\big)^{\otimes 2}(x,y)\bigg|\\
+\bigg|\int_{(\R^\d)^2\setminus\triangle}(v(x)-v(y))\cdot\nabla\f_\eta(x-y)d\big(\frac1N\sum_{i=1}^N\delta_{x_i}-\mu\big)^{\otimes 2}\bigg|.
\end{multline}
Since $|v(x)-v(y)| |\nab\g_\eta(x-y)|$ vanishes along the diagonal by \cref{rem:nabGsvan}, we have re-inserted the diagonal in the first integral on the right-hand side above.

As $\int_{(\R^\d)^2}\g_\eta(x-y)d\Big(\frac1N\sum_{i=1}^N \delta_{x_i}-\mu\Big)^{\otimes 2}<\infty$,  \cref{rem:commdens} and \cref{cor:scaledcomm} imply that
\begin{multline}\label{eq:rcomm1}
\bigg|\int_{(\R^\d)^2}(v(x)-v(y))\cdot\nabla\g_\eta(x-y)d\big(\frac1N\sum_{i=1}^N\delta_{x_i}-\mu\big)^{\otimes 2}(x,y)\bigg|
\\
\leq C\Big(\|\nab v\|_{L^\infty}  + \|\Dm^{\frac{\as}{2}}v\|_{L^{\frac{2\d}{\as-2}}}{\indic_{\as>2}}\Big)\int_{(\R^\d)^2} \g_\eta(x-y)d\big(\frac1N\sum_{i=1}^N\delta_{x_i}-\mu\big)^{\otimes 2}.
\end{multline}
On the other hand, applying the mean-value theorem to $(v(x)-v(y))\cdot\nab\f_\eta(x-y)$,  
\begin{align}\label{eq:rcomm2}
&\bigg|\int_{(\R^\d)^2\setminus\triangle}(v(x)-v(y))\cdot\nabla\f_\eta(x-y)d\big(\frac1N\sum_{i=1}^N\delta_{x_i}-\mu\big)^{\otimes 2}\bigg|\nn\\
&\leq\|\nabla v\|_{L^\infty}\int_{(\R^\d)^2\setminus\triangle}|x-y||\nabla \f_\eta(x-y)|d\Big(\frac1N\sum_{i=1}^N \delta_{x_i}+\mu\Big)^{\otimes 2} \nn\\
&\le C \|\nabla v\|_{L^\infty}\Bigg(\frac{ C_\zeta^2}{N^2}\sum_{1\le i\ne j\le N} \f_{\eta/c}(x_i-x_j) + C_\zeta\|\mu\|_{L^{p}}\eta^{\frac{\d(p-1)}{p}-\s}\Bigg),
\end{align}
where the third line follows from \cref{prop:geta}(\ref{item:'geta3}) and  estimate \eqref{eq:nabfetaHol} of \cref{lem:feta}.

Applying the estimates \eqref{eq:rcomm1}, \eqref{eq:rcomm2} to the right-hand side of \eqref{eq:rcomm0}, we arrive at
\begin{multline}
\bigg|\int_{(\R^\d)^2}(v(x)-v(y))\cdot\nabla\g(x-y)d\big(\frac1N\sum_{i=1}^N\delta_{x_i}-\mu\big)^{\otimes 2}(x,y)\bigg| \\
\leq C\Big(\|\nab v\|_{L^\infty}  + \|\Dm^{\frac{\as}{2}}v\|_{L^{\frac{2\d}{\as-2}}}{\indic_{\as>2}}\Big)\bigg(\int_{(\R^\d)^2} \g_{\eta}(x-y)d\big(\frac1N\sum_{i=1}^N\delta_{x_i}-\mu\big)^{\otimes 2}\\
+ {\frac{ C_\zeta^2}{N^2}\sum_{1\le i\ne j\le N} \f_{\eta/c}(x_i-x_j) + C_\zeta\|\mu\|_{L^{p}}\eta^{\frac{\d(p-1)}{p}-\s}}\bigg),
\end{multline}
Applying \cref{prop:MEmon} to the preceding expression and choosing $\eta=\la$, the conclusion of the proof follows after a little bookkeeping.
\end{proof}

\section{Application: optimal mean-field convergence rate}\label{sec:appMF}
This section is devoted to the mean-field limit and the analysis of the mean-field equation.

\subsection{Modulated energy bound}\label{ssec:appMFmain}
Using the main functional inequality of \cref{thm:FI'}, we now show convergence of the empirical measure for the mean-field particle dynamics \eqref{eq:MFode} to a (necessarily unique) solution of the limiting PDE \eqref{eq:MFlim} in the modulated energy distance with the optimal rate $N^{\frac{\s}{\d}-1}$. This proves \Cref{thm:mainMF,thm:mainMFs} (cf. \cite[Theorem 1.5]{RS2022}, \cite[Theorem 1.1]{RS2021}), which we obtain as a special case of the following result that holds for $(\s,\phi)$-admissible potentials.

\begin{thm}\label{thm:mainMF'}
Suppose that $0\le \s<\d$ and $\be = \infty$. Let $\g$ be $(\s,\as)$-admissible for $\as\in (\d,\d+2)$ and $\mathsf{V}$ be as in \cref{thm:mainMF}. Assume the equation \eqref{eq:MFlim} admits a solution $\mu \in L^\infty([0,T],\P(\R^\d)\cap L^\infty(\R^\d))$, for some $T>0$, satisfying \eqref{nmut}. If $\s=0$, then also assume that $\int_{\R^\d}\log(1+|x|)d\mu^t(x) < \infty$ for all $t\in [0,T]$.

Then there exists a constant  $C=C(\d,\s,\as)>0$ such that for any solution $\ux_N$ of \eqref{eq:MFode},
\begin{multline}
\Fr_N(\ux_N^t,\mu^t) + \frac{(-\log\la^t + C +C_\rho) }{2 N} \indic_{\s=0} + C C_\zeta \|\mu^t\|_{L^\infty}(\la^t)^{\d-\s}\\
 \le Ce^{C {C_\zeta^2}\mathcal{N}(u^t)}\Bigg(\Fr_N(\ux_N^0,\mu^0) + \sup_{\tau \in [0,t]}\Big( \frac{(-\log\la^\tau + C + C_\rho) }{2 N} \indic_{\s=0} + C C_\zeta \|\mu^\tau\|_{L^\infty}(\la^\tau)^{\d-\s}\Big)\Bigg),
\end{multline}
where $\la^t \coloneqq (N\|\mu^t\|_{L^\infty})^{-1/\d}$, $\mathcal{N}(u^t)$ is as in \eqref{nmut}, and $C_\rho=0$ if $\rho=0$.

Suppose now that $\s<\d-2$ and $\be \in (0,\infty]$. Let $\g$ be $(\s,\as)$-admissible such that $-\Delta\g$ is ($\s+2,\as)$-admissible for $\as \in (\d,\d+2)$. Assume the equation \eqref{eq:MFlims} admits a solution $\mu$ obeying the same properties as above. Then for any strong solution $\ux_N$ of \eqref{eq:MFsde},
\begin{multline}
\E\Big[\Fr_N(\ux_N^t,\mu^t)  + \frac{(-\log\la^t + C +C_\rho) }{2 N} \indic_{\s=0} + C C_\zeta \|\mu^t\|_{L^\infty}(\la^t)^{\d-\s}+\frac{C C_\zeta}{\be}\|\mu^t\|_{L^\infty} (\la^t)^{\s+2-\d}\Big] \\
 \le Ce^{C {C_\zeta^2}\mathcal{N}(u^t)}\Bigg(\Fr_N(\ux_N^0,\mu^0) \\
 + \sup_{\tau \in [0,t]}\Big(  + \frac{(-\log\la^\tau + C +C_\rho) }{2 N} \indic_{\s=0} + C C_\zeta \|\mu^\tau\|_{L^\infty}(\la^t)^{\d-\s}+\frac{C C_\zeta}{\be}\|\mu^\tau\|_{L^\infty} (\la^\tau)^{\d-\s-2}\Big)\Bigg).
\end{multline}
\end{thm}

\begin{remark}
Building upon \cref{rem:commrho}, let us note that we also have a mean-field convergence rate \emph{without} additive error if $\g$ is of the form \eqref{eq:grhoform} for $\as \in (\d,\d+2)$. This follows from the same proof for \cref{thm:mainMF'} below, only now simply applying \cref{cor:scaledcomm} with $\eta=0$ (i.e. no renormalization). For $\beta=\infty$, we have the estimate
\begin{align}
\Fr_N(\ux_N^t,\mu^t) \le Ce^{C\mathcal{N}(u^t)}\Fr_N(\ux_N^0,\mu^0).
\end{align}
\end{remark}

\begin{proof}[Proof of \cref{thm:mainMF'}]
We first consider the case $\beta=\infty$. We recall (e.g. see \cite[Lemma 2.1]{Serfaty2020} or \cite[Lemma 3.6]{RS2024ss}) that $\Fr_N(\ux_N^t,\mu^t)$ satisfies the differential inequality
\begin{align}
\frac{d}{dt}\Fr_N(\ux_N^t,\mu^t) \leq \int_{(\R^\d)^2\setminus\triangle}\nabla\g(x-y)\cdot\pa*{u^t(x)-u^t(y)}d\Big(\frac{1}{N}\sum_{i=1}^N\delta_{x_i^t}-\mu^t\Big)^{\otimes2}(x,y) .
\end{align}
where $u^t\coloneqq -\M\nabla\g\ast\mu^t+\mathsf{V}^t$. Since $\g$ is $(\s,\as)$-admissible, we may apply \cref{thm:FI'} pointwise in $t$ to the preceding right-hand side. Integrating in time on both sides of the resulting inequality and applying the fundamental theorem of calculus, it follows that
\begin{multline}
\Fr_N(\ux_N^t,\mu^t) + \frac{(-\log\la^t + C +C_\rho) }{2 N} \indic_{\s=0} + C C_\zeta \|\mu^t\|_{L^\infty}(\la^t)^{\d-\s} \leq \Fr_N(\ux_N^0,\mu^0)  \\ 
+\frac{(-\log\la^t + C +C_\rho) }{2 N} \indic_{\s=0} + C C_\zeta \|\mu^t\|_{L^\infty}(\la^t)^{\d-\s}\\
+ C{C_\zeta^2}\int_0^t \|\nabla u^\tau \|_{L^\infty}\Big(\Fr_N(\ux_N^\tau,\mu^\tau) + \frac{(-\log\la^t + C +C_\rho) }{2 N} \indic_{\s=0} + C C_\zeta \|\mu^t\|_{L^\infty}(\la^t)^{\d-\s}\Big)d\tau
\end{multline}
for some constant $C>0$ depending only on $\d,\s,\as$. 
Assuming that the constant $C>0$ above is sufficiently large, the left-hand side defines a nonnegative quantity  in view of \cref{rem:MElb}. An application of the Gr\"onwall-Bellman lemma then completes the proof.

We now consider the case $\beta<\infty$, for which we restrict to $\s<\d-2$. Proceeding formally (see \cite[Section 6]{RS2021} for the rigorous justification of the computation), It\^o's formula yields
\begin{multline}
d\Fr_N(\ux_N^t,\mu^t) \le \frac12\int_{(\R^\d)^2\setminus\triangle}\nab\g(x-y)\cdot(u^t(x)-u^t(y))d\Big(\frac1N\sum_{i=1}^N\delta_{x_i^t}-\mu^t\Big)^{\otimes 2}\\
+ \frac{1}{\be}\int_{(\R^\d)^2\setminus\triangle}\Delta\g(x-y)d\Big(\frac1N\sum_{i=1}^N\delta_{x_i^t}-\mu^t\Big)^{\otimes 2}dt + \sqrt{\frac{2}{\be}}\frac1N\sum_{i=1}^N \PV\int_{\R^\d}\nab\g(x_i^t-y)d\Big(\frac1N\sum_{i=1}^N\delta_{x_i}^t-\mu^t\Big)(y) \cdot dW_i^t,
\end{multline}
Taking expectations, the second term on the last line disappears, and rewriting in integral form yields
\begin{multline}\label{eq:ExpFNnoise}
\E\Big[\Fr_N(\ux_N^T,\mu^T)\Big] -  \E\Big[\Fr_N(\ux_N^0,\mu^0)\Big]\le  \frac12\int_0^t\E\Bigg[\int_{(\R^\d)^2\setminus\triangle}\nab\g(x-y)\cdot(u^t(x)-u^t(y))d\Big(\frac1N\sum_{i=1}^N\delta_{x_i^t}-\mu^t\Big)^{\otimes 2}\Bigg]dt \\
+\frac{1}{\be}\int_0^t\E\Bigg[\int_{(\R^\d)^2\setminus\triangle}\Delta\g(x-y)d\Big(\frac1N\sum_{i=1}^N\delta_{x_i^t}-\mu^t\Big)^{\otimes 2}\Bigg]dt 
\end{multline}
By assumption that $\Delta\g$ is $(\s+2,\as)$-admissible, we may use \cref{rem:MElb} with $\s$ replaced by $\s+2<\d$ (here, we are using the sub-Coulomb assumption) to trivially bound
\begin{align}
\frac{1}{\be}\int_{(\R^\d)^2\setminus\triangle}\Delta\g(x-y)d\Big(\frac1N\sum_{i=1}^N\delta_{x_i^t}-\mu^t\Big)^{\otimes 2} \le \frac{C C_\zeta}{\be}\|\mu^t\|_{L^\infty}(\la^t)^{\d-\s-2}.
\end{align}
Inserting into \eqref{eq:ExpFNnoise} above, we can now conclude with Gr\"onwall as before, using the linearity of expectation.
\end{proof}

\subsection{Mean-field regularity}\label{ssec:appMFreg}
Suppose that $\mu^t \in\mathcal{P}(\R^\d)$ is a solution to \eqref{eq:MFlims}. We conclude the main body of the paper by verifying that for exact sub-Coulomb Riesz potentials, if $v^t = \M\nab\g\ast\mu^t$, then the condition
\begin{align}\label{eq:vfregcon}
\sup_{t\in [0,T]}\Big( \|\nab v^t\|_{L^\infty} + \|\Dm^{\frac{\as}{2}}v^t\|_{L^{\frac{2\d}{\as-2}}}\Big) <\infty,
\end{align}
for $\as\in (\d,\d+2)$, is satisfied provided that the initial density $\mu^0$ is sufficiently regular. 
We only consider the exact Riesz case---again for reasons of simplicity---and leave it to the reader to generalize the results of this subsection to the Riesz-type case. Moreover, we only consider the sub-Coulomb case because the Coulomb/super-Coulomb case only features $\|\nab v^t\|_{L^\infty}$, the control of which has already been established, e.g. see \cite{CJ2021}.  The addition of an external field $\mathsf{V}^t$ poses no additional complications. It only adds extra terms to the computation (cf. \cref{rem:extVFreg}), and while $\|\mu^t\|_{L^p}$ is no longer monotone, it is controlled by $\|\mu^0\|_{L^p}$ locally uniformly in time.

We focus here only on the \emph{a priori} estimates: existence of such solutions follows by standard vanishing viscosity arguments (e.g. see \cite[Sections 5,6]{BIK2015}), while uniqueness follows from Cauchy-Lipschitz theory. Throughout this section, we assume that $\mu^t$ is a smooth, probability density-valued solution of the mean-field equation.

The satisfaction of the condition \eqref{eq:vfregcon} is a consequence of the following two lemmas and the nonincrease of $\|\mu^t\|_{L^r}$ for any $1\le r\le \infty$. For the latter fact, see the proof of \cite[Remark 3.4]{RS2022}.

\begin{lemma}\label{lem:vfreg1}
Let $v\coloneqq \M\nab\g\ast f$ for any $f\in\Sc(\R^\d)$. Then it holds that for any $p>\frac{\d}{\d-\s-2}$
\begin{align}\label{eq:vfreg11}
\|\nab v\|_{L^\infty} \le C|\M|_{\infty}\|f\|_{L^1}^{\frac{p(\d-\s-2)-\d}{\d(p-1)}} \|f\|_{L^p}^{\frac{p(\s+2)}{\d(p-1)}}
\end{align}
and further supposing $\d\ge 3$, for any $\al \in (\s+2,\d-1-\frac{\as}{2})$,
\begin{align}\label{eq:vfreg12}
\|\Dm^{\frac{\as}{2}}v\|_{L^{\frac{2\d}{\as-2}}} \le C \begin{cases} \|f\|_{L^{\frac{\d}{\d-2-\s}}}  , & \s\le \d-1-\frac{\as}{2} \\ \|\Dm^\al f\|_{L^{\frac{\d}{\al + \d-2-\s}}}, & \s>\d-1-\frac{\as}{2}\end{cases} ,
\end{align}
where $|\cdot|_{\infty}$ denotes the operator norm and $C>0$ depends only on $\d,\s,\al$. 
\end{lemma}

\begin{lemma}\label{lem:vfreg2}
Let $\al\ge 1$ and $1<p<\infty$. Then there exists a continuous, nondecreasing function $\W:[0,\infty)^3\rightarrow [0,\infty)$, depending only on $\d,\s,\al,p$ and vanishing if any of its arguments is zero, such that\footnote{In the course of proving the lemma, we show a more precise estimate. We choose to omit it here to simplify the exposition. Also, one could consider $p\in\{1,\infty\}$ for integer $\alpha$, replacing $\Dm^\al$ by $\nab^{\otimes \al}$.}
\begin{align}\label{eq:vfreg2}
\forall T\ge 0, \qquad \sup_{t\in [0,T]} \|\Dm^\al\mu^t\|_{L^p} \le \W(T,\|\mu^0\|_{L^\infty}, \|\Dm^\al\mu^0\|_{L^p}).
\end{align}
\end{lemma}

Remark that the assumption $\al\ge 1$ in the statement is necessary for the argument to close. We begin with the proof of \cref{lem:vfreg1}, which is a consequence of some elementary harmonic analysis.

\begin{proof}[Proof of \cref{lem:vfreg1}]
Observe that for any $R>0$, 
\begin{align}
|\nab v(x)| &\le  C|\M|_{\infty}\int_{\R^\d} |x-y|^{2+\s-\d}f(y) \nn\\
&\le  C|\M|_{\infty}\Bigg(\int_{|x-y|\le R}|x-y|^{2+\s-\d}f(y)  + \int_{|x-y|> R}|x-y|^{2+\s-\d}f(y)\Bigg) \nn\\
&\le  C|\M|_{\infty}\Big( \|f\|_{L^p} R^{\frac{\d(p-1)}{p} + 2+\s-\d} + R^{2+\s-\d}\|f\|_{L^1}\Big),
\end{align}
where the ultimate line is by H\"older's inequality. Implicitly, we are using that $\s<\d-2$. Optimizing the choice of $R$ then yields \eqref{eq:vfreg11}.

Recalling that $\frac{1}{\cd}\g$ is the convolution kernel of the Fourier multiplier $\Dm^{\s-\d}$, we see that
\begin{align}
\|\Dm^{\frac{\as}{2}}v\|_{L^{\frac{2\d}{\as-2}}} = \cd \|\Dm^{\frac{\as}{2}+\s-\d}\M\nab\mu\|_{L^{\frac{2\d}{\as-2}}} \le C\cd|\M|_{\infty}\|\Dm^{\frac{\as}{2}+1+\s-\d}\mu\|_{L^{\frac{2\d}{\as-2}}},
\end{align}
where we have also used the boundedness of the Riesz transform on $L^{\frac{2\d}{\as-2}}$ to replace the operator $\M\nab$ by $|\M|_{\infty}\Dm$. Note that since we restrict to $\d\ge 3$ and $\as<\d+2$, we always have $2<\frac{2\d}{\as-2}<\infty$. 
The operator $\Dm^{\frac{\as}{2}+1+\s-\d}$ is smoothing if $\s<\d-1-\frac{\as}{2}$, the identity if $\s=\d-1-\frac{\as}{2}$, and differentiating if $\s>\d-1-\frac{\as}{2}$. 

If $\s\le \d-1-\frac{\as}{2}$ (the equality case is trivial), then the Hardy-Littlewood-Sobolev lemma implies that
\begin{align}
\|\Dm^{\frac{\as}{2}+1+\s-\d}f\|_{L^{\frac{2\d}{\as-2}}} \le C \|f\|_{L^{\frac{\d}{\d-2-\s}}}.
\end{align}
If $\s>\d-1-\frac{\as}{2}$, then Sobolev embedding implies that for any $\s+2> \al\ge \frac{\as}{2}+1+\s-\d$, 
\begin{align}
\|\Dm^{\frac{\as}{2}+1+\s-\d}f\|_{L^{\frac{2\d}{\as-2}}} \le C \|\Dm^{\al} f\|_{L^p}, \qquad p = \frac{\d}{\al + \d-2-\s}.
\end{align}
Together, these bounds yield \eqref{eq:vfreg12}.
\end{proof}

\begin{remark}
Evidently, we may allow for $\al>\s+2$ (the case $\al=\s+2$ is excluded due to the failure of the endpoint Sobolev embedding), provided that we replace the homogeneous seminorm on the right-hand side with the inhomogeneous norm $\|f\|_{W^{\al,1}}$.
\end{remark}

We now prove \cref{lem:vfreg2}.
\begin{proof}[Proof of \cref{lem:vfreg2}]
Observe from the chain rule and mean-field equation \eqref{eq:MFlims} that
\begin{align}
\frac{d}{dt} \|\Dm^{\al}\mu^t\|_{L^{p}}^p &= p\int_{\R^\d} |\Dm^{\al}\mu^t|^{p-2}\Dm^{\al}\mu^t \Dm^{\al}\div\Big(\mu^t\M\nab\g\ast\mu^t + \frac1\be\nab\mu^t\Big) \nn\\
&\le \underbrace{p\int_{\R^\d} |\Dm^{\al}\mu^t|^{p-2}\Dm^{\al}\mu^t \nab\Dm^\al\mu^t \cdot\M\nab\g\ast\mu^t}_{I_1} \nn\\
&\ph + \underbrace{p\int_{\R^\d} |\Dm^{\al}\mu^t|^{p-2}\Dm^{\al}\mu^t \nab\mu^t \cdot \Dm^{\al}\M\nab\g\ast\mu^t}_{I_2} \nn\\
&\ph -  \underbrace{p\al\indic_{\al\ge1}\int_{\R^\d} |\Dm^{\al}\mu^t|^{p-2}\Dm^{\al}\mu^t \nab^{\otimes 2}\Dm^{\al-2}\mu^t : \nab\M\nab\g\ast\mu^t}_{I_3} \nn\\
&\ph  +  \underbrace{p\int_{\R^\d} |\Dm^{\al}\mu^t|^{p-2}\Dm^{\al}\mu^t \Dm^{\al}\mu^t  \div\M\nab\g\ast\mu^t}_{I_4}  \nn\\ 
&\ph  +  \underbrace{p\int_{\R^\d} |\Dm^{\al}\mu^t|^{p-2}\Dm^{\al}\mu^t \mu^t  \div\Dm^{\al}\M\nab\g\ast\mu^t}_{I_5} \nn\\
&\ph  -  \underbrace{p\al\indic_{\al\ge1}\int_{\R^\d} |\Dm^{\al}\mu^t|^{p-2}\Dm^{\al}\mu^t \nab\Dm^{\al-2}\mu^t \cdot \nab\div\M\nab\g\ast\mu^t}_{I_6} \nn\\
&\ph + \underbrace{\int_{\R^\d} |\Dm^{\al}\mu^t|^{p-2}\Dm^{\al}\mu^t\Big(\mathcal{C}_1^t  + \mathcal{C}_2^t\Big)}_{I_7},
\end{align}
where we have abbreviated the (higher-order) commutators 
\begin{multline}
\mathcal{C}_1^t \coloneqq  \Dm^\al\Big(\nab\mu^t \cdot \M\nab\g\ast\mu^t\Big) - \nab\Dm^\al\mu^t \cdot \M\nab\g\ast\mu^t - \nab\mu^t \cdot\Dm^\al\M\nab\g\ast\mu^t \\
 +\al\indic_{\al\ge 1}\nab^{\otimes 2}\Dm^{\al-2}\mu^t : \nab\M\nab\g\ast\mu^t,
\end{multline}
\begin{multline}
\mathcal{C}_2^t \coloneqq\Dm^\al\Big(\mu^t \div\M\nab\g\ast\mu^t\Big) - \Dm^\al\mu^t \div\M\nab\g\ast\mu^t - \mu^t \Dm^\al\div\M\nab\g\ast\mu^t \\
 +\al\indic_{\al\ge 1}\nab\Dm^{\al-2}\mu^t \cdot  \nab\div\M\nab\g\ast\mu^t.
\end{multline}
Remark that if $\M$ is antisymmetric, then all terms involving $\div\M\nab$ vanish.

Integrating by parts, we see that
\begin{align}
I_1 +   I_4 = (p-1)\int_{\R^\d}  |\Dm^{\al}\mu^t|^{p} \div\M\nab\g\ast\mu^t. \label{eq:MFreg0}
\end{align}
By our assumption \eqref{eq:Mnd} on $\M$ and the explicit form of $\g$, we have that $\div\M\nab\g \le 0$. As $\mu^t\ge 0$, it follows that the preceding right-hand side is nonpositive and may be discarded. We estimate directly the remaining terms.

Assuming that $\al\ge 1$ (otherwise, the estimate will not close due to the $\nab\mu^t$ factor below), H\"older's inequality yields for $\frac{1}{p_1} + \frac{1}{p_2} = \frac1p$,
\begin{align}
I_2 &\le \|\Dm^\al \mu^t\|_{L^p}^{p-1} \|\nab\mu^t\|_{L^{p_1}} \|\Dm^{\al}\M\nab\g\ast\mu^t\|_{L^{p_2}} \nn\\
 &\le C|\M|_{\infty} \|\Dm^\al \mu^t\|_{L^p}^{p-1} \|\nab\mu^t\|_{L^{p_1}} \|\Dm^{\al+1+\s-\d}\mu^t\|_{L^{p_2}}.
\end{align}
If $\al + 1+\s-\d \le 0$, then we choose $(p_1, p_2) = (p,\infty)$ and Gagliardo-Nirenberg interpolation \cite[Theorem 2.44]{BCD2011} plus H\"older's inequality imply
\begin{align}\label{eq:HolGN1}
 \|\nab\mu^t\|_{L^{p_1}} \|\Dm^{\al+1+\s-\d}\mu^t\|_{L^{p_2}} \le C\|\mu^t\|_{L^p}^{1-\frac{1}{\al}} \|\Dm^\al\mu^t\|_{L^p}^{\frac{1}{\al}} \|\mu^t\|_{L^1}^{1-\frac{(\al+1+\s)}{\d}} \|\mu^t\|_{L^{\infty}}^{\frac{(\al+1+\s)}{\d}}.
\end{align}
If $\al+1+\s-\d>0$, then we choose $(p_1,p_2) = (\al p, \frac{\al p}{\al-1})$ so that by Gagliardo-Nirenberg interpolation,
\begin{align}\label{eq:HolGN1'}
 \|\nab\mu^t\|_{L^{p_1}} \|\Dm^{\al+1+\s-\d}\mu^t\|_{L^{p_2}}  &\le C\|\mu^t\|_{L^\infty}^{\frac{\al-1}{\al}}   \|\mu^t\|_{L^{\frac{p(\d-1-\s)}{\d-2-\s}}}^{\frac{\d-1-\s}{\al}} \|\Dm^\al\mu^t\|_{L^p}^{\frac{\al+2+\s-\d}{\al}}.
\end{align}
Thus, in all cases,
\begin{multline}
I_2\le C|\M|_{\infty}\|\Dm^\al \mu^t\|_{L^p}^{p-1} \Big(\|\mu^t\|_{L^p}^{1-\frac{1}{\al}} \|\Dm^\al\mu^t\|_{L^p}^{\frac{1}{\al}} \|\mu^t\|_{L^1}^{1-\frac{(\al+1+\s)}{\d}} \|\mu^t\|_{L^{\infty}}^{\frac{(\al+1+\s)}{\d}}\indic_{\al+1+\s-\d\le 0} \\
 + \|\mu^t\|_{L^\infty}^{\frac{\al-1}{\al}}   \|\mu^t\|_{L^{\frac{p(\d-1-\s)}{\d-2-\s}}}^{\frac{\d-1-\s}{\al}} \|\Dm^\al\mu^t\|_{L^p}^{\frac{\al+2+\s-\d}{\al}}\indic_{\al+1+\s-\d> 0}\Big). \label{eq:MFreg1}
\end{multline}


Arguing similarly, using also the $L^p$ boundedness of the Riesz transform, 
\begin{align}
I_3 &\le \|\Dm^\al\mu^t\|_{L^{p}}^{p-1} \|\nab^{\otimes 2}\Dm^{\al-2}\mu^t\|_{L^p} \|\nab\M\nab\g\ast\mu^t\|_{L^\infty} \nn\\
&\le C|\M|_{\infty}\|\Dm^\al\mu^t\|_{L^{p}}^{p}\|\Dm^{2+\s-\d}\mu^t\|_{L^\infty} \nn\\
&\le C'|\M|_{\infty}\|\Dm^\al\mu^t\|_{L^{p}}^{p} \|\mu^t\|_{L^1}^{1-\frac{\d-\s-2}{\d}} \|\mu^t\|_{L^\infty}^{\frac{\d-\s-2}{\d}}, \label{eq:MFreg2}
\end{align}
where we have also used H\"older's inequality on the $L^\infty$ factor in obtaining the final line. 
Implicitly, we are using the assumption $\s<\d-2$. 

Again using H\"older's inequality, we find that
\begin{align}
I_5 \le \|\Dm^\al\mu^t\|_{L^p}^{p-1} \|\mu^t\|_{L^{p_1}} \| \div\Dm^{\al}\M\nab\g\ast\mu^t\|_{L^{p_2}},
\end{align}
where $\frac{1}{p_1}+\frac{1}{p_2} = \frac1p$. We choose $p_2$ so that
\begin{align}
\d-\s-2 = \d\Big(\frac{1}{p} - \frac{1}{p_2}\Big) \Rightarrow p_2 = \frac{\d p}{\d-p(\d-\s-2)},
\end{align}
which, by the Hardy-Littlewood-Sobolev lemma, implies that
\begin{align}
\| \div\Dm^{\al}\M\nab\g\ast\mu^t\|_{L^{p_2}} \le C|\M|_{\infty} \|\Dm^{\al}\mu^t\|_{L^p}.
\end{align}
Therefore,
\begin{align}
I_5  \le C|\M|_{\infty}\|\Dm^\al\mu^t\|_{L^p}^p \|\mu^t\|_{L^{\frac{\d}{\d-\s-2}}}. \label{eq:MFreg3}
\end{align}

Next, we use H\"older's inequality to bound
\begin{align}
I_6  &\le \|\Dm^\al\mu^t\|_{L^p}^{p-1} \|\Dm^{\al-1}\mu^t\|_{L^{p_1}} \|\nab\div\M\nab\g\ast\mu^t\|_{L^{p_2}} \nn\\
&\le C|\M|_{\infty}\|\Dm^\al\mu^t\|_{L^p}^{p-1} \|\Dm^{\al-1}\mu^t\|_{L^{p_1}} \|\Dm^{3+\s-\d}\mu^t\|_{L^{p_2}},
\end{align}
where $\frac{1}{p_1}+\frac{1}{p_2}=\frac{1}{p}$. If $3+\s-\d\le 0$, then we may choose $(p_1,p_2) = (p,\infty)$ and use Gagliardo-Nirenberg interpolation plus H\"older's inequality (similar to \eqref{eq:HolGN1}) to estimate
\begin{align}
\|\Dm^{\al-1}\mu^t\|_{L^{p_1}} \|\Dm^{3+\s-\d}\mu^t\|_{L^{p_2}} \le C\|\mu^t\|_{L^p}^{\frac{1}{\al}} \|\Dm^\al\mu^t\|_{L^p}^{\frac{\al-1}{\al}} \|\mu^t\|_{L^1}^{1-\frac{\d-\s-3}{\d}} \|\mu^t\|_{L^\infty}^{\frac{\d-\s-3}{\d}}.
\end{align}
If $3+\s-\d>0$, then (similar to \eqref{eq:HolGN1'}) we choose $(p_1,p_2) = (\frac{\al p}{p-1},\al p)$ and use Gagliardo-Nirenberg interpolation to obtain
\begin{align}
\|\Dm^{\al-1}\mu^t\|_{L^{p_1}} \|\Dm^{3+\s-\d}\mu^t\|_{L^{p_2}}  \le \|\mu^t\|_{L^\infty}^{\frac{1}{\al}} \|\Dm^\al\mu^t\|_{L^p}^{\frac{\al-1}{\al}} \|\mu^t\|_{L^{\frac{p(\al+\d-3-\s)}{\d-\s-2}}}^{\frac{\al+\d-3-\s}{\al}} \|\Dm^\al\mu^t\|_{L^p}^{\frac{3+\s-\d}{\al}}.
\end{align}
Thus, in all cases, we have
\begin{multline}
I_6 \le C|\M|_{\infty}\|\Dm^\al\mu^t\|_{L^p}^{p-\frac1\al}\Big(\|\mu^t\|_{L^p}^{\frac1\al}\|\mu^t\|_{L^1}^{1-\frac{\d-\s-3}{\d}} \|\mu^t\|_{L^\infty}^{\frac{\d-\s-3}{\d}}\indic_{3+\s-\d\le 0} \\
+ \|\mu^t\|_{L^\infty}^{\frac{1}{\al}}  \|\mu^t\|_{L^{\frac{p(\al+\d-3-\s)}{\d-\s-2}}}^{\frac{\al+\d-3-\s}{\al}} \|\Dm^\al\mu^t\|_{L^p}^{\frac{3+\s-\d}{\al}}\indic_{3+\s-\d>0}\Big). \label{eq:MFreg4}
\end{multline}

Finally, we handle the commutator terms. Using the Kato-Ponce commutator inequality \cite[Theorem 5.1]{Li2019}, we find that
\begin{align}
\|\mathcal{C}_1^t\|_{L^p} \le C\Big(\|\nab\mu^t \|_{L^{p_1}}\|\Dm^{\al}\M\nab\g\ast\mu^t\|_{L^{p_2}}  + \indic_{\al\ge 2} \|\nab\Dm^{\al-1}\mu^t \|_{L^{p_3}} \|\nab\M\nab\g\ast\mu^t \|_{L^{p_4}}\Big),
\end{align}
where $C>0$ depends only on $\d,\s,p$. Here, $1 < p_1,p_2, p_3, p_4 \le  \infty$ are such that $\frac{1}{p_1} + \frac{1}{p_2} = \frac{1}{p_3} + \frac{1}{p_4} = \frac{1}{p}$. Each of the terms on the right-hand side have been estimated above, and we find that
\begin{multline}\label{eq:MFreg5C1}
\|\mathcal{C}_1^t\|_{L^p} \le C|\M|_{\infty}\Big(\|\mu^t\|_{L^p}^{1-\frac{1}{\al}} \|\Dm^\al\mu^t\|_{L^p}^{\frac{1}{\al}} \|\mu^t\|_{L^1}^{1-\frac{(\al+1+\s)}{\d}} \|\mu^t\|_{L^{\infty}}^{\frac{(\al+1+\s)}{\d}}\indic_{\al+1+\s-\d\le 0} \\
+ \|\mu^t\|_{L^\infty}^{\frac{\al-1}{\al}}   \|\mu^t\|_{L^{\frac{p(\d-1-\s)}{\d-2-\s}}}^{\frac{\d-1-\s}{\al}} \|\Dm^\al\mu^t\|_{L^p}^{\frac{\al+2+\s-\d}{\al}}\indic_{\substack{\al+1+\s-\d> 0 }} +  \|\Dm^\al\mu^t\|_{L^{p}} \|\mu^t\|_{L^1}^{1-\frac{\d-\s-2}{\d}} \|\mu^t\|_{L^\infty}^{\frac{\d-\s-2}{\d}}\indic_{\al\ge 2}\Big).
\end{multline}
Similarly, 
\begin{align}
\|\mathcal{C}_2^t\|_{L^p} \le C\Big(\|\mu^t \|_{L^{p_1}}\|\div\Dm^{\al}\M\nab\g\ast\mu^t\|_{L^{p_2}}  + \indic_{\al\ge 2} \|\Dm^{\al-1}\mu^t \|_{L^{p_3}} \|\nab\div\M\nab\g\ast\mu^t \|_{L^{p_4}}\Big),
\end{align}
with the $p_i$ as above. Again, we note that each of the terms on the right-hand side has been previously estimated, and we find that
\begin{multline}\label{eq:MFreg5C2}
\|\mathcal{C}_2^t\|_{L^p} \le C|\M|_{\infty}\Big(\|\mu^t\|_{L^{\frac{\d}{\d-\s-2}}}\|\Dm^\al\mu^t\|_{L^p} + \|\mu^t\|_{L^p}^{\frac{1}{\al}} \|\Dm^\al\mu^t\|_{L^p}^{\frac{\al-1}{\al}} \|\mu^t\|_{L^1}^{1-\frac{\d-\s-3}{\d}} \|\mu^t\|_{L^\infty}^{\frac{\d-\s-3}{\d}}\indic_{\substack{3+\s-\d\le 0 \\ \al\ge 2}}\\
+\|\mu^t\|_{L^\infty}^{\frac{1}{\al}} \|\Dm^\al\mu^t\|_{L^p}^{\frac{\al-1}{\al}} \|\mu^t\|_{L^{\frac{p(\al+\d-3-\s)}{\d-\s-2}}}^{\frac{\al+\d-3-\s}{\al}} \|\Dm^\al\mu^t\|_{L^p}^{\frac{3+\s-\d}{\al}}\indic_{\substack{3+\s-\d> 0 \\ \al\ge 2}} \Big).
\end{multline} 
Combining the preceding estimates \eqref{eq:MFreg5C1}, \eqref{eq:MFreg5C2} with H\"older's and triangle inequalities, we conclude that
\begin{multline}
I_7 \le C|\M|_{\infty}\|\Dm^\al\mu^t\|_{L^p}^{p-1} \Big(\|\mu^t\|_{L^p}^{1-\frac{1}{\al}} \|\Dm^\al\mu^t\|_{L^p}^{\frac{1}{\al}} \|\mu^t\|_{L^1}^{1-\frac{(\al+1+\s)}{\d}} \|\mu^t\|_{L^{\infty}}^{\frac{(\al+1+\s)}{\d}}\indic_{\al+1+\s-\d\le 0} \\
+ \|\mu^t\|_{L^\infty}^{\frac{\al-1}{\al}}   \|\mu^t\|_{L^{\frac{p(\d-1-\s)}{\d-2-\s}}}^{\frac{\d-1-\s}{\al}} \|\Dm^\al\mu^t\|_{L^p}^{\frac{\al+2+\s-\d}{\al}}\indic_{\substack{\al+1+\s-\d> 0 }} +  \|\Dm^\al\mu^t\|_{L^{p}} \|\mu^t\|_{L^1}^{1-\frac{\d-\s-2}{\d}} \|\mu^t\|_{L^\infty}^{\frac{\d-\s-2}{\d}}\indic_{\al\ge 2}\\
+\|\mu^t\|_{L^{\frac{\d}{\d-\s-2}}}\|\Dm^\al\mu^t\|_{L^p} + \|\mu^t\|_{L^p}^{\frac{1}{\al}} \|\Dm^\al\mu^t\|_{L^p}^{\frac{\al-1}{\al}} \|\mu^t\|_{L^1}^{1-\frac{\d-\s-3}{\d}} \|\mu^t\|_{L^\infty}^{\frac{\d-\s-3}{\d}}\indic_{\substack{3+\s-\d\le 0 \\ \al\ge 2}}\\
+\|\mu^t\|_{L^\infty}^{\frac{1}{\al}} \|\Dm^\al\mu^t\|_{L^p}^{\frac{\al-1}{\al}} \|\mu^t\|_{L^{\frac{p(\al+\d-3-\s)}{\d-\s-2}}}^{\frac{\al+\d-3-\s}{\al}} \|\Dm^\al\mu^t\|_{L^p}^{\frac{3+\s-\d}{\al}}\indic_{\substack{3+\s-\d> 0 \\ \al\ge 2}}\Big). \label{eq:MFreg5}
\end{multline}

Collecting the estimates \eqref{eq:MFreg0}, \eqref{eq:MFreg1}, \eqref{eq:MFreg2}, \eqref{eq:MFreg3}, \eqref{eq:MFreg4}, \eqref{eq:MFreg5}, we note that the exponent of the $\|\Dm^\al\mu^t\|_{L^p}$ factor is always $\le p$. The desired conclusion now follows from Gr\"onwall's inequality and the nonincrease of $\|\mu^t\|_{L^r}$ for any $1\le r\le \infty$, using interpolation to control all the factors $\|\mu^t\|_{L^r}$ in terms of $\|\mu^t\|_{L^1}$ and $\|\mu^t\|_{L^\infty}$.
\end{proof}

\appendix

\section{Properties of the Bessel Potential}\label{ap:bessel}

In this appendix, we prove \cref{lem:bespot}, which records various important properties of the Bessel potential.

\begin{proof}[Proof of \cref{lem:bespot}]
	Using the gamma function identity $A^{-\frac{\as}{2}} = \frac{1}{\Gamma(\as/2)}\int_0^\infty e^{-t A}t^{\as/2}\frac{dt}{t}$, we see that
	\begin{align}
		(1+4\pi^2|\xi|^2)^{-\as/2} = \frac{1}{\Gamma(\as/2)}\int_0^\infty e^{-t}e^{-\pi |2\sqrt{t\pi}\xi|^2}t^{\as/2}\frac{dt}{t}.
	\end{align}
	Taking inverse Fourier transforms of both sides,
	\begin{align}
		G_\as(x) =  \frac{1}{(2\sqrt{\pi})^{\d}\Gamma(\as/2)} \int_0^\infty e^{-t}e^{-|x|^2/4t} t^{\frac{\as-\d}{2}} \frac{dt}{t}.
	\end{align}
	It follows immediately from this identity that $G_\as>0$, smooth on $\R^\d$, and $\int_{\R^\d}G_\as = 1$. That $G_\as \in \mathcal{C}^{\as-\d}$ is easily seen through the Littlewood-Paley characterization of these spaces.
	
	Suppose that $|x|\ge 2$. Let $\vec\al \in \N_0^\d$ be a multi-index. Differentiating inside the integral,
	\begin{align}\label{eq:palGs}
		\p_{\vec\al}G_\as(x) =  \frac{(-2)^{-|\vec\al|}}{(2\sqrt{\pi})^{\d}\Gamma(\as/2)} \int_0^\infty \prod_{i=1}^\d H_{\al_i}(x_i/2\sqrt{t}) e^{-t}e^{-|x|^2/4t} t^{\frac{\as-|\vec\al|-\d}{2}} \frac{dt}{t},
	\end{align} 
	where $H_n$ is the $n$-th (physicist's) Hermite polynomial. Since $|H_{n}(y)|\le C_n(1+|y|^n)$ and
	\begin{align}
		t+\frac{|x|^2}{ct} \ge \max\Big(\frac{|x|}{\sqrt{c}},t+\frac{4}{ct}\Big) \Longrightarrow  t+\frac{|x|^2}{ct}  \ge \frac{|x|}{2\sqrt{c}} + \frac{t}{2}+\frac{2}{ct},
	\end{align}
	it follows that for any $c>4$, 
	\begin{align}
		|\p_{\vec\al}G_\as(x)| \le \frac{C_{|\vec\al|}}{(2\sqrt{\pi})^{\d}\Gamma(\as/2)} e^{-\frac{|x|}{2\sqrt{c}}}\int_0^\infty e^{-t/2}e^{-2/(ct)}  t^{\frac{\as-|\vec\al|-\d}{2}} \frac{dt}{t}.
	\end{align}
	Evidently, the integral in $t$ converges, yielding us \eqref{eq:bespot1}. 
	
	For $|x|\le 2$, we decompose
	\begin{align}
		G_\as = G_{\as,1} + G_{\as,2} + G_{\as,3},
	\end{align}
	where
	\begin{align}
		G_{\as,1}(x) &\coloneqq |x|^{\as-\d} \frac{1}{(2\sqrt{\pi})^\d \Gamma(\as/2)}\int_0^1 e^{-t|x|^2}e^{-1/4t}t^{\frac{\as-\d}{2}}\frac{dt}{t}, \\
		G_{\as,2}(x) &\coloneqq \frac{1}{(2\sqrt{\pi})^\d \Gamma(\as/2)}\int_{|x|^2}^4 e^{-t}e^{-|x|^2/4t}t^{\frac{\as-\d}{2}}\frac{dt}{t}, \\
		G_{\as,3}(x) &\coloneqq \frac{1}{(2\sqrt{\pi})^\d \Gamma(\as/2)}\int_{4}^\infty e^{-t}e^{-|x|^2/4t}t^{\frac{\as-\d}{2}}\frac{dt}{t},
	\end{align}
	and we note that we have used a change of variables in defining $G_{\as,1}$. By the Leibniz rule,
	\begin{align}
		\p_{\vec\al}G_{\as,1}(x) &= \sum_{\vec\be\le \vec\al} {\vec\al\choose\vec\be} \p_{\vec\be}(|x|^{\as-\d})\frac{1}{(2\sqrt{\pi})^\d \Gamma(\as/2)}\int_0^1 \p_{\vec{\al}-\vec\be}(e^{-t|x|^2})e^{-1/4t}t^{\frac{\as-\d}{2}}\frac{dt}{t} \nn\\
		&=\sum_{\vec\be\le \vec\al} {\vec\al\choose\vec\be} \p_{\vec\be}(|x|^{\as-\d})\frac{1}{(2\sqrt{\pi})^\d \Gamma(\as/2)}\int_0^1 (-\sqrt{t})^{|\vec\al-\vec\be|} \prod_{i=1}^\d H_{\al_i-\be_i}(\sqrt{t}x_i)e^{-t|x|^2}e^{-1/4t}t^{\frac{\as-\d}{2}}\frac{dt}{t}.
	\end{align}
	If $\vec\be\ne \vec\al$, then  we may bound $|\p_{\vec\be}|x|^{\as-\d}| \le C |x|^{\as-|\vec\be|-\d}$, $|\sqrt{t}x| \le 2$, and $e^{-t|x|^2}\le 1$, to obtain
	\begin{multline}
		\Bigg|\p_{\vec\be}(|x|^{\as-\d})\frac{1}{(2\sqrt{\pi})^\d \Gamma(\as/2)}\int_0^1 (-\sqrt{t})^{|\vec\al-\vec\be|} \prod_{i=1}^\d H_{\al_i-\be_i}(\sqrt{t}x_i)e^{-t|x|^2}e^{-1/4t}t^{\frac{\as-\d}{2}}\frac{dt}{t}\Bigg| \\
		\le \frac{C_{|\vec\al|}}{(2\sqrt{\pi})^\d \Gamma(\as/2)} |x|^{\as-\d-|\vec\be|} \int_0^1 e^{-1/4t} t^{\frac{\as-\d}{2}} \frac{dt}{t},
	\end{multline}
	where the integral in $t$ evidently converges. If $\vec\be=\vec\al$, then Taylor expanding $e^{-t|x|^2} = 1+O(t|x|^2)$ it follows that
	\begin{align}
		\Bigg|\frac{\p_{\vec\al}(|x|^{\as-\d})}{(2\sqrt{\pi})^\d \Gamma(\as/2)}\int_0^1\Big(e^{-t|x|^2}-1\Big)e^{-1/4t}t^{\frac{\as-\d}{2}}\frac{dt}{t}\Bigg| &\le  \frac{C_{|\vec\al|} |x|^{2} |\p_{\vec\al}|x|^{\as-\d}|}{(2\sqrt{\pi})^\d \Gamma(\as/2)}\int_0^1 e^{-1/4t}t^{\frac{\as-\d}{2}}dt.
	\end{align}
	We conclude that if $\vec\al = \vec{0}$, 
	\begin{align}\label{eq:Ga11}
		{G_{\as,1}(x)} = C_{\as,\d}|x|^{\as-\d} + O(|x|^{\as+2-\d}),
	\end{align}
	and in general,
	\begin{align}\label{eq:Ga12}
		|{\p_{\vec\al}G_{\as,1}(x)}| \lesssim_{\as,|\vec\al|,\d} {|x|^{\as-|\vec\al|-\d}}.
	\end{align}
	
	For $\vec\al = (\al^1,\ldots,\al^\d)$, letting $\vec{\al}_i \coloneqq \vec\al - \vec{e}_i$, we have by the Leibniz rule and fundamental theorem of calculus,
	\begin{multline}
		\p_{\vec\al}G_{\as,2}(x) = \frac{1}{(2\sqrt{\pi})^\d \Gamma(\as/2)} \sum_{1\le i\le \d : \al^i > 0}{\al^i} \p_{\vec{\al}_i}{\Big(-2x^i  e^{-|x|^2} e^{-1/4} |x|^{\as-\d-2}\Big)}\\
		+ \frac{1}{(2\sqrt{\pi})^\d \Gamma(\as/2)} \int_{|x|^2}^4 e^{-t}\p_{\vec\al} \Big(e^{-|x|^2/4t}\Big) t^{\frac{\as-\d}{2}}\frac{dt}{t}.
	\end{multline}
	If $\vec\al = \vec{0}$, then using the pointwise bound $e^{-\frac{17}{4}} \le e^{-t-\frac{|x|^2}{4t}} \le 1$, we estimate
	\begin{align}\label{eq:Ga21}
		\int_{|x|^2}^4 e^{-t}\Big(e^{-|x|^2/4t}\Big) t^{\frac{\as-\d}{2}}\frac{dt}{t} \approx \int_{|x|^2}^4  e^{-|x|^2/4t}t^{\frac{\as-\d}{2}}\frac{dt}{t}  = \begin{cases} \frac{2}{\as-\d}(2^{\as-\d} - |x|^{\as-\d}) , & {\as\ne \d} \\ -2\log(\frac{|x|}{2}),& {\as =\d}. \end{cases}
	\end{align}
	For general $\vec\al$, writing
	\begin{align}
		-x^i|x|^{\as-\d-2} = \p_{i}\begin{cases}\frac{|x|^{\as-\d}}{\as-\d}, & {\as \ne \d} \\ -\log|x|, & {\as=\d}, \end{cases}
	\end{align}
	it is immediate from further application of the Leibniz rule that
	\begin{align}
		\frac{1}{(2\sqrt{\pi})^\d \Gamma(\as/2)} \sum_{1\le i\le \d : \al^i > 0}\Big|{\al^i} \p_{\vec{\al}_i}{\Big(-2x^i  e^{-|x|^2} e^{-1/4} |x|^{\as-\d-2}\Big)}\Big| \lesssim_{\d,\as,|\vec\al|} |x|^{\as-\d-|\vec\al|}. \label{eq:Ga22'}
	\end{align}
	Furthermore, using \eqref{eq:palGs} and that $\frac{|x|}{2\sqrt{t}} \le \frac12$, we crudely majorize
	\begin{align}
		&\frac{1}{(2\sqrt{\pi})^\d \Gamma(\as/2)} \int_{|x|^2}^4 e^{-t}\Big|\p_{\vec\al} \Big(e^{-|x|^2/4t}\Big)\Big| t^{\frac{\as-\d}{2}}\frac{dt}{t} \nn\\
		&\le \frac{1}{(2\sqrt{\pi})^\d \Gamma(\as/2)}\int_{|x|^2}^4 (2\sqrt{t})^{-|\vec\al|}\prod_{i=1}^\d   |H_{\al_i}(x_i/2\sqrt{t})| e^{-t}e^{-|x|^2/4t} t^{\frac{\as-\d}{2}}\frac{dt}{t}  \nn\\
		&\lesssim_{\d,|\vec\al|}  \begin{cases} \frac{2^{\as-\d-|\vec\al|} - |x|^{\as-\d-|\vec\al|}}{{(\as-\d-|\vec\al|)}}, & {\as-\d\ne |\vec\al|} \\ \log(\frac{|x|}{2}) , & {\as-\d = |\vec\al|}. \end{cases} \label{eq:Ga22''}
	\end{align}
	Together \eqref{eq:Ga22'}, \eqref{eq:Ga22''}, yield
	\begin{align}\label{eq:Ga22}
		|\p_{\vec\al}G_{\as,2}(x)| \lesssim_{\d,\as,|\vec\al|}  1 +  \begin{cases} \frac{2^{\as-\d-|\vec\al|} - |x|^{\as-\d-|\vec\al|}}{{(\as-\d-|\vec\al|)}}, & {\as-\d\ne |\vec\al|} \\ -\log(\frac{|x|}{2}) , & {\as-\d = |\vec\al|}. \end{cases}
	\end{align}
	
	Lastly, using that $0\le\frac{|x|^2}{4t} \le \frac14$, we have that if $\vec\al=\vec 0$,
	\begin{align}\label{eq:Ga31}
		G_{\as,3}(x) \approx \int_4^\infty e^{-t}t^{\frac{\as-\d}{2}}\frac{dt}{t}.
	\end{align}
	For general $\vec\al$, we find that 
	\begin{align}
		|\p_{\vec\al}G_{\as,3}(x)| &\le \frac{1}{(2\sqrt{\pi})^\d \Gamma(\as/2)}\int_{4}^\infty (2\sqrt{t})^{-|\vec\al|}\prod_{i=1}^\d   |H_{\al_i}(x_i/2\sqrt{t})| e^{-t}e^{-|x|^2/4t} t^{\frac{\as-\d}{2}}\frac{dt}{t} \nn\\
		&\le \frac{C_{|\vec\al|}}{(2\sqrt{\pi})^\d \Gamma(\as/2)}\int_{4}^\infty  t^{\frac{\as-|\vec\al|-\d}{2}} e^{-t}\frac{dt}{t}. \label{eq:Ga32}
	\end{align}
	
	Combining the bounds \eqref{eq:Ga11}, \eqref{eq:Ga21}, \eqref{eq:Ga31} yields the assertion \eqref{eq:bespot2}; and combining \eqref{eq:Ga12}, \eqref{eq:Ga22}, \eqref{eq:Ga32} yields \eqref{eq:bespot2'}.
	
	Finally, recalling the identity \eqref{eq:palGs}, majorizing $|H_{\al^i}(y)| \le C_{\al^i}(1+|y|)^{\al^i}$, and using that for any $c<1$, there is a constant $C_{|\vec\al|,c}>0$ such that $(1+|y|)^{|\vec\al|}e^{-|y|}\le C_{|\vec\al|,c} e^{-c|y|}$, we see that
	\begin{align}
		|x|^{|\vec\al|} |\p_{\vec\al}G_\as(x)| &\le C_{|\vec\al|,\epsilon}\frac{2^{|\vec\al|}}{(2\sqrt{\pi})^{\d}\Gamma(\as/2)} \int_0^\infty  e^{-t}e^{-|x|^2/(4+\epsilon)t} t^{\frac{\as-|\vec\al|-\d}{2}} \frac{dt}{t} \nn\\
		&=C_{|\vec\al|,\epsilon}' G_{\as}(\frac{4}{4+\epsilon}x),
	\end{align}
	for any $\epsilon>0$. This establishes \eqref{eq:bespot3} and completes the proof of the lemma.
	
\end{proof}

\bibliographystyle{alpha}\bibliography{../../MASTER}

\newcommand{\etalchar}[1]{$^{#1}$}
\begin{thebibliography}{CdCRS25}

\bibitem[AS21]{AS2021}
Scott Armstrong and Sylvia Serfaty.
\newblock Local laws and rigidity for {C}oulomb gases at any temperature.
\newblock {\em Ann. Probab.}, 49(1):46--121, 2021.

\bibitem[BBNY19]{BBNY2019}
Roland Bauerschmidt, Paul Bourgade, Miika Nikula, and Horng-Tzer Yau.
\newblock The two-dimensional {C}oulomb plasma: quasi-free approximation and
  central limit theorem.
\newblock {\em Adv. Theor. Math. Phys.}, 23(4):841--1002, 2019.

\bibitem[BCD11]{BCD2011}
Hajer Bahouri, Jean-Yves Chemin, and Rapha\"{e}l Danchin.
\newblock {\em Fourier analysis and nonlinear partial differential equations},
  volume 343 of {\em Grundlehren der Mathematischen Wissenschaften [Fundamental
  Principles of Mathematical Sciences]}.
\newblock Springer, Heidelberg, 2011.

\bibitem[BDJ24]{BDJ2024}
Didier Bresch, Mitia Duerinckx, and Pierre-Emannuel Jabin.
\newblock A duality method for mean-field limits with singular interactions.
\newblock {\em arXiv preprint arXiv:2402.04695}, 2024.

\bibitem[BG13]{BG2013}
G.~Borot and A.~Guionnet.
\newblock Asymptotic expansion of {$\beta$} matrix models in the one-cut
  regime.
\newblock {\em Comm. Math. Phys.}, 317(2):447--483, 2013.

\bibitem[BG24]{BG2024}
Ga\"{e}tan Borot and Alice Guionnet.
\newblock Asymptotic expansion of {$\beta$} matrix models in the multi-cut
  regime.
\newblock {\em Forum Math. Sigma}, 12:Paper No. e13, 93, 2024.

\bibitem[BHS19]{BHS2019}
Sergiy~V. Borodachov, Douglas~P. Hardin, and Edward~B. Saff.
\newblock {\em Discrete energy on rectifiable sets}.
\newblock Springer Monographs in Mathematics. Springer, New York, 2019.

\bibitem[BIK15]{BIK2015}
Piotr Biler, Cyril Imbert, and Grzegorz Karch.
\newblock The nonlocal porous medium equation: {B}arenblatt profiles and other
  weak solutions.
\newblock {\em Arch. Ration. Mech. Anal.}, 215(2):497--529, 2015.

\bibitem[BJS25]{BJS2022}
Didier Bresch, Pierre-Emmanuel Jabin, and Juan Soler.
\newblock A new approach to the mean-field limit of {V}lasov-{F}okker-{P}lanck
  equations.
\newblock {\em Anal. PDE}, 18(4):1037--1064, 2025.

\bibitem[BJW19a]{BJW2019edp}
Didier Bresch, Pierre-Emmanuel Jabin, and Zhenfu Wang.
\newblock Modulated free energy and mean field limit.
\newblock {\em S{\'e}minaire Laurent Schwartz--EDP et applications}, pages
  1--22, 2019.

\bibitem[BJW19b]{BJW2019crm}
Didier Bresch, Pierre-Emmanuel Jabin, and Zhenfu Wang.
\newblock On mean-field limits and quantitative estimates with a large class of
  singular kernels: application to the {P}atlak-{K}eller-{S}egel model.
\newblock {\em C. R. Math. Acad. Sci. Paris}, 357(9):708--720, 2019.

\bibitem[BJW23]{BJW2020}
Didier Bresch, Pierre-Emmanuel Jabin, and Zhenfu Wang.
\newblock Mean field limit and quantitative estimates with singular attractive
  kernels.
\newblock {\em Duke Math. J.}, 172(13):2591--2641, 2023.

\bibitem[BLS18]{BLS2018}
Florent Bekerman, Thomas Lebl\'{e}, and Sylvia Serfaty.
\newblock C{LT} for fluctuations of {$\beta$}-ensembles with general potential.
\newblock {\em Electron. J. Probab.}, 23:Paper no. 115, 31, 2018.

\bibitem[BO19]{BO2019}
Robert~J. Berman and Magnus \"{O}nnheim.
\newblock Propagation of chaos for a class of first order models with singular
  mean field interactions.
\newblock {\em SIAM J. Math. Anal.}, 51(1):159--196, 2019.

\bibitem[Cal80]{Calderon1980}
A.-P. Calder\'{o}n.
\newblock Commutators, singular integrals on {L}ipschitz curves and
  applications.
\newblock In {\em Proceedings of the {I}nternational {C}ongress of
  {M}athematicians ({H}elsinki, 1978)}, pages 85--96. Acad. Sci. Fennica,
  Helsinki, 1980.

\bibitem[CCH14]{CCH2014}
Jos\'{e}~Antonio Carrillo, Young-Pil Choi, and Maxime Hauray.
\newblock The derivation of swarming models: mean-field limit and {W}asserstein
  distances.
\newblock In {\em Collective dynamics from bacteria to crowds}, volume 553 of
  {\em CISM Courses and Lect.}, pages 1--46. Springer, Vienna, 2014.

\bibitem[CD22]{CD2021}
Louis-Pierre Chaintron and Antoine Diez.
\newblock Propagation of chaos: a review of models, methods and applications.
  {I}. {M}odels and methods.
\newblock {\em Kinet. Relat. Models}, 15(6):895--1015, 2022.

\bibitem[CdCRS23]{CdCRS2023}
Antonin Chodron~de Courcel, Matthew Rosenzweig, and Sylvia Serfaty.
\newblock Sharp uniform-in-time mean-field convergence for singular periodic
  {Riesz} flows.
\newblock {\em {Ann. Inst. H. Poincar\'{e} Anal. Non Lin\'{e}aire}}, 2023.
\newblock published online first.

\bibitem[CdCRS25]{CdCRS2023a}
Antonin Chodron~de Courcel, Matthew Rosenzweig, and Sylvia Serfaty.
\newblock The attractive log gas: {S}tability, uniqueness, and propagation of
  chaos.
\newblock {\em Commun. Am. Math. Soc.}, 5:695--773, 2025.

\bibitem[CFGW24]{CFGW2024}
Shuzhe Cai, Xuanrui Feng, Yun Gong, and Zhenfu Wang.
\newblock Propagation of chaos for 2d log gas on the whole space.
\newblock {\em arXiv preprint arXiv:2411.14777}, 2024.

\bibitem[CFP12]{CFP2012}
Jos\'{e}~A. Carrillo, Lucas C.~F. Ferreira, and Juliana~C. Precioso.
\newblock A mass-transportation approach to a one dimensional fluid mechanics
  model with nonlocal velocity.
\newblock {\em Adv. Math.}, 231(1):306--327, 2012.

\bibitem[CJ21]{CJ2021}
Young-Pil Choi and In-Jee Jeong.
\newblock Classical solutions for fractional porous medium flow.
\newblock {\em Nonlinear Anal.}, 210:Paper No. 112393, 13, 2021.

\bibitem[CN25]{CN2025}
Alekos Cecchin and Paul Nikolaev.
\newblock Convergence rate for fluctuations of mean field interacting diffusion
  and application to 2d viscous vortex model and coulomb potential.
\newblock {\em arXiv preprint arXiv:2509.01266}, 2025.

\bibitem[CS07]{CS2007}
Luis Caffarelli and Luis Silvestre.
\newblock An extension problem related to the fractional {L}aplacian.
\newblock {\em Comm. Partial Differential Equations}, 32(7-9):1245--1260, 2007.

\bibitem[DGR]{DGR}
Matias Delgadino, Rishabh Gvalani, and Matthew Rosenzweig.
\newblock Entropic commutator estimates.
\newblock Unpublished note.

\bibitem[DJ25]{DJ2025}
Mitia Duerinckx and Pierre-Emmanuel Jabin.
\newblock Correlation estimates for brownian particles with singular
  interactions.
\newblock {\em arXiv preprint arXiv:2510.01507}, 2025.

\bibitem[Due16]{Duerinckx2016}
Mitia Duerinckx.
\newblock Mean-field limits for some {Riesz} interaction gradient flows.
\newblock {\em SIAM Journal on Mathematical Analysis}, 48(3):2269--2300, 2016.

\bibitem[FS09]{FS2009}
Philippe Flajolet and Robert Sedgewick.
\newblock {\em Analytic combinatorics}.
\newblock Cambridge University Press, Cambridge, 2009.

\bibitem[FW23]{FW2023}
Xuanrui Feng and Zhenfu Wang.
\newblock Quantitative propagation of chaos for 2d viscous vortex model on the
  whole space.
\newblock {\em arXiv preprint arXiv:2310.05156}, 2023.

\bibitem[GBM24]{GlBM2021}
Arnaud Guillin, Pierre~Le Bris, and Pierre Monmarch\'{e}.
\newblock Uniform in time propagation of chaos for the 2d vortex model and
  other singular stochastic systems.
\newblock {\em Journal of the European Mathematical Society}, 2024.

\bibitem[GBR{\etalchar{+}}06]{GBRSS2006}
Arthur Gretton, Karsten Borgwardt, Malte Rasch, Bernhard Sch\"{o}lkopf, and
  Alex Smola.
\newblock A kernel method for the two-sample-problem.
\newblock In B.~Sch\"{o}lkopf, J.~Platt, and T.~Hoffman, editors, {\em Advances
  in Neural Information Processing Systems}, volume~19. MIT Press, 2006.

\bibitem[GBR{\etalchar{+}}07]{GBRSS2007}
Arthur Gretton, Karsten~M Borgwardt, Malte Rasch, Bernhard Sch{\"o}lkopf, and
  Alexander~J Smola.
\newblock A kernel approach to comparing distributions.
\newblock In {\em Proceedings of the national conference on artificial
  intelligence}, volume~22, page 1637. Menlo Park, CA; Cambridge, MA; London;
  AAAI Press; MIT Press; 1999, 2007.

\bibitem[GBR{\etalchar{+}}12]{GBRSS2012}
Arthur Gretton, Karsten~M. Borgwardt, Malte~J. Rasch, Bernhard Sch\"{o}lkopf,
  and Alexander Smola.
\newblock A kernel two-sample test.
\newblock {\em J. Mach. Learn. Res.}, 13:723--773, 2012.

\bibitem[GO14]{GO2014}
Loukas Grafakos and Seungly Oh.
\newblock The {K}ato-{P}once inequality.
\newblock {\em Comm. Partial Differential Equations}, 39(6):1128--1157, 2014.

\bibitem[Gol16]{Golse2016ln}
Fran\c{c}ois Golse.
\newblock On the dynamics of large particle systems in the mean field limit.
\newblock In {\em Macroscopic and large scale phenomena: coarse graining, mean
  field limits and ergodicity}, volume~3 of {\em Lect. Notes Appl. Math.
  Mech.}, pages 1--144. Springer, 2016.

\bibitem[Gol22]{Golse2022ln}
Fran{\c{c}}ois Golse.
\newblock Mean-field limits in statistical dynamics.
\newblock {\em arXiv preprint arXiv:2201.02005}, 2022.

\bibitem[GP22]{GP2021}
Fran\c{c}ois Golse and Thierry Paul.
\newblock Mean-field and classical limit for the {$N$}-body quantum dynamics
  with {C}oulomb interaction.
\newblock {\em Communications on Pure and Applied Mathematics},
  75(6):1332--1376, 2022.

\bibitem[Gra14]{Grafakos2014m}
Loukas Grafakos.
\newblock {\em Modern Fourier Analysis}.
\newblock Number 250 in Graduate Texts in Mathematics. Springer, third edition,
  2014.

\bibitem[Hau09]{Hauray2009}
Maxime Hauray.
\newblock Wasserstein distances for vortices approximation of {E}uler-type
  equations.
\newblock {\em Math. Models Methods Appl. Sci.}, 19(8):1357--1384, 2009.

\bibitem[HC23]{hC2023}
Elias Hess-Childs.
\newblock Large deviation principles for singular {R}iesz-type diffusive flows.
\newblock {\em arXiv preprint arXiv:2312.02904}, 2023.

\bibitem[HCR25]{hCR2023}
Elias Hess-Childs and Keefer Rowan.
\newblock Higher-order propagation of chaos in {L}2 for interacting diffusions.
\newblock {\em Probab. Math. Phys.}, 6(2):581--646, 2025.

\bibitem[HCRSa]{hCRScx}
Elias Hess-Childs, Matthew Rosenzweig, and Sylvia Serfaty.
\newblock Another look at regularity in transport commutator estimates.
\newblock In preparation.

\bibitem[HCRSb]{HcRS2024oq}
Elias Hess-Childs, Matthew Rosenzweig, and Sylvia Serfaty.
\newblock Optimal quantization for {R}iesz {M}axmimum {M}ean {D}iscrepancies.
\newblock In preparation.

\bibitem[HKI21]{HkI2021}
Daniel Han-Kwan and Mikaela Iacobelli.
\newblock From {N}ewton's second law to {E}uler's equations of perfect fluids.
\newblock {\em Proc. Amer. Math. Soc.}, 149(7):3045--3061, 2021.

\bibitem[HM14]{HM2014}
Maxime Hauray and St\'{e}phane Mischler.
\newblock On {K}ac's chaos and related problems.
\newblock {\em J. Funct. Anal.}, 266(10):6055--6157, 2014.

\bibitem[HRS]{HRSclt}
Jiaoyang Huang, Matthew Rosenzweig, and Sylvia Serfaty.
\newblock Fluctuations around the mean field limit for noisy singular {R}iesz
  flows.
\newblock In preparation.

\bibitem[HSSS17]{HSSS2017}
Douglas~P. Hardin, Edward~B. Saff, Brian~Z. Simanek, and Yujian Su.
\newblock Next order energy asymptotics for {R}iesz potentials on flat tori.
\newblock {\em Int. Math. Res. Not. IMRN}, (12):3529--3556, 2017.

\bibitem[HSST21]{HSST2020}
Douglas Hardin, Edward~B. Saff, Ruiwen Shu, and Eitan Tadmor.
\newblock Dynamics of particles on a curve with pairwise hyper-singular
  repulsion.
\newblock {\em Discrete Contin. Dyn. Syst.}, 41(12):5509--5536, 2021.

\bibitem[HWAH23]{HWAH2023}
Johannes Hertrich, Christian Wald, Fabian Altekr{\"u}ger, and Paul Hagemann.
\newblock Generative sliced {MMD} flows with {R}iesz kernels.
\newblock {\em arXiv preprint arXiv:2305.11463}, 2023.

\bibitem[Jab14]{Jabin2014}
Pierre-Emmanuel Jabin.
\newblock A review of the mean field limits for {V}lasov equations.
\newblock {\em Kinet. Relat. Models}, 7(4):661--711, 2014.

\bibitem[JW17]{JW2017_survey}
Pierre-Emmanuel Jabin and Zhenfu Wang.
\newblock {Mean field limit for stochastic particle systems}.
\newblock In {\em Act. Part. {V}ol. 1. {A}dvances theory, Model. Appl.}, Model.
  Simul. Sci. Eng. Technol., pages 379--402. Birkh{\"{a}}user/Springer, Cham,
  2017.

\bibitem[JW18]{JW2018}
Pierre-Emmanuel Jabin and Zhenfu Wang.
\newblock Quantitative estimates of propagation of chaos for stochastic systems
  with {$W^{-1,\infty}$} kernels.
\newblock {\em Invent. Math.}, 214(1):523--591, 2018.

\bibitem[KNSS22]{KNSS2022}
Soheil Kolouri, Kimia Nadjahi, Shahin Shahrampour, and Umut {Sim{s}ekli}.
\newblock Generalized sliced probability metrics.
\newblock In {\em ICASSP 2022 - 2022 IEEE International Conference on
  Acoustics, Speech and Signal Processing (ICASSP)}, pages 4513--4517, 2022.

\bibitem[KP88]{KP1988}
Tosio Kato and Gustavo Ponce.
\newblock Commutator estimates and the {E}uler and {N}avier-{S}tokes equations.
\newblock {\em Comm. Pure Appl. Math.}, 41(7):891--907, 1988.

\bibitem[KPV93]{KPV1993}
Carlos~E. Kenig, Gustavo Ponce, and Luis Vega.
\newblock Well-posedness and scattering results for the generalized
  {K}orteweg-de {V}ries equation via the contraction principle.
\newblock {\em Comm. Pure Appl. Math.}, 46(4):527--620, 1993.

\bibitem[Lac23]{Lacker2023}
Daniel Lacker.
\newblock Hierarchies, entropy, and quantitative propagation of chaos for mean
  field diffusions.
\newblock {\em Probab. Math. Phys.}, 4(2):377--432, 2023.

\bibitem[Li19]{Li2019}
Dong Li.
\newblock On {K}ato-{P}once and fractional {L}eibniz.
\newblock {\em Rev. Mat. Iberoam.}, 35(1):23--100, 2019.

\bibitem[LLF23]{LlF2023}
Daniel Lacker and Luc Le~Flem.
\newblock Sharp uniform-in-time propagation of chaos.
\newblock {\em Probability Theory and Related Fields}, 2023.

\bibitem[LLN20]{LLN2020}
Tau~Shean Lim, Yulong Lu, and James~H. Nolen.
\newblock Quantitative propagation of chaos in a bimolecular chemical
  reaction-diffusion model.
\newblock {\em SIAM J. Math. Anal.}, 52(2):2098--2133, 2020.

\bibitem[LS18]{LS2018}
Thomas Lebl\'{e} and Sylvia Serfaty.
\newblock Fluctuations of two dimensional {C}oulomb gases.
\newblock {\em Geom. Funct. Anal.}, 28(2):443--508, 2018.

\bibitem[LS20]{LS2020}
Enno Lenzmann and Armin Schikorra.
\newblock Sharp commutator estimates via harmonic extensions.
\newblock {\em Nonlinear Anal.}, 193:111375, 37, 2020.

\bibitem[M\"97]{Muller1997}
Alfred M\"uller.
\newblock Integral probability metrics and their generating classes of
  functions.
\newblock {\em Adv. in Appl. Probab.}, 29(2):429--443, 1997.

\bibitem[M\'24]{Menard2022}
Matthieu M\'{e}nard.
\newblock Mean-field limit derivation of a monokinetic spray model with
  gyroscopic effects.
\newblock {\em SIAM J. Math. Anal.}, 56(1):1068--1113, 2024.

\bibitem[MD24]{MD2024}
Thibault Modeste and Cl\'ement Dombry.
\newblock Characterization of translation invariant {MMD} on {$\Bbb R^d$} and
  connections with {W}asserstein distances.
\newblock {\em J. Mach. Learn. Res.}, 25:Paper No. [237], 39, 2024.

\bibitem[MN24]{MN2024}
Roberta Musina and Alexander~I. Nazarov.
\newblock Fractional operators as traces of operator-valued curves.
\newblock {\em J. Funct. Anal.}, 287(2):Paper No. 110443, 33, 2024.

\bibitem[NRS22]{NRS2021}
Quoc-Hung Nguyen, Matthew Rosenzweig, and Sylvia Serfaty.
\newblock Mean-field limits of {R}iesz-type singular flows.
\newblock {\em Ars Inven. Anal.}, pages Paper No. 4, 45, 2022.

\bibitem[PCJ25]{bPCJ2025}
Immanuel~Ben Porat, Jos{\'e}~A Carrillo, and Pierre-Emmanuel Jabin.
\newblock Singular flows with time-varying weights.
\newblock {\em arXiv preprint arXiv:2503.02276}, 2025.

\bibitem[Por23]{Porat2022}
Immanuel~Ben Porat.
\newblock Derivation of {E}uler's equations of perfect fluids from von
  {N}eumann's equation with magnetic field.
\newblock {\em J. Stat. Phys.}, 190(7):Paper No. 121, 44, 2023.

\bibitem[PS]{PS2024}
Luke Peilen and Sylvia Serfaty.
\newblock Gaussian fluctuations and free energy expansion for {R}iesz gases.
\newblock In preparation.

\bibitem[PS17]{PS2017}
Mircea Petrache and Sylvia Serfaty.
\newblock Next order asymptotics and renormalized energy for {R}iesz
  interactions.
\newblock {\em J. Inst. Math. Jussieu}, 16(3):501--569, 2017.

\bibitem[Ros]{Rosenzweigcum}
Matthew Rosenzweig.
\newblock Cumulants of {R}iesz gases.
\newblock Unpublished note.

\bibitem[Ros20]{Rosenzweig2020spv}
Matthew Rosenzweig.
\newblock The mean-field limit of stochastic point vortex systems with
  multiplicative noise.
\newblock {\em arXiv preprint arXiv:2011.12180}, 2020.
\newblock accepted by Comm. Pure Appl. Math.

\bibitem[Ros21]{Rosenzweig2021qe}
Matthew Rosenzweig.
\newblock From quantum many-body systems to ideal fluids.
\newblock {\em arXiv preprint arXiv:2110.04195}, 2021.

\bibitem[Ros22a]{Rosenzweig2022a}
Matthew Rosenzweig.
\newblock The mean-field approximation for higher-dimensional {Coulomb} flows
  in the scaling-critical {$L^\infty$} space.
\newblock {\em Nonlinearity}, 35(6):2722--2766, may 2022.

\bibitem[Ros22b]{Rosenzweig2022}
Matthew Rosenzweig.
\newblock Mean-{F}ield {C}onvergence of {P}oint {V}ortices to the
  {I}ncompressible {E}uler {E}quation with {V}orticity in {$L^\infty$}.
\newblock {\em Arch. Ration. Mech. Anal.}, 243(3):1361--1431, 2022.

\bibitem[Ros23]{Rosenzweig2021ne}
Matthew Rosenzweig.
\newblock On the rigorous derivation of the incompressible {E}uler equation
  from {N}ewton's second law.
\newblock {\em Lett. Math. Phys.}, 113(1):Paper No. 13, 32, 2023.

\bibitem[RSa]{RSstein}
Matthew Rosenzweig and Sylvia Serfaty.
\newblock Commutator estimates, {S}tein's method, and the transport approach to
  fluctuations of {R}iesz gases.
\newblock In preparation.

\bibitem[RSb]{RS2022}
Matthew Rosenzweig and Sylvia Serfaty.
\newblock Sharp commutator estimates of all order for {C}oulomb and {R}iesz
  modulated energies.
\newblock {\em Communications on Pure and Applied Mathematics}, n/a(n/a).

\bibitem[RS16]{RS2016}
Nicolas Rougerie and Sylvia Serfaty.
\newblock Higher-dimensional {C}oulomb gases and renormalized energy
  functionals.
\newblock {\em Comm. Pure Appl. Math.}, 69(3):519--605, 2016.

\bibitem[RS23]{RS2021}
Matthew Rosenzweig and Sylvia Serfaty.
\newblock Global-in-time mean-field convergence for singular {R}iesz-type
  diffusive flows.
\newblock {\em Ann. Appl. Probab.}, 33(2):754--798, 2023.

\bibitem[RS24]{RS2024ss}
Matthew Rosenzweig and Sylvia Serfaty.
\newblock Relative entropy and modulated free energy without confinement via
  self-similar transformation.
\newblock {\em arXiv preprint arXiv:2402.13977}, 2024.

\bibitem[RS25a]{RSlake}
Matthew Rosenzweig and Sylvia Serfaty.
\newblock The {l}ake equation as a supercritical mean-field limit.
\newblock {\em J. \'Ec. polytech. Math.}, 12:1019--1068, 2025.

\bibitem[RS25b]{RS2023lsi}
Matthew Rosenzweig and Sylvia Serfaty.
\newblock Modulated logarithmic {S}obolev inequalities and generation of chaos.
\newblock {\em Ann. Fac. Sci. Toulouse Math. (6)}, 34(1):107--134, 2025.

\bibitem[RSW]{RSW2025}
Matthew Rosenzweig, Dejan Slep{\v{c}ev}, and Lihan Wang.
\newblock Wasserstein gradient flow of {M}aximum {M}ean {D}iscrepancy with
  energy kernels.
\newblock In preparation.

\bibitem[Rub24]{Rubin2024}
Boris Rubin.
\newblock {\em Fractional integrals, potentials, and {R}adon transforms}.
\newblock Chapman and Hall/CRC, 2024.

\bibitem[Ser20]{Serfaty2020}
Sylvia Serfaty.
\newblock Mean field limit for {Coulomb-type} flows.
\newblock {\em Duke Math. J.}, 169(15):2887--2935, 10 2020.
\newblock Appendix with Mitia Duerinckx.

\bibitem[Ser23]{Serfaty2023}
Sylvia Serfaty.
\newblock Gaussian fluctuations and free energy expansion for {C}oulomb gases
  at any temperature.
\newblock {\em Ann. Inst. Henri Poincar\'{e} Probab. Stat.}, 59(2):1074--1142,
  2023.

\bibitem[Ser24]{SerfatyLN}
Sylvia Serfaty.
\newblock Lectures on {C}oulomb and {R}iesz gases.
\newblock {\em arXiv preprint arXiv:2407.21194}, 2024.

\bibitem[SS15a]{SS2015log}
Etienne Sandier and Sylvia Serfaty.
\newblock 1{D} log gases and the renormalized energy: crystallization at
  vanishing temperature.
\newblock {\em Probab. Theory Related Fields}, 162(3-4):795--846, 2015.

\bibitem[SS15b]{SS2015}
Etienne Sandier and Sylvia Serfaty.
\newblock 2{D} {C}oulomb gases and the renormalized energy.
\newblock {\em Ann. Probab.}, 43(4):2026--2083, 2015.

\bibitem[ST10]{ST2010ext}
Pablo~Ra\'ul Stinga and Jos\'e{}~Luis Torrea.
\newblock Extension problem and {H}arnack's inequality for some fractional
  operators.
\newblock {\em Comm. Partial Differential Equations}, 35(11):2092--2122, 2010.

\bibitem[Tay91]{Taylor1991}
Michael~E. Taylor.
\newblock {\em Pseudodifferential operators and nonlinear {PDE}}, volume 100 of
  {\em Progress in Mathematics}.
\newblock Birkh\"auser Boston, Inc., Boston, MA, 1991.

\bibitem[Tay03]{Taylor2003}
Michael Taylor.
\newblock Commutator estimates.
\newblock {\em Proc. Amer. Math. Soc.}, 131(5):1501--1507, 2003.

\bibitem[Wan26]{Wang2024sharp}
Songbo Wang.
\newblock Sharp local propagation of chaos for mean field particles with
  {$W^{-1,\infty}$} kernels.
\newblock {\em J. Funct. Anal.}, 290(3):Paper No. 111240, 2026.

\end{thebibliography}

\end{document}